\documentclass[a4paper]{article}

\usepackage[margin=25mm]{geometry}
\usepackage{amsmath,amsthm,amssymb,mathrsfs}
\usepackage{enumitem} 
\setlist[description]
{noitemsep,labelindent=4ex}
\setlist{noitemsep}
\usepackage{graphicx}
\usepackage{mathtools}
\usepackage[export]{adjustbox}
\usepackage{todonotes}

\usepackage{setspace}

\graphicspath{{./figures/}}
\newdimen\BBoxXMin
\newdimen\BBoxYMin
\newdimen\BBoxXMax
\newdimen\BBoxYMax
\colorlet{bgcolor}{gray!20}
\colorlet{defectcolor}{gray!50}

\definecolor{myPink}{HTML}{FDF5FF}
\definecolor{myRed}{HTML}{eb0707}
\definecolor{myOrange}{HTML}{eb7907}
\definecolor{myGrey}{HTML}{D1D5DB}
\colorlet{myPurple}{violet!70!magenta}

\definecolor{T-region}{HTML}{E88444} 
\definecolor{G-region}{HTML}{B3CCE6} 
\definecolor{parabolic-defect}{HTML}{77CCAA} 
\definecolor{weyl-defect}{HTML}{888811} 

\tikzset{
    thick,
    font=\small\sffamily,
    line_defect/.style={draw=defectcolor, line width=4pt}, 
    skein/.style={thick},
    3D_line/.style={thick},
    2D_line/.style={thick},
    1D_line/.style={thick},
    crossing_over/.style={thick,preaction={draw=bgcolor, -, line width=4pt}},
    every node/.style={scale=.9}
}

\newcommand{\FullWeylCatScaled}[2]{%
  \mathord{%
    \tikz[
      baseline=-0.5ex,
      scale=#1,
      every node/.style={transform shape}
    ]{%
      \begin{scope}
        \begin{scope}
          \path[clip] (0,0) circle[radius=6mm];
          \fill[T-region] (0,0) circle[radius=6mm];
          \filldraw[fill=G-region,draw=parabolic-defect,very thick]
            (-6mm,3mm) to[out=0,in=0,looseness=1.6]
            coordinate[pos=.5] (left-junction) (-6mm,-3mm);
          \filldraw[fill=G-region,draw=parabolic-defect,very thick]
            (6mm,3mm) to[out=180,in=180,looseness=1.6]
            coordinate[pos=.5] (right-junction) (6mm,-3mm);

          \draw[very thick,draw=weyl-defect]
            (left-junction) --
            node[pos=.5,above,scale=.7] {$ #2\mathcal{T}_{w_0} $}
            (right-junction);

          \fill (left-junction) circle[radius=1.5pt]
                (right-junction) circle[radius=1.5pt];
        \end{scope}

         \draw[thick] (0,0) circle[radius=6mm];
         \pgfresetboundingbox
      \path[use as bounding box] (-19.7pt,-19.7pt) rectangle (19.7pt,19.7pt);
      \end{scope}
    }%
  }%
}

\newcommand{\FullWeylCat}{%
  \mathchoice
    {\FullWeylCatScaled{1}{\textstyle}}
    {\FullWeylCatScaled{1}{\textstyle}}
    {\FullWeylCatScaled{0.7}{\scriptstyle}}
    {\FullWeylCatScaled{0.55}{\scriptscriptstyle}}%
}

\newcommand{\LeftWeylCatScaled}[2]{%
  \mathord{%
    \tikz[
      baseline=-0.5ex,
      scale=#1,
      every node/.style={transform shape}
    ]{%
      \begin{scope}
        \fill[G-region]
          (120:6mm) coordinate (-NW)
          arc[start angle=120, delta angle=120, radius=6mm]
          to (0,0) coordinate (-center)
          to cycle;

        \fill[T-region]
          (240:6mm) coordinate (-SW)
          arc[start angle=240, delta angle=240, radius=6mm]
          to (0,0)
          to cycle;

        \path (0:6mm) coordinate (-E);

        \draw[very thick,parabolic-defect]
          (-center) to coordinate[pos=.5] (-mid-NW) coordinate[pos=1] (-label-NW) (-NW)
          (-center) to coordinate[pos=.5] (-mid-SW) coordinate[pos=1] (-label-SW) (-SW);

        \draw[very thick,weyl-defect]
          (-center) to
          coordinate[pos=.5] (-mid-E)
          coordinate[pos=.8] (-late-E)
          coordinate[black,pos=1] (-label-E)
          (-E);

        \fill (-center) circle[radius=2pt];

        \draw[thick] (-center) circle[radius=6mm];
      \end{scope}
    }%
  }%
}
\newcommand{\LeftWeylCat}{%
  \mathchoice
    {\LeftWeylCatScaled{1}{\textstyle}}
    {\LeftWeylCatScaled{1}{\textstyle}}
    {\LeftWeylCatScaled{0.7}{\scriptstyle}}
    {\LeftWeylCatScaled{0.55}{\scriptscriptstyle}}%
}

\newcommand{\LeftWeylAlgGTTScaled}[2]{%
  \mathord{%
    \tikz[
      baseline=-0.5ex,
      scale=#1,
      every node/.style={transform shape}
    ]{%
      \begin{scope}
        \fill[G-region]
          (120:6mm) coordinate (-NW)
          arc[start angle=120, delta angle=120, radius=6mm]
          to (0,0) coordinate (-center)
          to cycle;

        \fill[T-region]
          (240:6mm) coordinate (-SW)
          arc[start angle=240, delta angle=240, radius=6mm]
          to (0,0)
          to cycle;

        \path (0:6mm) coordinate (-E);

        \draw[very thick,parabolic-defect]
          (-center) to coordinate[pos=.5] (-mid-NW) coordinate[pos=1] (-label-NW) (-NW)
          (-center) to coordinate[pos=.5] (-mid-SW) coordinate[pos=1] (-label-SW) (-SW);

        \draw[very thick,weyl-defect]
          (-center) to
          coordinate[pos=.5] (-mid-E)
          coordinate[pos=.8] (-late-E)
          coordinate[black,pos=1] (-label-E)
          (-E);

        \fill (-center) circle[radius=2pt];

        \draw[thick]
    ( 90-15:6mm-.5mm) -- ( 90-15:6mm+.5mm)
    ( 90+15:6mm-.5mm) -- ( 90+15:6mm+.5mm)
    (270-15:6mm-.5mm) -- (270-15:6mm+.5mm)
    (270+15:6mm-.5mm) -- (270+15:6mm+.5mm)
    (180-15:6mm-.5mm) -- (180-15:6mm+.5mm)
    (180+15:6mm-.5mm) -- (180+15:6mm+.5mm)
    ;
  \draw[thick]
    (90-15:6mm) arc[start angle=90-15,end angle=-90+15,radius=6mm]
    (90+15:6mm) arc[start angle=90+15,end angle=180-15,radius=6mm]
    (180+15:6mm) arc[start angle=180+15,end angle=270-15,radius=6mm]
    ;
      \end{scope}
    }%
  }%
}
\newcommand{\LeftWeylAlgGTT}{%
  \mathchoice
    {\LeftWeylAlgGTTScaled{1}{\textstyle}}
    {\LeftWeylAlgGTTScaled{1}{\textstyle}}
    {\LeftWeylAlgGTTScaled{0.7}{\scriptstyle}}
    {\LeftWeylAlgGTTScaled{0.55}{\scriptscriptstyle}}%
}
\newcommand{\LeftWeylAlgTTScaled}[2]{%
  \mathord{%
    \tikz[
      baseline=-0.5ex,
      scale=#1,
      every node/.style={transform shape}
    ]{%
      \begin{scope}
        \fill[G-region]
          (120:6mm) coordinate (-NW)
          arc[start angle=120, delta angle=120, radius=6mm]
          to (0,0) coordinate (-center)
          to cycle;

        \fill[T-region]
          (240:6mm) coordinate (-SW)
          arc[start angle=240, delta angle=240, radius=6mm]
          to (0,0)
          to cycle;

        \path (0:6mm) coordinate (-E);

        \draw[very thick,parabolic-defect]
          (-center) to coordinate[pos=.5] (-mid-NW) coordinate[pos=1] (-label-NW) (-NW)
          (-center) to coordinate[pos=.5] (-mid-SW) coordinate[pos=1] (-label-SW) (-SW);

        \draw[very thick,weyl-defect]
          (-center) to
          coordinate[pos=.5] (-mid-E)
          coordinate[pos=.8] (-late-E)
          coordinate[black,pos=1] (-label-E)
          (-E);

        \fill (-center) circle[radius=2pt];

        \draw[thick]
    ( 90-15:6mm-.5mm) -- ( 90-15:6mm+.5mm)
    ( 90+15:6mm-.5mm) -- ( 90+15:6mm+.5mm)
    (270-15:6mm-.5mm) -- (270-15:6mm+.5mm)
    (270+15:6mm-.5mm) -- (270+15:6mm+.5mm)
    ;
  \draw[thick]
    (90-15:6mm) arc[start angle=90-15,end angle=-90+15,radius=6mm]
    (90+15:6mm) arc[start angle=90+15,end angle=270-15,radius=6mm]
    ;
      \end{scope}
    }%
  }%
}
\newcommand{\LeftWeylAlgTT}{%
  \mathchoice
    {\LeftWeylAlgTTScaled{1}{\textstyle}}
    {\LeftWeylAlgTTScaled{1}{\textstyle}}
    {\LeftWeylAlgTTScaled{0.7}{\scriptstyle}}
    {\LeftWeylAlgTTScaled{0.55}{\scriptscriptstyle}}%
}
\newcommand{\RightWeylCatScaled}[2]{%
  \mathord{%
    \tikz[
      baseline=-0.5ex,
      scale=#1,
      every node/.style={transform shape}
    ]{%
      \begin{scope}[xscale=-1]
        \fill[G-region]
          (120:6mm) coordinate (-NW)
          arc[start angle=120, delta angle=120, radius=6mm]
          to (0,0) coordinate (-center)
          to cycle;

        \fill[T-region]
          (240:6mm) coordinate (-SW)
          arc[start angle=240, delta angle=240, radius=6mm]
          to (0,0)
          to cycle;

        \path (0:6mm) coordinate (-E);

        \draw[very thick,parabolic-defect]
          (-center) to coordinate[pos=.5] (-mid-NW) coordinate[pos=1] (-label-NW) (-NW)
          (-center) to coordinate[pos=.5] (-mid-SW) coordinate[pos=1] (-label-SW) (-SW);

        \draw[very thick,weyl-defect]
          (-center) to
          coordinate[pos=.5] (-mid-E)
          coordinate[pos=.8] (-late-E)
          coordinate[black,pos=1] (-label-E)
          (-E);

        \fill (-center) circle[radius=2pt];

        \draw[thick] (-center) circle[radius=6mm];
      \end{scope}
    }%
  }%
}
\newcommand{\RightWeylCat}{%
  \mathchoice
    {\RightWeylCatScaled{1}{\textstyle}}
    {\RightWeylCatScaled{1}{\textstyle}}
    {\RightWeylCatScaled{0.7}{\scriptstyle}}
    {\RightWeylCatScaled{0.55}{\scriptscriptstyle}}%
}

\newcommand{\RightWeylAlgTTScaled}[2]{%
  \mathord{%
    \tikz[
      baseline=-0.5ex,
      scale=#1,
      every node/.style={transform shape}
    ]{%
      \begin{scope}[xscale=-1]
        \fill[G-region]
          (120:6mm) coordinate (-NW)
          arc[start angle=120, delta angle=120, radius=6mm]
          to (0,0) coordinate (-center)
          to cycle;

        \fill[T-region]
          (240:6mm) coordinate (-SW)
          arc[start angle=240, delta angle=240, radius=6mm]
          to (0,0)
          to cycle;

        \path (0:6mm) coordinate (-E);

        \draw[very thick,parabolic-defect]
          (-center) to coordinate[pos=.5] (-mid-NW) coordinate[pos=1] (-label-NW) (-NW)
          (-center) to coordinate[pos=.5] (-mid-SW) coordinate[pos=1] (-label-SW) (-SW);

        \draw[very thick,weyl-defect]
          (-center) to
          coordinate[pos=.5] (-mid-E)
          coordinate[pos=.8] (-late-E)
          coordinate[black,pos=1] (-label-E)
          (-E);

        \fill (-center) circle[radius=2pt];

        \draw[thick]
    ( 90-15:6mm-.5mm) -- ( 90-15:6mm+.5mm)
    ( 90+15:6mm-.5mm) -- ( 90+15:6mm+.5mm)
    (270-15:6mm-.5mm) -- (270-15:6mm+.5mm)
    (270+15:6mm-.5mm) -- (270+15:6mm+.5mm)
    ;
  \draw[thick]
    (90-15:6mm) arc[start angle=90-15,end angle=-90+15,radius=6mm]
    (90+15:6mm) arc[start angle=90+15,end angle=270-15,radius=6mm]
    ;
      \end{scope}
    }%
  }%
}
\newcommand{\RightWeylAlgTT}{%
  \mathchoice
    {\RightWeylAlgTTScaled{1}{\textstyle}}
    {\RightWeylAlgTTScaled{1}{\textstyle}}
    {\RightWeylAlgTTScaled{0.7}{\scriptstyle}}
    {\RightWeylAlgTTScaled{0.55}{\scriptscriptstyle}}%
}
\newcommand{\WeylCatScaled}[2]{%
  \mathord{%
    \tikz[
      baseline=-0.5ex,
      scale=#1,
      every node/.style={transform shape}
    ]{%
\begin{scope}
    \path[clip] (0,0) circle[radius=6mm];
    \fill[T-region] (0,0) circle[radius=6mm];
  \draw[very thick,draw=weyl-defect]
    (-6mm,0) -- (6mm,0);
\end{scope}

  \draw[thick] (0,0) circle[radius=6mm];
  \pgfresetboundingbox
      \path[use as bounding box] (-19.7pt,-19.7pt) rectangle (19.7pt,19.7pt);
  
    }%
  }%
}
\newcommand{\WeylCat}{%
  \mathchoice
    {\WeylCatScaled{1}{\textstyle}}
    {\WeylCatScaled{1}{\textstyle}}
    {\WeylCatScaled{0.7}{\scriptstyle}}
    {\WeylCatScaled{0.55}{\scriptscriptstyle}}%
}
\newcommand{\WeylAlgScaled}[2]{%
  \mathord{%
    \tikz[
      baseline=-0.5ex,
      scale=#1,
      every node/.style={transform shape}
    ]{%
\begin{scope}
    \path[clip] (0,0) circle[radius=6mm];
    \fill[T-region] (0,0) circle[radius=6mm];
    
  \end{scope}
  \draw[very thick,draw=weyl-defect]
    (-6mm,0) -- (6mm,0);

  \draw[thick]
    ( 90-15:6mm-.5mm) -- ( 90-15:6mm+.5mm)
    ( 90+15:6mm-.5mm) -- ( 90+15:6mm+.5mm)
    (270-15:6mm-.5mm) -- (270-15:6mm+.5mm)
    (270+15:6mm-.5mm) -- (270+15:6mm+.5mm)
    ;
  \draw[thick]
    (90-15:6mm) arc[start angle=90-15,end angle=-90+15,radius=6mm]
    (90+15:6mm) arc[start angle=90+15,end angle=270-15,radius=6mm]
    ;
    }%
  }%
}
\newcommand{\WeylAlg}{%
  \mathchoice
    {\WeylAlgScaled{1}{\textstyle}}
    {\WeylAlgScaled{1}{\textstyle}}
    {\WeylAlgScaled{0.7}{\scriptstyle}}
    {\WeylAlgScaled{0.55}{\scriptscriptstyle}}%
}
\newcommand{\FullWeylAlgScaled}[2]{%
  \mathord{%
    \tikz[
      baseline=-0.5ex,
      scale=#1,
      every node/.style={transform shape}
    ]{%
      \begin{scope}
        \begin{scope}
          \path[clip] (0,0) circle[radius=6mm];
          \fill[T-region] (0,0) circle[radius=6mm];
          \filldraw[fill=G-region,draw=parabolic-defect,very thick]
            (-6mm,3mm) to[out=0,in=0,looseness=1.6]
            coordinate[pos=.5] (left-junction) (-6mm,-3mm);
          \filldraw[fill=G-region,draw=parabolic-defect,very thick]
            (6mm,3mm) to[out=180,in=180,looseness=1.6]
            coordinate[pos=.5] (right-junction) (6mm,-3mm);

          \draw[very thick,draw=weyl-defect]
            (left-junction) --
            (right-junction);

          \fill (left-junction) circle[radius=1.5pt]
                (right-junction) circle[radius=1.5pt];
        \end{scope}

        \draw[thick]
          ( 90-15:6mm-.5mm) -- ( 90-15:6mm+.5mm)
          ( 90+15:6mm-.5mm) -- ( 90+15:6mm+.5mm)
          (270-15:6mm-.5mm) -- (270-15:6mm+.5mm)
          (270+15:6mm-.5mm) -- (270+15:6mm+.5mm);

        \draw[thick]
          (90-15:6mm) arc[start angle=90-15,end angle=-90+15,radius=6mm]
          (90+15:6mm) arc[start angle=90+15,end angle=270-15,radius=6mm];
      \end{scope}
    }%
  }%
}

\newcommand{\FullWeylAlg}{%
  \mathchoice
    {\FullWeylAlgScaled{1}{\textstyle}}
    {\FullWeylAlgScaled{1}{\textstyle}}
    {\FullWeylAlgScaled{0.7}{\scriptstyle}}
    {\FullWeylAlgScaled{0.55}{\scriptscriptstyle}}%
}

\newcommand{\DigonScaled}[2]{%
  \mathord{%
    \tikz[
      baseline=-0.5ex,
      scale=#1,
      every node/.style={transform shape}
    ]{%
       \begin{scope} 
    \path[clip] (0,0) circle[radius=6mm];
    \fill[G-region] (0,0) circle[radius=6mm];
    \fill[T-region] (-6mm,3mm) -- ++(16mm,0) -- ++ (0,3mm) -- ++ (-16mm,0);
    \fill[T-region] (-6mm,-3mm) -- ++(16mm,0) -- ++ (0,-3mm) -- ++ (-20mm,0);

    \draw[very thick,parabolic-defect] (-6mm,3mm) --node[black,pos=.3,below,scale=.7] {$\widetilde{\mathcal{D}_B}$} (6mm,3mm) (-6mm,-3mm) -- node[black,pos=.7,above,scale=.7] {$\widetilde{\mathcal{D}_B}$} (6mm,-3mm);
  \end{scope}
   \draw[thick] (0,0) circle[radius=6mm];
   \pgfresetboundingbox
      \path[use as bounding box] (-19.7pt,-19.7pt) rectangle (19.7pt,19.7pt);
    }%
  }%
}

\newcommand{\Digon}{%
  \mathchoice
    {\DigonScaled{1}{\textstyle}}
    {\DigonScaled{1}{\textstyle}}
    {\DigonScaled{0.7}{\scriptstyle}}
    {\DigonScaled{0.55}{\scriptscriptstyle}}%
}

\newcommand{\DigonAlgScaled}[2]{%
  \mathord{%
    \tikz[
      baseline=-0.5ex,
      scale=#1,
      every node/.style={transform shape}
    ]{%
       \begin{scope} 
    \path[clip] (0,0) circle[radius=6mm];
    \fill[G-region] (0,0) circle[radius=6mm];
    \fill[T-region] (-6mm,3mm) -- ++(16mm,0) -- ++ (0,3mm) -- ++ (-16mm,0);
    \fill[T-region] (-6mm,-3mm) -- ++(16mm,0) -- ++ (0,-3mm) -- ++ (-20mm,0);

    \draw[very thick,parabolic-defect] (-6mm,3mm) --node[black,pos=.3,below,scale=.7] {} (6mm,3mm) (-6mm,-3mm) -- node[black,pos=.7,above,scale=.7] {} (6mm,-3mm);
  \end{scope}
   \draw[thick]
          ( 90-15:6mm-.5mm) -- ( 90-15:6mm+.5mm)
          ( 90+15:6mm-.5mm) -- ( 90+15:6mm+.5mm)
          (270-15:6mm-.5mm) -- (270-15:6mm+.5mm)
          (270+15:6mm-.5mm) -- (270+15:6mm+.5mm);

        \draw[thick]
          (90-15:6mm) arc[start angle=90-15,end angle=-90+15,radius=6mm]
          (90+15:6mm) arc[start angle=90+15,end angle=270-15,radius=6mm];
   \pgfresetboundingbox
      \path[use as bounding box] (-19.7pt,-19.7pt) rectangle (19.7pt,19.7pt);
    }%
  }%
}

\newcommand{\DigonAlg}{%
  \mathchoice
    {\DigonAlgScaled{1}{\textstyle}}
    {\DigonAlgScaled{1}{\textstyle}}
    {\DigonAlgScaled{0.7}{\scriptstyle}}
    {\DigonAlgScaled{0.55}{\scriptscriptstyle}}%
}

\usetikzlibrary{decorations.markings}
\tikzset{
  strand/.style={semithick}, 
  dstrand/.style={dotted}, 
  rdstrand/.style={semithick,dash pattern=on 0.7 off 1pt, color=red},  
  oriented/.style={postaction={decorate},
                   decoration={markings, mark=at position 0.85 with {\arrow{>}}}} 
}

\newcommand{\singlestrand}[1]{
  \vcenter{\hbox{\begin{tikzpicture}[scale=0.6]
    \draw[#1,oriented] (0,-0.6) -- (0,0.6);
  \end{tikzpicture}}}
}

\newcommand{\loopstrandJ}[1]{
  \vcenter{\hbox{\begin{tikzpicture}[scale=0.6,,xscale=-1]
    \coordinate (C) at (-0.1,0);
    \coordinate (X) at (-0.1,-0.6);
    \coordinate (Y) at (-0.1,0.6);
    \coordinate (A0) at (0.1,0.1);
    \coordinate (A1) at (0.1,-0.3);
    \coordinate (B) at (0.3,-0.1);
    \draw[#1] (A1) [out=180,in=-80] to (C);
    \draw[#1,oriented] (C) [out=100,in=-90] to (Y);
    \draw[white,line width=4.5pt] (X) [out=90,in=-100] to (C);
    \draw[white,line width=4.5pt] (C) [out=80,in=90] to (B);
    \draw[white,line width=4.5pt] (B) [out=-90,in=0] to (A1);
    \draw[#1] (X) [out=90,in=-100] to (C);
    \draw[#1] (C) [out=80,in=180] to (A0);
    \draw[#1] (A0) [out=0,in=90] to (B);
    \draw[#1] (B) [out=-90,in=0] to (A1);
  \end{tikzpicture}}}
}

\newcommand{\circlestrand}[2]{
  \vcenter{\hbox{\begin{tikzpicture}[scale=0.6]
    \ifthenelse{\equal{#2}{ccw}}{
      \draw[#1,oriented] (0,0) circle (0.5);
    }{
      \draw[#1,oriented] (0.5,0) arc[start angle=0,end angle=-360,radius=0.5];
    }
  \end{tikzpicture}}}
}

\newcommand{\Nvertex}[1]{
  \vcenter{\hbox{\begin{tikzpicture}[scale=0.6]
    \tikzset{lowarrow/.style={semithick, postaction={decorate},
                               decoration={markings, mark=at position 0.25 with {\arrow{>}}}}}

    \draw[strand,lowarrow] (-1.2,-1.2) -- (0,-0.2);
    \draw[strand,lowarrow] (-0.7,-1.2) -- (0,-0.2);

    \node at (0,-1.2) {\scriptsize $\cdots$};

    \draw[strand,lowarrow] (0.7,-1.2) -- (0,-0.2);
    \draw[strand,lowarrow] (1.2,-1.2) -- (0,-0.2);

    \filldraw (0,-0.2) circle (2pt);

    \draw[rdstrand,oriented] (0,-0.2) -- (0,1.0);
  \end{tikzpicture}}}
}

\newcommand{\dualNvertex}[1]{
  \vcenter{\hbox{\begin{tikzpicture}[scale=0.6]
    \tikzset{lowarrow/.style={semithick, postaction={decorate},
                               decoration={markings, mark=at position 0.25 with {\arrow{<}}}}}

    \draw[rdstrand,oriented] (0,-1.0) -- (0,-0.2);

    \filldraw (0,-0.2) circle (2pt);

    \draw[strand,lowarrow] (-1.2,0.8) -- (0,-0.2);
    \draw[strand,lowarrow] (-0.7,0.8) -- (0,-0.2);

    \node at (0,0.8) {\scriptsize $\cdots$};

    \draw[strand,lowarrow] (0.7,0.8) -- (0,-0.2);
    \draw[strand,lowarrow] (1.2,0.8) -- (0,-0.2);
  \end{tikzpicture}}}
}

\newcommand{\NtoNvertex}[1]{
  \vcenter{\hbox{\begin{tikzpicture}[scale=0.6]
    \tikzset{
      lowarrow/.style={semithick,postaction={decorate},
                        decoration={markings, mark=at position 0.25 with {\arrow{>}}}},
      dashedarrow/.style={semithick,postaction={decorate},
                           decoration={markings, mark=at position 0.8 with {\arrow{>}}}},
      outarrow/.style={semithick,postaction={decorate},
                        decoration={markings, mark=at position 0.25 with {\arrow{<}}}}
    }

    \draw[strand,lowarrow] (-1.2,-1.5) -- (0,-0.7);
    \draw[strand,lowarrow] (-0.7,-1.5) -- (0,-0.7);
    
    \node at (0,-1.5) {\scriptsize $\cdots$};

    \draw[strand,lowarrow] (0.7,-1.5) -- (0,-0.7);
    \draw[strand,lowarrow] (1.2,-1.5) -- (0,-0.7);

    \filldraw (0,-0.7) circle (2pt);

    \draw[rdstrand,dashedarrow] (0,-0.7) -- (0,0.0);

    \filldraw (0,0.0) circle (2pt);

    \draw[strand,outarrow] (-1.2,0.8) -- (0,0.0);
    \draw[strand,outarrow] (-0.7,0.8) -- (0,0.0);

    \node at (0,0.8) {\scriptsize $\cdots$};

    \draw[strand,outarrow] (0.7,0.8) -- (0,0.0);
    \draw[strand,outarrow] (1.2,0.8) -- (0,0.0);
  \end{tikzpicture}}}
}

\newcommand{\NtoNvertexSL}[1]{
  \vcenter{\hbox{\begin{tikzpicture}[scale=0.6]
    \tikzset{
      lowarrow/.style={semithick,postaction={decorate},
                        decoration={markings, mark=at position 0.25 with {\arrow{>}}}},
      dashedarrow/.style={semithick,postaction={decorate},
                           decoration={markings, mark=at position 0.8 with {\arrow{>}}}},
      outarrow/.style={semithick,postaction={decorate},
                        decoration={markings, mark=at position 0.25 with {\arrow{<}}}}
    }

    \draw[strand,lowarrow] (-1.2,-1.5) -- (0,-0.7);
    \draw[strand,lowarrow] (-0.7,-1.5) -- (0,-0.7);
    
    \node at (0,-1.5) {\scriptsize $\cdots$};

    \draw[strand,lowarrow] (0.7,-1.5) -- (0,-0.7);
    \draw[strand,lowarrow] (1.2,-1.5) -- (0,-0.7);

    \filldraw (0,-0.7) circle (2pt);

    \filldraw (0,0.0) circle (2pt);

    \draw[strand,outarrow] (-1.2,0.8) -- (0,0.0);
    \draw[strand,outarrow] (-0.7,0.8) -- (0,0.0);

    \node at (0,0.8) {\scriptsize $\cdots$};

    \draw[strand,outarrow] (0.7,0.8) -- (0,0.0);
    \draw[strand,outarrow] (1.2,0.8) -- (0,0.0);
  \end{tikzpicture}}}
}

\newcommand{\Nbox}[1]{
  \vcenter{\hbox{\begin{tikzpicture}[scale=0.6]
  
    \tikzset{lowarrow/.style={semithick, postaction=
                               {decorate},
                               decoration={markings, mark=at position 0.6 with {\arrow{>}}}}}

    \draw[strand,lowarrow] (-1.2,-1.5) -- (-1.2,-0.5);
    \draw[strand,lowarrow] (-0.8,-1.5) -- (-0.8,-0.5);
    \node at (0,-1) {\scriptsize $\cdots$};
    \draw[strand,lowarrow] (0.8,-1.5) -- (0.8,-0.5);
    \draw[strand,lowarrow] (1.2,-1.5) -- (1.2,-0.5);

    \draw[thick,fill=white] (-1.4,-0.5) rectangle (1.4,0.5);
    \node at (0,0) {$#1$};

    \draw[strand,lowarrow] (-1.2,0.5) -- (-1.2,1.5);
    \draw[strand,lowarrow] (-0.8,0.5) -- (-0.8,1.5);
    \node at (0,1) {\scriptsize $\cdots$};
    \draw[strand,lowarrow] (0.8,0.5) -- (0.8,1.5);
    \draw[strand,lowarrow] (1.2,0.5) -- (1.2,1.5);
  \end{tikzpicture}}}
}

\usepackage{hyperref}
\hypersetup{colorlinks=true,linkcolor=blue,citecolor=blue}

\usepackage{comment}

\usepackage{tikz}
\usepackage{tikz-cd}

\numberwithin{equation}{section}

\newtheorem{counter}{Counter}[section]

\newtheorem{definition}[counter]{Definition}
\newtheorem{definition-proposition}[counter]{Definition-Proposition}
\newtheorem{theorem}[counter]{Theorem}
\newtheorem*{theorem*}{Theorem}
\newtheorem{lemma}[counter]{Lemma}
\newtheorem{corollary}[counter]{Corollary}
\newtheorem*{corollary*}{Corollary}
\newtheorem{example}[counter]{Example}
\newtheorem{proposition}[counter]{Proposition}
\newtheorem*{proposition*}{Proposition}
\newtheorem{remark}[counter]{Remark}

\newcommand{\idty}{{\mathrm{1}\mkern-4mu{\mathchoice{}{}{\mskip-0.5mu}{\mskip-1mu}}\mathrm{l}}}
\newcommand{\1}{{\mathrm{1}\mkern-4mu{\mathchoice{}{}{\mskip-0.5mu}{\mskip-1mu}}\mathrm{l}}}
\newcommand{\CC}{\mathbb{C}}
\newcommand{\DD}{\mathbb{D}}
\newcommand{\PP}{\mathbb{P}}
\newcommand{\RR}{\mathbb{R}}
\newcommand{\ZZ}{\mathbb{Z}}

\DeclareMathOperator{\Hom}{Hom}

\DeclareMathOperator{\Fun}{Fun}
\DeclareMathOperator{\Rep}{Rep}
\DeclareMathOperator{\ev}{ev}
\DeclareMathOperator{\SL}{SL}
\DeclareMathOperator{\PGL}{PGL}
\DeclareMathOperator{\Frac}{Frac}
\DeclareMathOperator{\Dist}{Dist}
\renewcommand{\sl}{\mathfrak{sl}}

\newcommand{\A}{\mathcal{A}}
\newcommand{\B}{\mathcal{B}}
\newcommand{\C}{\mathcal{C}}
\newcommand{\D}{\mathcal{D}}
\newcommand{\E}{\mathcal{E}}
\newcommand{\G}{\mathcal{G}}
\newcommand{\K}{\mathcal{K}}
\newcommand{\M}{\mathcal{M}}

\newcommand{\T}{\mathcal{T}}
\renewcommand{\t}{\mathfrak{t}}
\newcommand{\g}{\mathfrak{g}}
\renewcommand{\L}{\mathcal{L}}

\renewcommand{\P}{\mathcal{P}}

\newcommand{\op}{\text{op}}
\newcommand{\bop}{\text{bop}}
\newcommand{\mop}{\text{mop}}
\newcommand{\Zdr}{\mathrm{Z}_{\text{Dr}}}

\renewcommand{\mod}{\mathrm{mod}}

\newcommand{\minus}{{\scalebox{0.6}[1.0]{$-$}}}

\newcommand{\DDsurf}{\DD_{1}}

\DeclareMathOperator{\Vect}{Vect}
\DeclareMathOperator{\Cat}{Cat}
\DeclareMathOperator{\Bimod}{Bimod}

\DeclareMathOperator{\Sk}{Sk}

\DeclareMathOperator{\SkCat}{SkCat}
\DeclareMathOperator{\SkAlg}{SkAlg}
\DeclareMathOperator{\BrTens}{BrTens}
\DeclareMathOperator{\Alg}{Alg}

\DeclareMathOperator{\id}{id}

\DeclareMathOperator{\tr}{tr}
\DeclareMathOperator{\RT}{RT} 
\DeclareMathOperator{\Rib}{Rib}
\DeclareMathOperator{\Planar}{Planar} 

\usepackage{xcolor}
\definecolor{our-orange}{HTML}{D86600}
\definecolor{our-blue}{HTML}{9BBBD6}
\definecolor{our-defect}{HTML}{77CCAA}
\definecolor{skein}{HTML}{8C0A60}

\title{Weyl defects in skein theory and quantum cluster charts}
\author{\parbox{\textwidth}{\centering Jennifer Brown, Juan Ramón Gómez García, David Jordan, and Matthias Vancraeynest}}

\begin{document}

\maketitle

\begin{abstract}

    We prove a modified invertibility property for the parabolic defects introduced in \cite{JLSS2021,Brown_Jordan_2025}. Namely, we show that the Borel defect cancels its dual, up to insertion of an invertible Weyl defect and up to restriction to an open-subcategory.  The main result holds for an arbitrary reductive group $G$ -- for illustration we give extended computations for $G=\SL_3$.
    
    Along the way, we introduce a defect skein theory with defects in both codimension one and two, which is compatible with gluing of defect 3-manifolds and defect surfaces.  Our results provide a potential defect skein theoretic construction of certain standard charts which appears frequently in quantum cluster varieties associated to surfaces.
\end{abstract}

\tableofcontents
\section{Introduction}

At the heart of geometric representation theory is the desire to express structures of a reductive group $G$ combinatorially via its Cartan subgroup $T$, and the associated Weyl group $W$.  Among the most useful and ubiquitous tools are parabolic induction/restriction along the Borel correspondence $T\twoheadleftarrow B\hookrightarrow G$, and the analysis of the resulting Weyl group symmetry.  This framework yields many of the most fundamental and spectacular successes of geometric representation theory:  a non-exhaustive list includes the Weyl character formulas \cite{Weyl1925,Weyl1926}, Borel-Weil-Bott theorem \cite{serre1954representations,bott1957homogeneous}, Beilinson--Bernstein localisation \cite{BeilinsonBernstein1981,Beilinson1983}, Springer theory \cite{Springer1976,Springer1978} and Lusztig's subsequent theory of character sheaves \cite{Lusztig1985a,Lusztig1985b,Lusztig1986a,Lusztig1986b}, Soergel bimodules \cite{soergel2007kazhdan,soergel1990kategorie} and KLR algebras \cite{Rouquier2008,KhovanovLauda2010,KhovanovLauda2011}. Indeed, from a certain perspective the entire geometric Langlands program is possible because of Hecke modifications, yet another instance of this paradigm. 

It has become increasingly understood in recent years that the skein theory associated to a reductive group $G$ is really concerned with a multiplicative deformation of the geometric representation theory of $G$ (specifically, via the Betti geometric Langlands program \cite{BenZvNadler2018}). Hence it is natural to expect that parabolic induction/restriction functors should play an equally important role in skein theory.
Indeed many important constructions in quantum topology -- quantum trace and abelianisation/non-abelianisation \cite{bonahon2011quantum,Le_2017,korinman2022quantum,Panitch_Park_2024,Garoufalidis_Yu_2024}, quantum cluster coordinates \cite{Fock_Goncharov_2009a,Fock_Goncharov_2009b}, spectral networks \cite{gaiotto2013spectral} and the quantum $A$-polynomial \cite{Garoufalidis_2004,Dimofte2011,Brown_Jordan_2025} -- all relate a complicated $G$-skein theory with a much simpler and more combinatorial $T$-skein theory, and so it is natural to seek a uniform explanation via parabolic induction/restriction.  This is certainly well-understood as a \emph{motivation} and \emph{interpretation} for those structures, however the foundations for parabolic induction/restriction in skein theory are much newer and not yet fully realised in these applications.

Our main motivation therefore is to build up the parabolic-induction and Weyl group toolkits fully locally, and to apply them uniformly to the problem of constructing geometric quantum cluster charts. 
Locality means roughly that one establishes higher Morita theory of the ribbon tensor category controlling the skein theory of the ball, and exports these using extended TQFT techniques to conclude properties on arbitrary surfaces and 3-manifolds.

Parabolic induction/restriction functors as codimension-one defects in skein theory have been recently introduced in \cite{Brown_Jordan_2025}, following \cite{JLSS2021}.
The motivating example in those works is the \emph{parabolic defect} $\D_B$, situated between $G$- and $T$-skeins, which implements parabolic reduction along the Borel correspondence $T\twoheadleftarrow B\hookrightarrow G$. These allow us to consider skein theories of bipartite surfaces and 3-manifolds, where open regions are coloured by $G$ and by $T$, with interfaces coloured by $B$. 
Our main new tool are \emph{Weyl defects}, which encode the Weyl group action on the weight lattice of $T$, and a further compatibility with parabolic induction/restriction functors.  Skein-theoretically, the Weyl group \emph{action} is modelled straightforwardly via invertible codimension-one defects in the $T$-skein theory, and the \emph{compatibility} with parabolic induction/restriction is encoded by allowing Weyl defects to end, in codimension-two, along Borel defects.

A beautiful property of the Borel defects -- witnessed and partially exploited in \cite{JLSS2021, Brown_Jordan_2025} in the case $G=\SL_2$ -- is that they are ``almost invertible".  By this, we mean that the composition of the Borel defect $\D_B$ with its dual defect $\D_B^\vee$ becomes invertible upon localisation to a natural subcategory.
This restriction corresponds cleanly to the localisation of cluster variables associated to edges of a triangulation in a geometric cluster chart of a decorated character variety.

Our main result interprets this key localisation step in terms of the generic invertibility of parabolic defects.
Stated precisely, it goes:
\begin{theorem*}
[Thm. \ref{thm:borel-almost-invertible}]
    Let $\widetilde{\D_B}$ be the redecorated parabolic defect associated to the Borel correspondence $T\twoheadleftarrow B\hookrightarrow G$. Denote by $w_0$ the longest word in the Weyl group and by $\T_{w_0}$ the Weyl associated defect. Then we have an equivalence of categories,
    \begin{align}
         \widehat{\SkCat}(\Digon)[A_\rho^{-1}] \cong \widehat{\SkCat}( \FullWeylCat )\label{eqn:weylcat}
     \end{align}
    induced by a $T\times T^{bop}$-equivariant isomorphism of internal skein algebras,
    \begin{equation}
      \SkAlg^{int}(\DigonAlg)[A_\rho^{-1}] \cong \SkAlg^{int}(\FullWeylAlg).
    \end{equation}
Here, we have localised the skein category and internal skein algebra at the skein $A_\rho$ connecting the two gates labelled by $V_\rho$ in the $G$-region. We use $\widehat{\cdot} \cong \Fun(\cdot^{op},\Vect)$ to denote the free co-completion of a linear category.
\end{theorem*}

Recall that skein categories are fully local, in the sense that the skein category of any surface can be computed from the skein category of the disk by repeatedly applying \emph{excision}.  More universally and directly, factorisation homology constructs the skein category as colimit, indexed over a suitable category of disk embeddings,  of the skein category of a disk.  In a similar way, defect skein categories can be constructed via stratified factorisation homology, and subsequently computed by excision.  Here, the initial data is the skein category of the disk, as well the skein category of all basic defect disks.

As a consequence of this locality, we obtain the following relation between skein categories of decorated and Weyl defect surfaces.

\begin{corollary*}[Cor. \ref{cor:general-surfaces-invertibility}]
  Let $\Sigma$ be an arbitrary redecorated surface, denote by $\Delta$ a collection of arcs on $\Sigma$ connecting boundary $T$ regions, by  $\widehat{\SkCat}(\Sigma)[\Delta^{-1}]$ the localisation of the skeins $A_\rho$ situated along each arc, and by $\Sigma_\triangle$ the Weyl defect surface where we compose the parabolic defects locally along each arc, as in \eqref{eqn:weylcat}. 
  Then we have an equivalence of categories,
\begin{align}
         \widehat{\SkCat}(\Sigma)[\Delta^{-1}] \cong \widehat{\SkCat}( \Sigma_\triangle ),
     \end{align}
    induced by an isomorphism of internal skein algebras
    \begin{equation}
      \SkAlg^{int}(\Sigma)[\Delta^{-1}] \cong \SkAlg^{int}(\Sigma_\triangle).
    \end{equation}
\end{corollary*}

Skein theory is intrinsically three dimensional, and our formalism extends naturally to skein modules of 3-manifolds.
In the non-defect setting skein theory gives a contravariant oriented (2+1)-TFT depending on a ribbon category $\A$:
\begin{equation}\label{eq:skein-TFT}
    \underline{\Sk}_\A: \mathrm{Cob}_{2+1}^{or,op} \to \Bimod.
\end{equation}
When defects are included so that multiple ribbon categories are at play, it's convenient to package the algebraic coefficient data in the cobordism category itself. We arrive at the notion of \emph{decorated 3-manifolds}, which provide cobordisms between decorated surfaces, see Definition \ref{def:decorated-manifold}.

\begin{figure}
  \centering
  \includegraphics[]{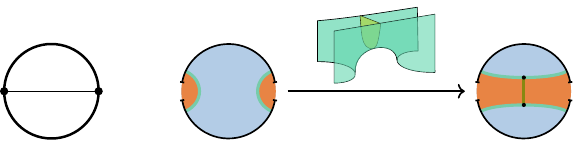}
  \caption{Localisation as a cobordism for the disk with two marked points, see Definition \ref{def:skein-cluster-functor}. Note that we have omitted the bulk regions from the cobordism.}\label{fig:abelian-cobordism}
\end{figure}

Taking advantage of \eqref{eq:skein-TFT} and using Theorem \ref{thm:borel-almost-invertible} to interpret localisation as a modification of defect data, we construct three dimensional decorated cobordism relating the un-localised and localised internal skein algebras, see Figure \ref{fig:abelian-cobordism}.
This extends to maps relating skein modules of the relevant surfaces.

\paragraph{Codimension-2 defects in skein theory}

Defects describe fundamental physical phenomena such as excitations, impurities, and interfaces, and therefore have been implemented in many physical models, including skein theory. They have been studied in blob homology, a derived generalisation of skein theory \cite{Morrison_Walker_2012}, and in Reshetikhin-Turaev theories \cite{CRS2018,Carqueville2018OrbifoldsOR,Carqueville_Meusburger_Schaumann_2020}, whose state spaces admit descriptions in terms of skein modules \cite{blanchet1995topological}.

A skein theoretic description of defects in these theories \cite{Kinnear_Runkel_2026} has opened the possibility for non-semisimple extensions of Reshetikhin-Turaev defect theory.
Recent work has established that defects model phenomena of independent interest in skein theory, including the Turaev coproduct on the HOMFLY-PT skein algebra \cite{GG2026homfly} and the quantum A-polynomial \cite{Brown_Jordan_2025}.

This paper contributes to this growing body of work by implementing Reshetikhin-Turaev evaluation and thereby skein relations in the presence of codimension-2 defects, see Section \ref{sec:3d-skein-relations}. We are mainly interested in junctions between codimension-1 defects, which arise naturally when composing defects (see Section \ref{sec:composition-junction}) or otherwise comparing them.

We emphasize that the codimension-2 defects considered in this paper are \emph{not skeins}, in that they are not ribbon graphs which locally represent morphisms in some ribbon category.

\begin{figure}
  \centering
  \includegraphics{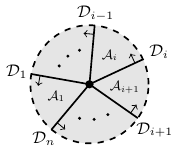}
  \caption{Codimension-2 defects as the junction of $n$ codimension-1 defects.}\label{fig:generic-junction}
\end{figure}

The local topological model for our codimension-2 defects is the stratified cylinder $C_n$ of Figure \ref{fig:annuli-and-thickened-cone}, where $n$ vertical walls meet at a central codimension-2 stratum.
Algebraic data is needed to determine the skein relations.
Let $\A_1,\ldots,\A_n$ be the ribbon categories labelling the bulk regions.
For each $i=1,\ldots,n$ taken modulo $n$, let $\D_i$ be an $(\A_i,\A_{i+1})$-central algebra decorating the interface between the $\A_i$- and $\A_{i+1}$-region, see Figure \ref{fig:generic-junction}.
Finally, let $\L$ be a left module category over the monoidal category $\tr\left(\D_1\boxtimes_{\A_2}\D_2\boxtimes_{\A_3}\cdots\boxtimes_{\A_n}\D_n\right)$.
We define a Reshetikhin-Turaev style evaluation functor for decorated ribbon graphs in $C_n$:

\begin{theorem*}[Thm. \ref{thm:skein-evaluation}] There exists a full linear functor $$\RT^n_\L : \Rib_\L^n \to \L$$ evaluating coloured ribbon graphs in the stratified cylinder $C_n$ into the underlying category $\L$.
\end{theorem*}

The proof of Theorem \ref{thm:skein-evaluation} centres around showing that the evaluation is unaffected under a host of choices, notably stratified isotopy representative of decorated ribbon graphs in $C_n$. 

\subsection{Organisation of the paper}\label{sec:organisation}

Section \ref{sec:defects} introduces codimension-2 defects in skein theory, starting with a review of stratified spaces. Section \ref{sec:surface_defects} recaps the codimension-1 defect theory then introduces the notion of dual defects and their relation to orientation.
Section \ref{sec:centered-bimod} starts with a review of the algebraic structures and operations needed to define and work with codimension-2 defects, then goes on to define decorated ribbon graphs (Definition \ref{def:decorated-ribbon-graph}), which live in decorated 3-manifolds (Definition \ref{def:decorated-manifold}) and form the basis for skeins in the presence of codimension-2 defects.

These are all setup for Section \ref{sec:evaluation}, which contains a construction of Reshetikhin-Turaev style evaluation of decorated ribbon graphs in the neighbourhood of a codimension-2 defect. In the interest of being self-contained, this is followed in Section \ref{sec:skein-modules-and-categories} by a review of skein categories and modules along with the associated internal constructions. Finally Section \ref{sec:composition-junction} introduces the special case of codimension-2 defects arising from the composition of codimension-1 defects.

Section \ref{sec:weyl-and-parabolic-defects} focuses on an important example in codimension-1: the Weyl defect. The relationship to parabolic defects is established in Section \ref{sec:weyl-parabolic-junction}, and this is used to formulate the generic invertibility of parabolic defects in Section \ref{sec:invertible-statement}. This is followed by the formulation of localisation as a stratified cobordism in section \ref{sec:localisation-cobordism}.

Section \ref{sec:monadic-reconstruction-without-G} establishes that taking $G$-invariants (an algebraic procedure which restricts where skeins can end) is a categorical equivalence in an appropriate localisation of freely co-completed skein categories. This result is used in Section \ref{sec:proof-of-generic-invertibility} to prove generic invertibility of parabolic defects.

Finally, Section \ref{sec:quantum-cluster-algebras} uses the parabolic defect theory developed here and in \cite{Brown_Jordan_2025} to study the relationship between decorated skein theory and quantum cluster structures on decorated character varieties of surfaces.
It starts in Section \ref{sec:clusters-background} with a review of quantum cluster algebras, then goes on in Section \ref{sec:computations-SL3-PGL3} to prove their equivalence for the triangle and $G=\SL_3\CC$, recovering a result of \cite{IY22}.

\subsection{AI disclosure}\label{sec:AI-disclosure}

Claude Sonnet was used to create TikZ figures, to search for additional references, and for proof-reading/spell-checking. All computations in the final article were done by the authors and no AI agents contributed writing or editing.

\subsection{Acknowledgements}\label{sec:acknowledgements}
We would like to thank Alexander Shapiro and Gus Schrader for sharing the groundwork for Section \ref{sec:monadic-reconstruction-without-G}, unpublished work stemming from a previous collaboration with the third author.

The work of JB was funded by NSF grant DMS-2202753 and by the ERC under the EU's Horizon 2020 programme grant agreement No 948885.
The work of DJ was supported by the EPSRC Open Fellowship “Complex Quantum Topology”, grant number EP/Y008812/1 and by the Simons Foundation award 888988 as part of the Simons Collaboration on Global Categorical Symmetry.

\section{Codimension-2 defects in skein theory}\label{sec:defects}

Here we extend the framework of \cite{Brown_Jordan_2025} to include codimension-2 defects as interfaces between codimension-1 defects.
Our approach follows the philosophy that defects are determined by high morphisms in an appropriate Morita category, as developed in \cite{Haugseng_2017,Johnson-Freyd_Scheimbauer_2017}. 

\begin{figure}
\centering
\includegraphics{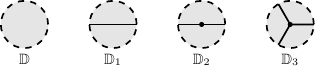}
\caption{Defects are modelled locally on thickened stratified disks. We will denote the trivially stratified disk by $\DD$, and use $\DD_n$ to denote a disk with $n$ codimension-1 strata as shown here.}\label{fig:little-disks}
\end{figure}
We recall first the basic notions on stratified spaces that we will use throughout the section:

\begin{definition}
    \emph{\textbf{(i)}} Let $P$ be a poset endowed with the topology whose open sets are generated by $P_{>p}\coloneqq\{x\in P\mid x>p\},$ for $p\in P$. A \emph{\bf stratified manifold indexed by $P$} is an oriented manifold $X$ equipped with a continuous map $\phi\colon X\to P,$ such that $X_p\coloneqq\phi^{-1}(\{p\})$ is a submanifold for each $p\in P.$ We call $X_p$ a \emph{\bf stratum} and assume that each stratum comes with an orientation.

    \noindent\emph{\textbf{(ii)}} A \emph{ \bf morphism of stratified manifolds} between $(X,\phi_1\colon X\to P)$ and $(Y,\phi_2\colon Y\to Q)$ consists of a smooth map $f\colon X\to Y$ and a continuous map $F\colon P\to Q$ such that 
    \begin{equation}\begin{tikzcd}
        X \arrow[d, "\phi_1"'] \arrow[r, "f"] & Y \arrow[d, "\phi_2"] \\
        P \arrow[r, "F"]                      & Q                    
\end{tikzcd}\end{equation} is a commutative diagram. This means that $f(X_p)\subseteq Y_{F(p)}$, for every $p\in P$. If $f$ is an embedding, we say that $(f,F)$ is a \emph{\bf stratified embedding}.

\noindent\emph{\textbf{(iii)}} A \emph{\bf stratified isotopy} between two embeddings $(f,F),(g,G)\colon(X,\phi_1)\to(Y,\phi_2)$ consists of a smooth map $h\colon X\times I\to Y$ and a continuous map $H\colon P\to Q$ such that the diagram 
\begin{equation}\begin{tikzcd}
        X\times I \arrow[d, "\phi_1\circ \mathrm{pr_1}"'] \arrow[r, "h"] & Y \arrow[d, "\phi_2"] \\
        P \arrow[r, "H"]                                                 & Q,                    
\end{tikzcd}\end{equation} where $\mathrm{pr}_1$ is the projection onto the first factor, is commutative. A \emph{\bf stratified isotopy of $X$} is a stratified isotopy $(h,H)\colon X\times I\to X$ such that $h(-,0)=\id_X$ and $H=\id_P$.
\end{definition}

The decorated ribbon graphs (Definition \ref{def:decorated-ribbon-graph}) which form the basis for skeins will be taken up to stratified isotopy. This means that in particular a skein can not be made to intersect or stop intersecting a defect via isotopy, only via skein relations.

\subsection{Codimension-1 defects: a review}\label{sec:surface_defects}
We give a brief review of the theory with only codimension-1 defects, see \cite{Brown_Jordan_2025} for full details. 

In what follows, let $\A^{\bop}$ denote the braided opposite (with braiding $c^{-1}_{y,x} : x\otimes y \to y\otimes x$), $\A^{\mop}$ the monoidal opposite (with braiding $c^{-1}_{x,y} : x\otimes^{\op} y \to y\otimes^{\op} x$) and $\A^{\op}$ the usual opposite of a braided monoidal category $\A$.

Let $M$ be a bipartite $3$-manifold \cite[Def. 2.2]{Brown_Jordan_2025} and let $s\in\pi_0(M_2)$. The orientation of $s$ determines an outward normal vector field that, by convention, points towards the \emph{target} bulk region $s_1\in\pi_0(M_3)$. Let $s_0\in\pi_0(M_3)$ be the \emph{source} bulk region and suppose that $s_0$ and $s_1$ are labelled with ribbon categories $\C$ and $\A$, respectively. Then, the algebraic data of the codimension-1 defect $s$ is a 1-morphism $\C\to\A$ in $\BrTens$, given by the following structure. 

 If $\A,\C$ are braided monoidal categories, a $1$-morphism $\C \to \A$ is an algebra object $\D$ in the category of $(\A,\C)$-bimodules or, equivalently \cite[Def-Prop. 3.2]{BJS2021}, a braided tensor functor 
 \begin{equation}
     (H,\sigma): \A\boxtimes\C^{\bop} \to \Zdr(\D),
 \end{equation} where $\Zdr(\D)$ is the \emph{Drinfeld centre} of $\D$. We will call $\D$ a \emph{$(\A,\C)$-central algebra} and denote the coherence isomorphisms by 
 \begin{equation}
 J^H : H(-) \otimes H(-) \Rightarrow H(-\otimes -)\quad\text{and}\quad J_0^H\colon\boldsymbol{1}_\D\Rightarrow H(\boldsymbol{1}_\C).
 \end{equation}
If $\D$ is pivotal, $\A,\C$ are ribbon, and the central structure respects the pivotal isomorphisms, then we get an evaluation functor \cite[Def-Prop. 2.16]{Brown_Jordan_2025}
\begin{equation}
  \RT_\D : \Rib_\D \to \D
\end{equation}
which generalises Reshetikhin-Turaev evaluation of coloured ribbon graphs. Here $\Rib_\D$ is the category whose hom-spaces are spanned by ribbon graphs in a stratified cylinder with a vertical defect wall (see \cite[Definition 2.11]{Brown_Jordan_2025}). With or without defects, two linear combinations of coloured ribbon graphs are \emph{skein equivalent} when they evaluate to the same morphism:
\begin{equation}
  \sum_{i=0}^n a_i \Gamma_i \sim \sum_{i=0}^m b_i \Gamma_i' \quad \Leftrightarrow \quad \RT\left(\sum_{i=0}^n a_i \Gamma_i\right) = \RT\left(\sum_{i=0}^m b_i \Gamma_i'\right).
\end{equation}
From defect skein relations it's a short step to define skein categories, functors, algebras, and modules for stratified spaces.

Before extending this story to include codimension-2 defects, we take a moment to describe some algebraic operations concerning codimension-1 defects which did not appear in \cite{Brown_Jordan_2025}.
First, there's a notion of dual defects, which are built on the adjoints of the underlying central algebras.
    Let $\D$ be an algebra object in $(\A,\C)$-bimodules, determined by a ribbon functor $\A\boxtimes \C^{\bop} \to \Zdr(\D)$.
    Taking the braided opposite produces a functor 
    \begin{equation}
        (\A\boxtimes \C^{\bop})^{\bop} \cong \C \boxtimes \A^{\bop} \to \Zdr(\D)^{\bop}
    \end{equation}.
    The functor
    \begin{equation}
        (x,\sigma : x\otimes - \to - \otimes x) \mapsto (x,\sigma^{-1})
    \end{equation}
    induces a categorical equivalence $\Zdr(\D)^{\bop} \overset{\cong}{\longrightarrow} \Zdr(\D^{\mop})$.  
    In this way $\D^{\mop}$ is the (right) adjoint of $\D$ as a 1-morphism in $\BrTens$.
\begin{definition}\label{def:dual-defect}
  The \emph{\bf dual defect $\D^\vee$} of a defect $\D$ is that induced by the central structure given by composition
  \begin{equation}
    \C \boxtimes \A^{\bop} \to \Zdr(\D)^{\bop} \to \Zdr(\D^{\mop}).
  \end{equation}
\end{definition}
 The dual defect always exists, see e.g. \cite[Prop 4.14]{Gwilliam_Scheimbauer_2018}.

\begin{lemma}\label{lemma:double-dual-defect}
  Every $\A$-central algebra $\D$ is canonically isomorphic to its double dual $\D^{\vee\vee}$.
\end{lemma}
\begin{proof}
  Let $(H,\sigma)\colon\A\to \Zdr(\D)$ be the central functor on $\D$. Then, $\D^{\vee\vee}$ is also an $\A$-central algebra with same underlying functor $H$ and half braiding $(\sigma^{-1})^{-1} = \sigma$.
  It remains to show that its monoidal structure has the same coherence isomorphism. First note that the monoidal structure $(J^H)^\vee$ on the central functor $(H,\sigma^{-1})\colon\A^{\bop}\to \Zdr(\D^{\mop})$ is given by
  \begin{equation}\label{eq:monoidal-structure}
        (J^H)^\vee_{x,y} : H(x)\otimes^{\op} H(y)=H(y)\otimes H(x)\xrightarrow{J^{H}_{y,x}} H(y\otimes x)\xrightarrow{H(\beta^\A_{x,y})^{-1}}H(x\otimes y),
    \end{equation}
    where $\beta^\A$ is the braiding of $\A$ and $J^{H}$ is the coherence isomorphism for the monoidal structure of $H.$
    Applying this formulation twice, we obtain $(J^H)^{\vee\vee}$ as 
    \begin{equation}
        H(x)\otimes H(y)=H(y)\otimes^{\op} H(x)\xrightarrow{(J^H)^\vee_{y,x}} H(y\otimes x)\xrightarrow{H(\beta^{\A^{\bop}}_{x,y})^{-1}=H(\beta^\A_{y,x})}H(x\otimes y).
    \end{equation} 
  Using the expression for $(J^H)^\vee$ in \eqref{eq:monoidal-structure} we recover the expression for $J^H$:
  \begin{equation}
      (J^H)^{\vee\vee}_{x,y} = H(\beta^\A_{y,x})\circ (J^H)^\vee_{y,x}=H(\beta^\A_{y,x})\circ H(\beta^\A_{y,x})^{-1}\circ J^{H}_{x,y}=J^H_{x,y},
  \end{equation}
  so we have an identity $\D=\D^{\vee\vee}.$
\end{proof}

\begin{figure}[t]
    \centering
    \includegraphics[scale=1.3, valign=c]{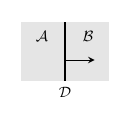}$\qquad\leftrightsquigarrow\qquad$\includegraphics[scale=1.3, valign=c]{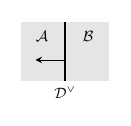} 
    \caption{
     In codimension-1 reversing orientation corresponds to reassigning the dual defect data, see Remark \ref{rmk:orientation_and_evaluation_morphism}.}
    \label{fig:defect_orientations}
\end{figure}

The operation of taking duals corresponds topologically to reversing orientations on codimension-1 defects:

\begin{remark}(Orientation of the defect and evaluation morphism.)\label{rmk:orientation_and_evaluation_morphism}
    The algebraic data of a defect is determined by the algebraic data of the bulk regions and the orientation of the defect. Reversing the orientation of a codimension-1 defect $\D$ and changing its algebraic data to $\D^\vee$ does not affect  the evaluation morphism defined in \cite{Brown_Jordan_2025}. Indeed, when we reverse the orientation, positive crossings become negative, and vice versa. For a positive crossing, the evaluation morphism for $\D$ will assign the corresponding half-braiding $\sigma^\D$. Reversing the orientation, the crossing becomes negative, so the evaluation morphism for $\D^\vee$ will send it to $(\sigma^{\D^\vee})^{-1}=\sigma^\D$.
\end{remark}

Next, the composition of two central structures $\A\overset{\D_1}{\longrightarrow}\B \overset{\D_2}{\longrightarrow}\C$ is given by the \emph{balanced (or relative) Deligne--Kelly tensor product} $\D_2\boxtimes_{\B}\D_1$ as an algebra object in $(\C,\A)$-bimodules \cite{Douglas_Schommer-Pries_Snyder_2019,BJS2021}. This is defined by a universal property and therefore up to unique isomorphism.
\begin{definition}\label{def:rel-DK-and-trace}
    Given a $(\A,\A)$-bimodule category $\mathcal{M}$, the \emph{\bf trace of $\mathcal{M}$} is a presentable category $\tr(\mathcal{M})$ together with a functor $\mathcal{M} \to \tr(\mathcal{M})$ which induces a categorical equivalence
\begin{align}
  \Hom_{\Pr}(\tr(\mathcal{M}), \mathcal{E}) \cong \Hom^{bal}_{\Pr}(\mathcal{M}, \mathcal{E})
\end{align}
    for all $\K$-linear $\E$. Here $bal$ refers to $\A$-balanced functors.

    When $\M=\D_1\boxtimes\D_2$ for $\D_1$ a right $\A$-module and $\D_2$ a left $\A$-module, the trace $\tr(\M)$ is called the \emph{\bf balanced Deligne--Kelly tensor product} and denoted $ \D_1\boxtimes_\A\D_2$.
    If $\D_1$ is a $(\B,\A)$-bimodule and $\D_2$ an $(\A,\C)$-bimodule, then $\D_1\boxtimes_\A \D_2$ will be a $(\B,\C)$-bimodule.
\end{definition}

On the level of skein categories and modules, we would like the composition of codimension-1 defects to be a local operation, i.e. one which takes place in a ball without changing the larger decorated manifold.
This preference for locality leads us to introduce codimension-2 defects between the composed and the delineated codimension-1 defects, see Section \ref{sec:composition-junction}.

\subsection{The algebraic data: centered bimodules}\label{sec:centered-bimod}

We will now dive into the algebraic structures needed to define skein relations in the presence of codimension-2 defects. 
This will culminate in the definition of a decorated 3-manifold (Def. \ref{def:decorated-manifold}.)
In Section \ref{sec:evaluation} we construct the evaluation functor that determines the precise relations.

Let $\D_1,\D_2$ be $(\C,\A)$-central algebras with central functors $(H_i,\sigma^i) : \C\boxtimes\A^{\bop} \to \Zdr(\D_i)$.
A \emph{codimension-2 defect} connecting them is determined by a 2-morphism $\L$ in $\BrTens$:
\begin{equation}
    \begin{tikzcd}
\A \arrow[r, bend left=50, ""{name=U,below},"\D_1"{above}]
\arrow[r, bend right=50, "\D_2"{below}, ""{name=D}]
& \C.
\arrow[Rightarrow, from=U, to=D,"\L"]
\end{tikzcd}
\end{equation}
These are given by centered bimodules:

\begin{definition}\label{def:centered-bimodule}
    Let $\A$ be a braided category and $\D_1,\D_2$ be $\A$-central algebras. An \emph{\bf $\A$-centered $(\D_1,\D_2)$-bimodule} is a $\K$-linear category $\L$ together with a $\K$-linear functor 
    \begin{equation}
       \D_1 \boxtimes \L \boxtimes \D_2 \to \L;\qquad d_1\boxtimes\ell\boxtimes d_2 \mapsto d_1\triangleright\ell\triangleleft d_2
    \end{equation}
    with bimodule associators
    \begin{equation}
      \begin{aligned}
        (a\otimes b) \triangleright \ell \xrightarrow{m^L_{a,b,\ell}} a\triangleright (b\triangleright \ell), &\quad& \ell \triangleleft (c\otimes d) \xrightarrow{m^R_{c,d,\ell}} (\ell \triangleleft c) \triangleleft d
      \end{aligned}
    \end{equation}
    which are compatible with each other, and both of which satisfy the pentagon  and triangle relations \cite{EGNO}. Together, $m^L$ and $m^R$ endow $\L$ with the structure of a $(\D_1,\D_2)$-bimodule category.
    This comes with a natural transformation called the \emph{\bf $\A$-centered structure} or \emph{\bf balancing}:
    \begin{equation}\label{eq:balancing}
        \eta_{x,\ell} : H_1(x)\triangleright \ell \overset{\cong}{\longrightarrow} \ell\triangleleft H_2(x),
    \end{equation}
    such that the following diagrams commute:
    \medskip

    \noindent\textbf{Left half-braiding:}
    \begin{equation}\label{eq:left-half-braiding}
        \begin{tikzcd}[row sep=large, column sep=large]
            (H_1(x) \otimes d_1) \triangleright \ell \arrow[r, "m^L"] \arrow[d, "\sigma^1_{H_1(x),d_1} \triangleright \id_\ell"'] 
            & H_1(x) \triangleright d_1 \triangleright \ell \arrow[r, "\eta_{x,d_1\triangleright\ell}"] 
            & d_1 \triangleright \ell \triangleleft H_2(x) \\
            (d_1 \otimes H_1(x)) \triangleright \ell \arrow[r, "m^L"'] 
            & d_1 \triangleright H_1(x) \triangleright \ell \arrow[ur, "\id_{d_1} \triangleright \eta_{x,\ell}"'] 
            &
        \end{tikzcd}
    \end{equation}

    \noindent\textbf{Right half-braiding:}

    \begin{equation}\label{eq:right-half-braiding}
        \begin{tikzcd}[row sep=large, column sep=large]
            \ell \triangleleft (d_2 \otimes H_2(x)) \arrow[r, "m^R"] 
            & \ell \triangleleft d_2 \triangleleft H_2(x) 
            & H_1(x) \triangleright \ell \triangleleft d_2 \arrow[l, "\eta_{x,\ell\triangleleft d_2}"'] \arrow[dl, "\eta_{x,\ell} \triangleleft \id_{d_2}"] \\
            \ell \triangleleft (H_2(x) \otimes d_2) \arrow[u, "\id_\ell \triangleleft \sigma^2_{H_2\!(x),d_2}"] \arrow[r, "m^R"'] 
            & \ell \triangleleft H_2(x) \triangleleft d_2 
            &
        \end{tikzcd}
    \end{equation}

    \noindent\textbf{Monoidal structure:}
    \begin{equation}\label{eq:monoidal-structure-centered-bimodule}
            \begin{tikzcd}[row sep=large, column sep=large]
                H_1(x\otimes y) \triangleright \ell 
                \arrow[r, "\eta_{x\otimes y,\ell}"] 
                \arrow[d, "m^L\circ(J^{H_1}_{x,y}\triangleright\id_\ell)"'] 
                & \ell \triangleleft H_2(x\otimes y) 
                \arrow[r, "m^R\circ(\id\triangleleft J^{H_2}_{x,y})"] 
                & \ell \triangleleft H_2(x) \triangleleft H_2(y) \\
                H_1(x) \triangleright H_1(y) \triangleright \ell 
                \arrow[r, "\eta_{x,H_1(y)\triangleright\ell}"'] 
                & H_1(y) \triangleright \ell \triangleleft H_2(x) 
                \arrow[ru, "\eta_{y,\ell\triangleleft H_2(x)}"'] 
                  & 
        \end{tikzcd}
    \end{equation}
    where we have used the shorthand $d_1 \triangleright \ell \triangleleft \idty = d_1 \triangleright \ell$ and $\idty \triangleright \ell \triangleleft d_2 = \ell \triangleleft d_2$.
    
    A \emph{\bf pointed centered bimodule} is a centered bimodule with a distinguished object $\Dist \in \L$.
\end{definition}

\begin{remark}
  Given a pointed $\A$-centered $(\D_1,\D_2)$-bimodule $\L$, the pointing determines functors
  \begin{equation}
    \begin{array}{ccc}
        \D_1 & \to & \L \\
        d_1 & \mapsto & d_1\triangleright\Dist 
    \end{array}\qquad\text{and}\qquad
    \begin{array}{ccc}
        \D_2 & \to & \L \\
        d_2 & \mapsto &  \Dist\triangleleft d_2 
    \end{array}
  \end{equation}
which will feature crucially in our definition of skein relations near a codimension-2 defect.
  Evaluation in the non-pointed theory could be defined as in this article with the extra condition that every codimension-2 defect contains a skein along its entire length.
  In our motivating example of parabolic and Weyl defects, $\L$ typically comes from a monoidal category with distinguished object the monoidal unit.
\end{remark}

As with codimension-1 defects, the algebraic data associated to a codimension-2 defect will depend on its orientation in relation to those of its adjacent codimension-1 defects $\D_1,\D_2$. Up to switching orientations and taking duals, we can always suppose that $\D_1$ and $\D_2$ have the same source $\A$ and target $\B$, as in the first and last images in Figure \ref{fig:normal-vectors}, and that the $(\A,\B)$ bimodule structures on $\L$ induced by $\D_1$ and $\D_2$ agree. In this way the codimension-2 defect $\L$ has the structure of a $(\B,\A)$-centered $(\D_1,\D_2)$-bimodule.
If the orientation on the line induced by that of the surface coincides with the actual orientation of the line, then the corresponding defect $\D_i$ acts on the left on $\L$. Otherwise it acts on the right. The following lemma guarantees that replacing an adjacent codimension-1 defect with its dual does not affect the algebraic structure of $\L$:

\begin{lemma}\label{lemma:dual-line}
  Let $\D_1$ and $\D_2$ be two $\A$-central algebras and $\L$ an $\A$-centered $(\D_1,\D_2)$-bimodule. Then, $\L$ is canonically an $\A^{\bop}$-centered $(\D_2^\vee,\D_1^\vee)$-bimodule.
\end{lemma}
\begin{proof}
  The categories underlying $\D_1^\vee$ and $\D_2^\vee$ are $\D_1^{\mop}$ and $\D_2^{\mop}$, respectively, so $\L$ is a $(\D_2^\vee,\D_1^\vee)$-bimodule. We write
  \begin{equation}
  d_2\blacktriangleright\ell\blacktriangleleft d_1 := d_1\triangleright\ell\triangleleft d_2
  \end{equation}
  for the corresponding action.
  Let us check that the following defines an $\A^{bop}$-centered structure:
  
  \begin{equation}
  \eta^{-1}_{a,\ell}\colon H_2(a)\blacktriangleright\ell=\ell\triangleleft H_2(a)\to H_1(a)\triangleright\ell=\ell\blacktriangleleft H_1(a)
  \end{equation}
  From the definition of dual defect, the left half-braiding condition \eqref{eq:left-half-braiding} is commutativity of the diagram
  \begin{equation}
        \begin{tikzcd}[column sep=huge, row sep=large]
            (H_2(x)\otimes^{\op}d_2)\blacktriangleright\ell \arrow[d, "{\left(\sigma^2\right)^{-1}_{H_2(x),d_2}\blacktriangleright\id_\ell}"'] \arrow[r, "{m^R_{d_2,H_2(x),\ell}}"] & H_2(x)\blacktriangleright d_2\blacktriangleright\ell \arrow[r, "{\eta^{-1}_{x,d_2\blacktriangleright\ell}}"]    & d_2\blacktriangleright\ell\blacktriangleleft H_1(x) \\
            (d_2\otimes^{\op} H_2(x))\blacktriangleright\ell \arrow[r, "{m^R_{H_2(x),d_2,\ell}}"]                                                                                   & d_2\blacktriangleright H_2(x)\blacktriangleright\ell. \arrow[ru, "{\id_{d_2}\blacktriangleright\eta_{x,\ell}}"'] &                                                    
        \end{tikzcd}
    \end{equation}
      Comparing with \eqref{eq:right-half-braiding}, we see that this is exactly the compatibility of the $\A$-centered structure of $\L$ with the right half-braiding, so the claim holds. A completely analogous argument holds for the right half-braiding. For the monoidal condition \eqref{eq:monoidal-structure-centered-bimodule}, recall that the monoidal structure on $(H_i,(\sigma^i)^{-1})$ is given by \eqref{eq:monoidal-structure}. We therefore need to check the commutativity of the diagram 
      \begin{equation}
          \begin{tikzcd}[row sep = large, column sep = 7 em]
                H_2(x)\blacktriangleright H_2(y)\blacktriangleright\ell \arrow[r, "{m^R_{H_2(y),H_2(x),\ell}}"] \arrow[d, "{\eta^{-1}_{x,H_2(y)\blacktriangleright\ell}}"']  & H_2(y\otimes x)\blacktriangleright\ell \arrow[r, "{H_2(\beta^\A_{y,x})^{-1}\blacktriangleright\id_\ell}"]  & H_2(x\otimes y)\blacktriangleright\ell \arrow[dd, "{\eta^{-1}_{x\otimes y,\ell}}"] \\
                H_2(y)\blacktriangleright\ell\blacktriangleleft H_1(x) \arrow[d, "{\eta^{-1}_{y,\ell\blacktriangleleft H_2(x)}}"']                                                                                                                     &  &                                                                                    \\
                \ell\blacktriangleleft H_1(x)\blacktriangleleft H_1(y)     \arrow[r, "{m^L_{H_1(x),H_1(y),\ell}}"]                                                                 & \ell\blacktriangleleft H_1(y\otimes x) \arrow[r, "{\id_\ell\blacktriangleleft H_1(\beta^\A_{y,x})^{-1}}"]  & \ell\blacktriangleleft H_1(x\otimes y).                                            
            \end{tikzcd}
      \end{equation}
      By the compatibility of the $\A$-centered structure of $\L$ with the monoidal structure, we can replace the composition of the two vertical arrows on the left by $\eta^{-1}_{x\otimes y,\ell}.$
      The resulting square is then commutative by the naturality of $\eta$.
\end{proof}

\begin{remark}
    In particular, since $\left(\A\boxtimes\C^{\bop}\right)^{\bop} \cong \C\boxtimes \A^{\bop}$, this implies that an $(\A,\C)$-centered $(\D_1,\D_2)$-bimodule $\L$ is canonically a $(\C,\A)$-centered $(\D_2^\vee,\D_1^\vee)$-bimodule.
\end{remark}

\begin{figure}
  \centering
  \includegraphics[scale=1.2]{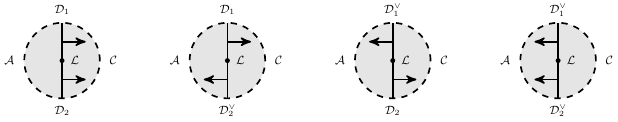}
  \caption{Suppose that the defect line is oriented out of the page. $\L$ is a $(\C,\A)$-centered $(\D_2,\D_1)$-bimodule, where $\D_1,\D_2$ are $(\C,\A)$-central algebras.}
  \label{fig:normal-vectors}
\end{figure}

\begin{definition}\label{def:data-for-many-surfaces}
Suppose there are $n$ codimension-1 defects $\D_1,\ldots \D_n$ meeting at a codimension-2 defect, as in Figure \ref{fig:generic-junction}, with each $\D_i$ an $(\A_i,A_{i+1})$-central algebra. For simplicity's sake we're assuming that the outward normals of the surfaces follow the right hand rule along the orientation of the codimension-2 defect, i.e. they point counter-clockwise in Figure \ref{fig:generic-junction}.
The \emph{\bf data specifying a codimension-2 defect $\L$} at the junction is a pointed ${\tr(\D_1\boxtimes_{\A_2}\cdots\boxtimes_{\A_{n}}\D_n)}$-module, where $\tr$ is the trace as defined in Definition \ref{def:rel-DK-and-trace}.
\end{definition}

\begin{figure}
  \centering
  \includegraphics{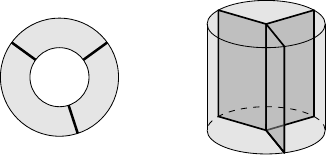}
  \caption{\textbf{Left:} The stratified annulus $\Gamma \times [0,1]$. \textbf{Right:} The thickened cone $C(\Gamma)\times[0,1]$.}\label{fig:annuli-and-thickened-cone}
\end{figure}

We can unpack the data of Definition \ref{def:data-for-many-surfaces} as follows:
\begin{lemma}\label{lemma:unpacking-line-data}
    A $\tr(\D_1\boxtimes_{\A_2}\cdots\boxtimes_{\A_{n}}\D_n)$-module is equivalently a $\D_1\boxtimes\cdots\boxtimes \D_n$-module which admits the structure of an $\A_i$-centered $(\D_{i},\D_{i-1}^\vee)$-bimodule, for $i = 1,\ldots ,n$ taken modulo $n$.  
\end{lemma}
\begin{proof}
    This follows directly from the definitions.
The relative Deligne-Kelly tensor product of a left $\A$-module $\mathcal{M}$ and right $\A$-module $\mathcal{N}$ is the colimit of the following diagram 
\begin{equation}
    \begin{tikzcd}
      \A\times \mathcal{M} \boxtimes \mathcal{N} \ar[r,yshift=3pt,"\triangleright"] \ar[r,yshift=-3pt,"\triangleleft"']& \mathcal{M} \boxtimes \mathcal{N} 
    \end{tikzcd}
\end{equation}
When $\A=\A_i$, $\mathcal{M} = \D_{i-1}^\vee$, and $\mathcal{N}=\D_i$, this is exactly the diagram whose commutativity characterises an $\A_i$-centered $(\D_i,\D_{i-1}^\vee)$-module. So by construction $\D_1\boxtimes_{\A_2} \cdots \boxtimes_{\A_{n}} \D_n$ has the structure of an $\A_i$-centered $(\D_i,\D_{i-1}^\vee)$-bimodule for $i=1,\ldots,n$, taken modulo $n$.

  The trace is likewise a colimit over the diagram which encodes the $\A_1$-balanced $(\D_1,\D_n^\vee)$-bimodule action, and hence itself inherits this same structure.
  \end{proof}

  This data can be described compactly in terms of defect skein categories, see Definition \ref{def:defect-skein-cat}.

\begin{remark}\label{rem:skeiny-version-of-defect-data}
    Every point on a codimension-2 defect has a neighbourhood which is stratified isotopic to a thickened cone $C(\Gamma)\times[0,1]$ on some stratified circle $\Gamma$, see Figure \ref{fig:annuli-and-thickened-cone}.
By excision, the skein category (Definition \ref{def:defect-skein-cat}) of this annular neighbourhood is exactly the category appearing in Definition \ref{def:data-for-many-surfaces}:
\begin{equation}
    \mathcal{S}_\Gamma := \SkCat(\Gamma\times[0,1]) \cong \tr(\D_1\boxtimes_{\A_2}\cdots\boxtimes_{\A_{n}}\D_n)
\end{equation}
Its product is induced by an embedding $\Gamma \times [0,1] \sqcup \Gamma \times [0,1] \hookrightarrow \Gamma\times [0,1]$ and monoidal unit determined by $\varnothing \hookrightarrow \Gamma\times [0,1]$.

In particular, the skein category $\SkCat(\Sigma)$ of any decorated surface $\Sigma$ with an embedding $\Gamma \hookrightarrow \partial \Sigma$ determines a codimension-2 defect on $\C(\Gamma)\times[0,1]$.
\end{remark}

All together, the following assignment of algebraic data will be used to determine skein theory in a 3-manifold with surface and line defects.

\begin{definition}\label{def:decorated-manifold}
  A \emph{\bf decorated 3-manifold} $M$ is defined as follows. Let $\phi : M \to \{1 \geq 2 \geq 3\}$ be an oriented stratified 3-manifold such that each $M_k := \phi^{-1}(k)$ is $k$-dimensional, and $M_2$ is a collection of smoothly embedded oriented surfaces meeting $\partial M$ transversely, if at all. Similarly, $M_1$ is permitted to meet the boundary transversely.
  We \emph{\bf decorate} $M$ by assigning the following algebraic data to its \emph{\bf regions}, i.e. the connected components of its strata:
  \begin{list}{\textbullet}{}
    \item For each $r \in \pi_0(M_3)$, let $\A_r$ be the associated ribbon category.
    \item For each $s \in \pi_0(M_2)$, let $\A_{s_0}$ (resp. $\A_{s_1}$) be the ribbon category in the direction of its inward (resp. outward) normal, and let $\D_s$ be the associated central $(\A_{s_1},\A_{s_0})$-algebra.

    \item For each $t \in \pi_0(M_1)$, which is in the closure of 2D strata labelled $\D_{t_1},\ldots,\D_{t_n}$ and 3D strata labelled $\A_{t_1},\ldots, \A_{t_n}$, let $\L_t$ be the associated pointed $\K$-linear category. Here the $\D$ and $\A$ are indexed in clockwise order relative to the orientation on $t$, as in Figure \ref{fig:generic-junction}. Suppose for simplicity that the outward normals are facing in the counter-clockwise direction, if not replace the relevant $\D$ by $\D^\vee$.  
      The category $\L_t$ is endowed with the structure of a $\tr(\D_{t_1}\boxtimes_{\A_{t_2}} \cdots \,\boxtimes_{\A_{t_{n}}} \D_{t_n})$-module.
  \end{list}
  We allow the possibility of non-separating defects, i.e. where $s_0 = s_1$ or $t_{i} = t_{i-1}$. 

  We define a \emph{\bf decorated surface} $\Sigma\to \{0\geq 1 \geq 2\}$ analogously, associating an $\A_r$ to each $r \in \pi_0(\Sigma_2)$, a $\D_s$ to each $s \in \pi_0(\Sigma_1)$, and an $\L_t$ to each $t \in \pi_0(\Sigma_0)$. We allow $\Sigma_1$ to intersect $\partial \Sigma$ transversely.
  The boundary of a decorated manifold is a decorated surface.

  We define a \emph{\bf decorated embedding} as an embedding of stratified manifolds that additionally respects the assigned algebraic data.
\end{definition}
A warning - our definition of decorated spaces is precisely what's needed for the setup in this paper. It is less general than it could be from the viewpoint of stratified factorisation homology \cite{AFT2017,AFT2017local}.
On the other hand, in allowing codimension-2 defects, this definition is more general that what appears under the same name in \cite{Brown_Jordan_2025,JLSS2021}. 

Our goal now is to extend skein relations to decorated ribbon graphs in 3-manifolds with defects in codimensions-1 and 2.
By ``extend'', we mean that these relations will be compatible with those given in \cite{Brown_Jordan_2025} as well as the bulk relations defined for the non-defect theory in \cite{RT1990}. These relations will be between certain decorated ribbon graphs in decorated 3-manifolds. The following is a generalisation of \cite[Def. 2.10]{Brown_Jordan_2025}.
\begin{definition}\label{def:decorated-ribbon-graph}
  Fix a decorated 3-manifold $M$. A \emph{\bf decorated ribbon graph $\Gamma$ embedded in $M$} is a graph $\Gamma$ together with the following data:
  \begin{description}
    \item[A stratified embedding] $\iota : (\Gamma,\{0 \geq 1\}) \to (M,\{1 \geq 2 \geq 3\})$, where $\Gamma_1$ and $\Gamma_0$ are the edges and vertices of $\Gamma$, respectively.

    Vertices in the interior $M\setminus \partial M$ will be called \emph{\bf coupons}, and those in the boundary will be call \emph{\bf endpoints}. We require that endpoints are univalent, with their unique edge meeting the boundary transversely.
  
    \item[A stratified framing] on $\iota(\Gamma)$, i.e. a non-vanishing section of the restricted normal bundle $N(\iota{e}) \cap TM_k$ for each edge $e$ landing in $M_k$ for $k \geq 1$, together with an orientation when $k \geq 2$, and a cyclic ordering on the half edges at each vertex.

    \item[An object for each edge] For each edge $e \in \pi_0(\Gamma_1)$, an assigned object $V_e$ of the category $\mathcal{E}_e$ associated to the region containing $\iota(e)$.  
    \item[A morphism at bulk coupons] For $\iota(v)$ contained in $M_3$, let $h(v)$ denote the cyclically ordered half-edges incident to $v$, and for each $e \in h(v)$ let $V_e^\epsilon \in \mathrm{Obj}(\mathcal{E}_e)$ be $V_e$ if $e$ is incoming and $V_e^*$ if $e$ is outgoing. Then $v$ is decorated by a morphism
      \begin{equation}
        f : \bigotimes_{e \in h(v)} V_e^\epsilon \to \idty_{\mathcal{E}_e}
      \end{equation}
      in $\E_e$. Here we've used that $\mathcal{E}_{e} = \mathcal{E}_{e'}$ for all $e,e' \in h(v)$, i.e. that the half edges are all in a single region.

    \item[A morphism at codimension-1 defect coupons]
      For $\iota(v)$ in some region  $s \in \pi_0(M_2)$, let $h(v)$ be the incident half edges contained in $s$, and $h_{0}(v),h_{1}(v)$ those contained in $s_0,s_1 \in \pi_0(M_3)$. (Recalling from Definition \ref{def:decorated-manifold} that the outward normal of $s$ points towards the 3-dimensional region $s_1$ and away from the region $s_0$.) 
      Let $F : \A_{s_1}\boxtimes \A_{s_0} \to \D_s$ be the functor determined by the centered bimodule structure on $\D_s$.
      Then $v$ is labelled by a morphism in $\D_s$
  \begin{equation}
    f : F\left(\bigotimes_{e \in h_{1}(v)} V_e^\epsilon \boxtimes \bigotimes_{e \in h_{0}(v)} V_e^\epsilon \right) \otimes \bigotimes_{e \in h(v)} V_e^\epsilon \to \idty_{\D_s},
  \end{equation}
  where as before $V_e^\epsilon$ is $V_e$ or $V_e^*$, depending on whether $e$ is incoming or outgoing.

\item[A morphism at codimension-2 defect coupons]
      Suppose $\iota(v)$ is in a codimension-2 defect labelled by $\L$, adjacent to codimension-1 defects labelled by $\D_1,
      \ldots \D_n$. Let $F_i : \A_i \to \D_i$, $i = 1,\ldots,n$ denote the functors induced by the centered bimodule structures.

      Let $X_{0}$, $X_{1} \in \L$ be the objects colouring the incident edges on either side of $\iota(v)$ inside the codimension-2 defect, with $X_0$ before the vertex and $X_1$ after it.
      Recall here that the codimension-2 defect itself is oriented, but that the edges of $\Gamma$ inside the codimension-2 defect are not.
      If $\Gamma$ has no edge in one of these directions, then let the corresponding $X_i$ be the distinguished object $\Dist_\L$. 

      Let $V_1,\ldots,V_n$ be the labels for the incident edges in the adjacent codimension-1 defects, with $V_i = \idty_{\D_{i}}$ if there is no such edge.
      Let $V_i^\epsilon = V_i$ if the edge is incoming and $V_i^*$ if it's outgoing.
    If there's more than one edge, replace the corresponding $V_i$ by the tensor product of the incident edges, and take $\epsilon$ as a multi-index.

    Similarly, let $W_1^\epsilon,\ldots,W_n^\epsilon$ be the labels of the half edges in the bulk regions decorated by ribbon categories $\A_1,\ldots, \A_n$, with as before $W_i = \idty_{\A_i}$ if there is no such edge and a product if there are multiple edges. 

      Then $v$ is labelled by some morphism in $\L$:
      \begin{equation}\label{eq:line-coupon}
        f : (F_1(W_1^\epsilon)\otimes V_1^\epsilon\boxtimes\cdots\boxtimes F_n(W_n^\epsilon)\otimes V_n^\epsilon) \triangleright X_0\longrightarrow X_1
      \end{equation}
  \end{description}
\end{definition}

Restricting a decorated ribbon graph to the boundary $\partial M$ determines a finite configuration of framed decorated points called a \emph{labelling}.
These are a fundamental object in skein theory.

\begin{definition}\label{def:labelling}
  A \emph{\bf labelling} $X$ of a decorated surface $\Sigma$ is a finite collection of points in $\Sigma$ with the following data:\\[-2em]
  \begin{itemize}
      \item framing within the ambient strata for each point, e.g. points in $\Sigma_1$ have two choices of framing while those in $\Sigma_0$ have no framing data;
      \item a sign for points in $\Sigma_2$ and $\Sigma_1$;
      \item a decorating object for each point, which comes from the category associated with its region in $\Sigma$.
  \end{itemize}
\end{definition}

\subsection{Skein relations near a codimension-2 defect}\label{sec:evaluation}

We start by considering a planar theory with codimension-2 defects, before defining the evaluation functor $\Rib^n_\L \to \L$ for the full 3-dimensional defect theory.

\subsubsection{The planar defect theory}\label{sec:planar-defect-theory}

Pivotal categories have a graphical calculus based on planar framed graphs. The resulting planar skein theory admits codimension-1 defects determined by bimodules.
In this section we will give an overview of this theory, extending the non-defect framework developed in \cite{GPV2012}.

\begin{definition}\label{def:planar-graph-cat}
  Let $\D$ be a pivotal category. $\Planar_\D$ is the monoidal \emph{\bf category of planar graphs coloured by $\D$}. It has:
  \begin{description}
    \item[objects:]
      Finite collections of points in $\RR$, each labelled by some sign $\epsilon = \pm 1$ and object $V \in \D$.
      We will denote a object by its set of labels $\{(V_i,\epsilon_i)\}_{i}.$
    \item[morphisms:]
      a morphism from $\{(V_i,\epsilon_i)\}$ to $\{(W_j,\epsilon_j)\}$ is an oriented planar graph in $\mathbb{R}\times[0,1]$,
      with edges coloured by objects in $\D$ and coupons (i.e. vertices) labelled by morphisms from the tensor product of incoming edges to that of outgoing edges.
      The orientation and colour of edges at the boundary $\RR \times \{0,1\}$ must agree with the colours and signs of the source and target objects.
      These graphs are considered up to isotopy.

      We will depict morphisms as going upward, from $\RR\times \{0\}$ to $\RR \times \{1\}$. Composition is given by vertical stacking followed by a rescaling $\mathbb{R} \times [0,2] \overset{\cong}{\to} \mathbb{R} \times [0,1]$.
    \item[product:]
      the monoidal structure is given by disjoint union, so that $T_1 \otimes T_2$ is $T_1$ placed to the left of $T_2$. On objects, it is given by concatenating sequences. The empty sequence is the monoidal unit.

    \item[duality:]
      the dual of $\eta = ((V_1, \varepsilon_1),\ldots,(V_n,\varepsilon_n))$, is $\eta^* := ((V_n, -\varepsilon_n),\ldots,(V_1,-\varepsilon_1))$. 
      Coevaluation and evaluation are given by oriented cups and caps. The pivotal structure is given by the identity isomorphism $(V,+)^{**}=(V,-)^*=(V,+).$
  \end{description}
\end{definition}

There is a canonical pivotal functor $\Planar_\D \to \D,$ defining planar skein relations for pivotal categories \cite{GPV2012}.
We can extend this to allow codimension-2 defects as follows:

\begin{figure}
  \centering
  \includegraphics{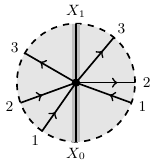}\hspace{1cm}
  \includegraphics{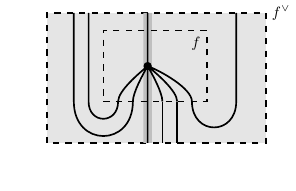}
  \caption{{\bf Left:} The ``lowest first'' ordering on the incident edges of a defect coupon, induced  by orientation of the codimension-2 defect (here pointed upwards) and used in Definition \ref{def:planar-defect-cat}.
  {\bf Right:} The relationship between $f$ and $f^\vee$, appearing in part {\bf (M)} of Proposition \ref{prop:planar-defect-evaluation}, shown here without orientations or labels.}\label{fig:planar-defect-vertex-ordering}
\end{figure}

\begin{definition}[$\Planar_\L$]\label{def:planar-defect-cat}
Let $\L$ be a $\K$-linear $(\D_1,\D_2)$-bimodule with distinguished object $\Dist$.
$\Planar_\L$ is the $\K$-linear category defined as follows:
\begin{description}
  \item[objects:] Triples $( X_1, \ell,X_2)$, or $(X_1,\varnothing,X_2)$, where $X_i$ is an object of $\Planar_{\D_i}$ and $\ell$ is an object of $\L$. We identify $(X_1,\varnothing,X_2)$ and $(X_1,\Dist,X_2)$.
  \item[morphisms:]
    A morphism from $(X_1,\ell,X_2)$ to $(Y_1,m,Y_2)$ is a planar graph $\Gamma$ with a stratified embedding $\iota : \Gamma \hookrightarrow \RR\times[0,1]$, where the stratification on $\RR$ is given by $\{0\} \subset \RR$. This is given a stratified framing, i.e the choice of an orientation for each edge in the two-dimensional region, and no additional data in the one-dimensional region.

    The colours and orientations of the graph must agree with those of $(X_1,\ell,X_2)$ on $\mathbb{R}\times \{0\}$ and those of $(Y_1,m,Y_2)$ on $\mathbb{R}\times \{1\}$. 

    Vertices in either planar region are labelled according to Definition \ref{def:planar-graph-cat}.
    Suppose $v$ is a vertex in the defect, with incident half edges $h_1(v)$ in the $\D_1$ region and $h_2(v)$ in the $\D_2$ region. Then $\iota(v)$ is labelled by some morphism in $\L$:
    \begin{equation}
      f: \left(\bigotimes_{e \in h_0(v)} \hspace{-1ex} V^\epsilon_e \right)\triangleright X_0 \triangleleft \left(\bigotimes_{e \in h_1(v)} \hspace{-1ex}V^\epsilon_e\right) \to X_1.
    \end{equation}
    The linear orders on $h_0(v),h_1(v)$ are induced by the orientation on the codimension-2 defect, i.e. the lowest edges become the innermost. See Figure \ref{fig:planar-defect-vertex-ordering}.
    As elsewhere, $V^\epsilon_e =V_e$ if $e$ is directed into $v$, and $V^*_e$ if its directed out of $v$. 

    Graphs are considered up to stratified isotopy, and composition is given by vertical stacking and rescaling.
\end{description}
\end{definition}

We construct an evaluation functor $\Planar_\L \to \L$ as follows.
\begin{proposition}\label{prop:planar-defect-evaluation}
The following procedure defines a linear functor $\ev : \Planar_\L \to \L$. On objects, it's given by $(X_1,\ell,X_2) \mapsto \ev_1(X_1) \triangleright \ell \triangleleft \ev_2(X_2)$, where $\ev_1,\ev_2$ are the evaluation functors for the planar diagrammatics of the pivotal categories $\D_1,\D_2$.  
Let $\Gamma$ be a morphism in $\Planar_\L$. We construct $\ev(\Gamma)$ as follows. 
\begin{enumerate}
  \item Fix a height function on the stratified space containing $\Gamma$, which is strictly increasing along the defect. Place $\Gamma$ in generic position with respect to this height function, i.e. all vertices and critical points have distinct heights, and the height function is Morse on all edges of $\Gamma$.
  \item Factor $\Gamma = \Gamma_n\Gamma_{n-1} \cdots \Gamma_1$ along the height function so that each $\Gamma_i$ is in either \emph{\bf(M)} a small strip containing a single `mixed vertex' with edges both the planar and the linear regions and no critical points or \emph{\bf(I)} a strip with no such vertices, i.e. no interaction between the planar and linear regions.
  \item Evaluate each $\Gamma_i$ according to its type:
    \begin{description}
      \item[(M)] here $\Gamma_i$ is a combination of a single mixed vertex labelled by some morphism 
        \begin{equation}
          f: V_n^\epsilon\otimes\cdots\otimes V_1^\epsilon \triangleright X_0 \triangleleft W_1^\epsilon\otimes \cdots \otimes W_m^\epsilon \to X_1 
        \end{equation}
        in $\L$ surrounded by straight lines labelled $x_1,\ldots,x_k$ and $y_1\ldots,y_m$ in the $\D_1$ and $\D_2$ regions, respectively. 
        First, we use the duality in $\D_1,\D_2$ to modify $f$ such that edges coming from above act on the target instead of the source, see Figure \ref{fig:planar-defect-vertex-ordering}.
        Concretely, suppose that the edges labelled $V_n,..., V_{k}, W_m,...,W_\ell$ connect the vertex to the upper boundary of the strip. Then we replace $f$ by
        \begin{equation}
          f^\vee: V_{k-1}^\epsilon\otimes\cdots\otimes V_1^\epsilon \triangleright X_0 \triangleleft W_1^\epsilon\otimes \cdots \otimes W_{\ell-1}^\epsilon \to \left(V_{m}^\epsilon\otimes\cdots\otimes V_k^\epsilon\right)^* \triangleright X_1 \triangleleft \left(W_\ell^\epsilon\otimes \cdots \otimes W_{m}^\epsilon\right)^* 
        \end{equation}

        Finally, evaluate such graphs as $\id_{x_1^\epsilon\otimes\cdots\otimes x_k^\epsilon} \triangleright f^\vee \triangleleft \id_{y_1^\epsilon\otimes \cdots \otimes y_m^\epsilon}$.
      \item[(I)] here $\Gamma_i$ decomposes into three isolated graphs $\Gamma_{i,1}$ in the $\D_1$ region, $\Gamma_{i,l}$  in the $\L$ region, and $\Gamma_{i,2}$ in the $\D_2$ region. Send such a graph to $\ev_1(\Gamma_{i,1}) \triangleright \ev_{lin}(\Gamma_{i,l}) \triangleleft \ev_2(\Gamma_{i,2})$. Here $\ev_{lin}$ the is evaluation functor for the 1-dimensional diagrammatics of the linear category $\L$.
  \end{description}
\item Finally, compose the individual pieces $\ev(\Gamma) := \ev(\Gamma_n)\cdots \ev(\Gamma_1)$ and extend linearly to the full hom-spaces of $\Planar_\L$. 
\end{enumerate}
\end{proposition}

\begin{figure}
  \centering
  \includegraphics{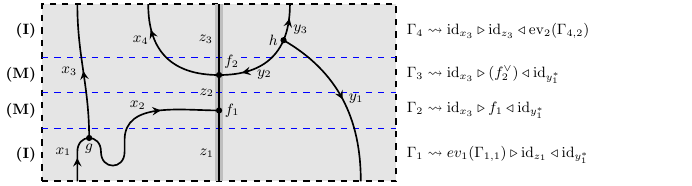}
  \vspace{1em}

  \includegraphics{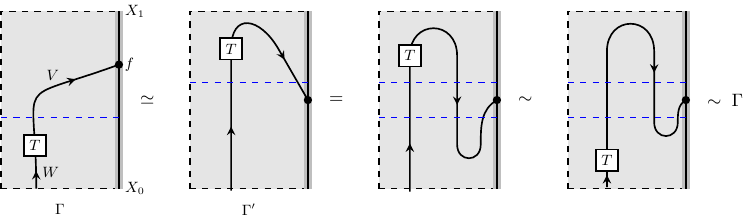}
  \caption{\textbf{Top:} Evaluation for the planar theory with defects is constructed relative to a decomposition of the graphs. \textbf{Bottom:} Isotopy invariance is a consequence of functoriality of the action and rigidity in the planar coefficient categories, because an isotopy passing a planar coupon above a mixed one introduces a cancelling cap/cup pair.}\label{fig:planar-defect-decomp}
\end{figure}

\begin{proof}
   We show that $\ev : \Planar_\L \to \L$ is invariant under stratified isotopy and the choice of decomposition $\Gamma = \Gamma_n \cdots \Gamma_1$. 
First, each slice of type (M) or (I) is isotopy invariant, for type (M) this is by construction and for type (I) this is a direct consequence of coherence of planar diagrammatics for a pivotal category and 1-dimensional diagrammatics of a linear category.

Next, suppose an isotopy $\Gamma \sim \Gamma'$ changes the decomposition into mixed and isolated regions by swapping the order of some planar vertex labelled $T$ and some mixed vertex labelled $f$, as in Figure \ref{fig:planar-defect-decomp}. Without loss of generality, we can consider just one planar region.

The two graphs evaluate as follows: $\ev(\Gamma) = f\circ\left(\ev_1(T)\triangleright \id_{X_0}\right)$, while
\begin{equation}
\begin{aligned}
  \ev(\Gamma') &= [(d_V\circ (\ev(T)\otimes \id_{V^*})) \triangleright \id_{X_1}] \circ (\id_{W}\triangleright f^\vee) \\
  &= [(d_V\circ(\ev(T)\otimes\id_{V^*}))\triangleright\id_{X_1}]\circ[(\id_W\otimes\id_{V^*})\triangleright f]\circ[(\id_{W}\otimes b_V)\triangleright\id_{X_0}] \\
  &= (d_V\triangleright\id_{X_1})\circ[(\ev(T)\otimes\id_{V^*})\triangleright\id_{X_1}] \circ [(\id_W\otimes\id_{V^*})\triangleright f]\circ[(\id_{W}\otimes b_V)\triangleright\id_{X_0}] \\
  &= (d_V\triangleright\id_{X_1})\circ[(\id_W\otimes\id_{V^*})\triangleright f]\circ[(\ev(T)\otimes b_V)\triangleright\id_{X_0}] \\
  &= f\circ [(d_V\otimes\id_{V^*})\triangleright\id_{X_0}]\circ[(\id_V\otimes b_V)\triangleright \id_{X_0}]\circ(\ev(T)\triangleright\id_{X_0}) \\
  &= f\circ\left(\ev_1(T)\triangleright \id_{X_0}\right)\\
  &=\ev(\Gamma),
\end{aligned}
\end{equation}
where $d_V : V\otimes V^* \to \idty $ and $b_V : \idty \to V^* \otimes V$ are the duality morphisms for $V$. Here we've used the definition of $f^\vee$ in the second equality, the zig-zag identity $(d_V\otimes \id_{V} )(\id_V\otimes b_V) = \id_V$ in the penultimate equality and the functoriality of the action in the rest (see Figure \ref{fig:planar-defect-decomp}).
\end{proof}

\begin{remark}
    The category $\Planar_\L$ is the free pointed bimodule category on $\L$. It has a canonical $(\Planar_{\D_1},\Planar_{\D_2})$-bimodule structure given by vertically stacking diagrams to the left and right, and these actions commute with the evaluation functors, up to conjugation by the coherences. The distinguished object is $(\emptyset,\emptyset,\emptyset)$.
\end{remark}

\subsubsection{Extension to three dimensions}\label{sec:3d-skein-relations}
We can now use labellings and decorated ribbon graphs to define the diagrammatic category for a codimension-2 defect in three dimensions.

\begin{definition}[$\Rib^n_\L$]\label{def:ribL}
  Let $\DD_n$ be the decorated surface shown on the left hand side of Figure \ref{fig:generic-junction} with codimension-2 defect $\L$. 
  The $\K$-linear category $\Rib_\L^n$ is defined as follows:\\[-2em]
  \begin{description}
    \item[objects:] Labellings of $\DD_n$.
    \item[morphisms:] The $\K$-vector space $\Hom_{\Rib_\L^n}(X,Y)$ is spanned by decorated ribbon graphs embedded in $\DD_n\times[0,1]$, taken up to stratified isotopy relative to the boundary, which restrict to $X$ on $\DD_n\times\{0\}$ and $Y$ on $\DD_n\times\{1\}$.
   
  \end{description}
\end{definition}

We now have all the ingredients necessary to define the evaluation functor $\Rib_\L^n \to \L$ whose kernel will determine skein relations near codimension-2 defects.
Up to replacing some of the $\D_i$ by their duals, we may suppose that the inward normal to every defect points clockwise.

\begin{theorem}\label{thm:skein-evaluation}
The following defines a linear functor $\RT^n_\L : \Rib_\L^n \to \L$. For a decorated ribbon graph, we define $\RT^n_\L(\Gamma)$ in three steps:
\begin{description}
    \item[Step 1: Retract.]
    Fix a deformation retract of $\DD_n\times[0,1]$ onto the union of its defects, $p:\DD_n\times [0,1] \to C(\{1,\ldots,n\})\times[0,1]$. (Recall that $C(X)$ denotes the cone on $X$.)
    Let $\Gamma\hookrightarrow \DD_n\times[0,1]$ be a decorated ribbon graph, not taken up to isotopy, which is in generic position with respect to $p$.
    That is: in the retract of $\Gamma$ along $p$ there are no triple points, double points correspond to transverse intersections and never lie on the defect line, pairs of vertices are neither identified nor sent to the defect line. 
    
    Retract the ribbon graph along $p$ onto the union of defects.
    Consider the image $p(\Gamma)$ as a graph in $C(\{1,\ldots,n\})\times [0,1]$ by treating each introduced double point and each intersection with the codimension-2 defect as a new vertex. Let $p(\Gamma)_1$ and $p(\Gamma)_0$ denote the new edges and vertices, respectively.

    \item[Step 2: Update decorations.]
    Decorate any edge of $p(\Gamma)_1$ which has been pushed from the bulk into the defect by $H_i(V_e)$, where $V_e \in \A_i$ is the corresponding decoration of $\Gamma$. Do the same for vertices coming from the bulk.
    By construction vertices which were originally in the one of the defects do not have to be updated, even if they involved bulk edges.
    
    Any vertex coming from a double-point involves at least one strand coming from some bulk region $\A_i$ and can therefore be decorated by the associated half-braiding in $Z_{Dr}(\D_i)$, as in \cite{Brown_Jordan_2025}.
    
    Finally, a vertex introduced by the intersection of an edge with the codimension-2 defect is labelled by the central structure $\eta_{V_e,\ell} : H_i(V_e)\triangleright \ell \to \ell \triangleleft H_{i+1}(V_e)$ or its inverse, depending on whether the crossing is positive or negative.
    
    \item[Step 3: Separate vertices then evaluate.] Apply an isotopy so that each vertex of $p(\Gamma)$ lies at a different height, even those on distinct surfaces. Planar evaluation with defects extends from 2 to $n$ adjacent codimension-1 defects because the actions of each of the defect walls factors commute. 
    
    Finish by evaluating $p(\Gamma)$ following the same rules as in Proposition \ref{prop:planar-defect-evaluation}, resulting in a morphism of $\L$.
\end{description}

Note that $\RT_\L(\Gamma)$ is a morphism in $\L$, and therefore this construction extends canonically to the objects of $\Rib^n_\L$.
\end{theorem}

\begin{definition}\label{def:skein-equivalent}
    We will say $\Gamma_1$ and $\Gamma_2$ are \emph{\bf skein equivalent} and write $\Gamma_1\sim\Gamma_2$ whenever $\RT_\L(\Gamma_1) = \RT_\L(\Gamma_2)$ for some choice of isotopy representatives and retract. 
\end{definition}

We will break the proof of Theorem \ref{thm:skein-evaluation} into a series of arguments involving the diagrammatics of centered bimodule categories.

\begin{lemma}\label{lemma:bulk-and-planar-isotopy-inv1}
    With the same notations as above, the following identities hold:
    \begin{enumerate}
        \item \textbf{Left half-braiding:} 
        \begin{equation}\label{eq:graphic_left_half-braiding}
            \includegraphics[valign=c]{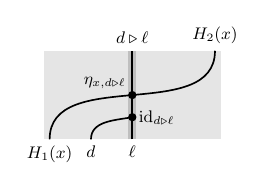}\sim\includegraphics[valign=c]{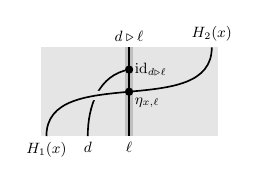}.
        \end{equation}
        \item \textbf{Right half-braiding:} 
        \begin{equation}\label{eq:graphic_right_half-braiding}
            \includegraphics[valign=c]{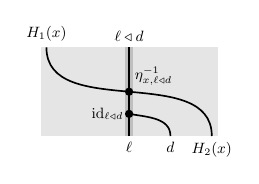}\sim\includegraphics[valign=c]{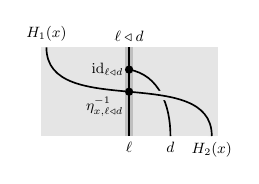}.
        \end{equation}
        \item \textbf{Monoidal structure:}
        \begin{equation}
            \includegraphics[valign=c]{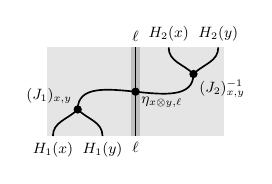}\sim\includegraphics[valign=c]{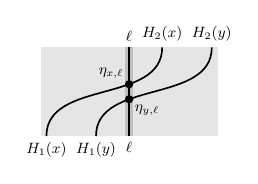}.
        \end{equation}
        \item \textbf{Naturality in the first variable:} 
        \begin{equation}
            \includegraphics[valign=c]{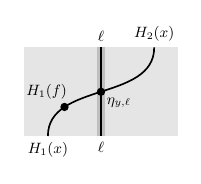}\sim\includegraphics[valign=c]{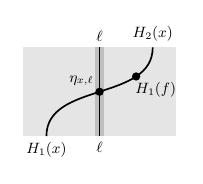}.
        \end{equation}
        \item \textbf{Naturality in the second variable:}
        \begin{equation}\label{eq:naturality_second_variable}
            \includegraphics[valign=c]{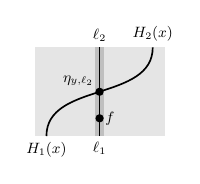}\sim\includegraphics[valign=c]{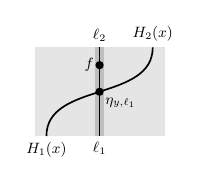}.
        \end{equation}
    \end{enumerate}
\end{lemma}
\begin{proof}
    Applying the evaluation process defined above to each of the identities yields the corresponding commutative diagram in $\mathcal{L}$ (cf. Definition \ref{def:centered-bimodule}).
\end{proof}

We also need the following version of the Reidemeister moves II and III with one of the strands running along the codimension-2 defect:

\begin{lemma}\label{lemma:bulk-and-planar-isotopy-inv2}
    The evaluation process defined above is invariant under the following version of Reidemeister moves II and III:
    \begin{enumerate}
        \item \textbf{Defect Reidemeister move II:}
        \begin{equation}
            \includegraphics[valign=c]{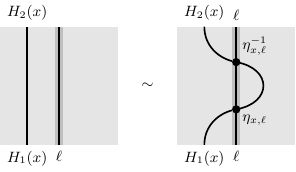}.
        \end{equation}
        \item \textbf{Defect Reidemeister move III:}
        \begin{equation}
          \includegraphics[valign=c]{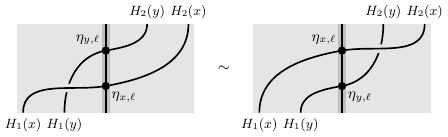}.
        \end{equation}
    \end{enumerate}
\end{lemma}

\begin{proof}
  The first claim follows from invertibility of $\eta_{x,\ell}$.
  To pass a braiding through the defect, we use the monoidal structure to reduce to naturality as follows:
    \begin{multline*}
            \includegraphics[valign=c]{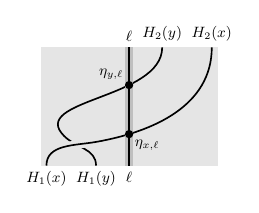}\sim\includegraphics[valign=c]{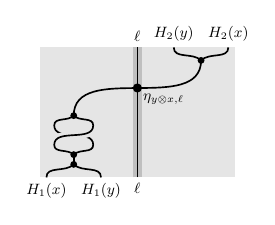}\\ \sim\includegraphics[valign=c]{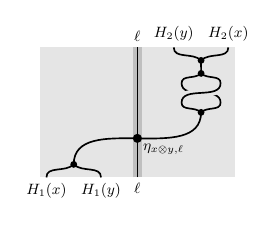}\sim\includegraphics[valign=c]{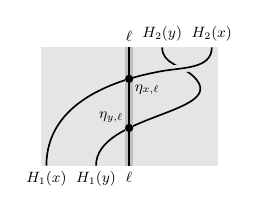},
        \end{multline*}
        where the first equality is the compatibility of $\eta$ with the monoidal structure, the second and the last are planar (non-stratified) skein relations and the third is a consequence of naturality.
\end{proof}

\begin{corollary}\label{cor:bulk-and-planar-isotopy-inv3}
    The following equalities hold:
    \begin{enumerate}
        \item \textbf{Left half-braiding 2.0:}
        \begin{equation}
            \includegraphics[valign=c]{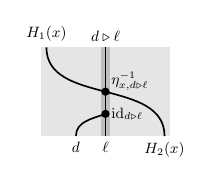}\sim\includegraphics[valign=c]{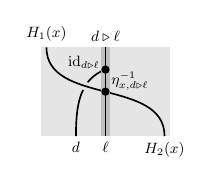}
        \end{equation}
        \item \textbf{Right half-braiding 2.0:}
        \begin{equation}\label{eq:graphic_right_half-braiding2.0}
             \includegraphics[valign=c]{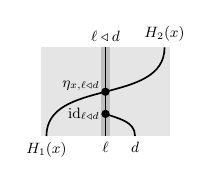}\sim\includegraphics[valign=c]{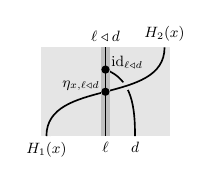}
        \end{equation}
    \end{enumerate}
\end{corollary}

\begin{proof}
    This follows by pre- and post-composing the equalities \eqref{eq:graphic_left_half-braiding} and \eqref{eq:graphic_right_half-braiding} and applying Reidemeister move II.
\end{proof}

\begin{lemma}\label{lemma:bulk-and-planar-isotopy-inv4}
  The morphism $\eta_{x,\ell}$ is natural in $\L$, so that
  \begin{equation}\label{eq:naturality-of-eta}
        \includegraphics[valign=c]{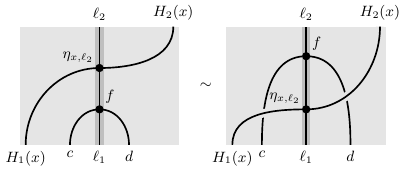}
    \end{equation}
\end{lemma}

    Note that \eqref{eq:naturality-of-eta} holds regardless of the orientation of the bulk defect.

\begin{proof}
    By the construction of the evaluation morphism, we can factor the coupon labelled $f$ as follows:
    \begin{equation}
        \includegraphics{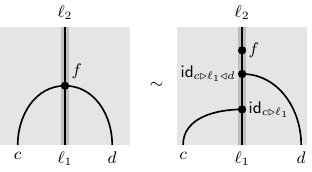}
    \end{equation}
    The result then follows from \eqref{eq:naturality_second_variable}, \eqref{eq:graphic_left_half-braiding} and \eqref{eq:graphic_right_half-braiding2.0}.
\end{proof}

\begin{proof}[Proof of Theorem \ref{thm:skein-evaluation}.]
    First note that Remark \ref{rmk:orientation_and_evaluation_morphism} implies that replacing a codimension-1 defect by its dual and switching orientations does not affect the evaluation, so there is no loss of generality in assuming that the outward normal vectors of all surface point counter-clockwise. Next, we show that the evaluation is invariant under isotopy and choice of retract.

    By Proposition \ref{prop:planar-defect-evaluation} and \cite[prop 2.6]{Brown_Jordan_2025}, $\RT_\L$ is invariant under isotopies supported in a ball whose retract does not intersect the codimension-2 defect.
    Invariance under changes impacting the codimension-2 defect follows directly from Lemmas \ref{lemma:bulk-and-planar-isotopy-inv1}, \ref{lemma:bulk-and-planar-isotopy-inv2}, \ref{lemma:bulk-and-planar-isotopy-inv4} and Corollary \ref{cor:bulk-and-planar-isotopy-inv3}.

\end{proof}

\begin{remark}[Point defects]\label{rem:point-defects}
  Point defects in 3-dimensional skein theory are fundamentally different than line and surface defects.
Topologically, this reflects two facts: points and lines do not link non-trivially in three dimensions, and point defects are not the thickening of some defect in two dimensions.
Algebraically, we run into the fact that codimension-3 defects should be coloured by 3-morphisms in $\Alg_2(\Cat_\K)$, which are functors.
A single functor does not have graphical calculus in the way that ribbon, pivotal, and even linear categories do.
Therefore we've lost the basis for the evaluation-based skein theoretic approach which works so well in lower codimensions.
\end{remark}

\subsection{Skein modules and categories}\label{sec:skein-modules-and-categories}
The skein relations detailed in Section \ref{sec:evaluation} allow us to define skein modules, categories, and functors built from decorated 3-manifolds.
This closely follows the treatments in \cite{GJS2023,Brown_Jordan_2025}.

\begin{definition}\label{def:relative-defect-skein-module}
Fix a decorated $3$-manifold $M$, possibly with boundary, and a possibly empty labelling $X$ on $\partial M$.
For each point $p \in M$, let $D_p$ denote the decorated disk such that $D_p\times [0,1]$ gives a decorated neighbourhood of $p$, and let $\RT_p$ denote the corresponding evaluation functor.

    The \emph{\bf relative defect skein module} $\Sk(M,X)$ is the $\K$-module spanned by isotopy classes of stratified ribbon graphs in $M$ compatible with $X$, modulo skein relations for every decorated embedding $\D_p \times [0,1] \hookrightarrow M$, induced by $\RT_p$.

    In the case where $X = \varnothing$, we write $\Sk(M)$ and call it the \textbf{defect skein module}.
\end{definition}

Relative skein modules provide the hom-spaces for the skein category of a decorated surface.
\begin{definition}\label{def:defect-skein-cat}
    Let $\Sigma$ denote a decorated surface. The \emph{\bf{defect skein category}} $\SkCat(\Sigma)$ is the $\K$-linear category with 
    \begin{description}
        \item[objects:] Labellings of $\Sigma$.
        \item[morphisms:] The vector space of homomorphisms from $X$ to $Y$ is given by the relative defect skein module
        \begin{align}
            \Hom(X,Y) := \Sk(\Sigma \times [0,1],\overline{X} \sqcup Y),
        \end{align}
        where $\overline{X}$ denotes a labelling $X$ of $\Sigma \times \{0\}$ with all orientations reversed and $Y$ is a labelling of $\Sigma \times \{1\}$. The composition is induced by $\Sigma\times[0,1] \sqcup \Sigma\times[0,1] \hookrightarrow \Sigma \times [0,1]$ and a smoothing at boundary points.
    \end{description}
\end{definition}

As in the bulk theory, the evaluation functor $\Rib_\L \to \L$ induces a categorical equivalence, so that the skein category of the neighbourhood of a defect recovers the original defect data:
\begin{lemma}\label{lemma:skein_cat_of_line_defect_is_L}
    Consider a decorated surface $\DD_n$ with $n\geq 2$ and codimension-2 defect $\L$. We have an equivalence $\SkCat(\DD_n) \cong \L$ of $\SkCat(\partial\DD_n\times[0,1])$-module categories.
\end{lemma}

As in \cite{Brown_Jordan_2025,GJS2023}, decorated cobordisms give rise to bimodules for skein categories.
\begin{definition}\label{def:defect-skein-functor}
  Let $M$ be a decorated 3-manifold and fix a decomposition $\partial M = \overline{\Sigma}_{in}\cup\Sigma_{out}$.
  The \emph{\bf defect skein functor} is the $\K$-linear functor given on objects by
  \begin{equation}
    \begin{aligned}
      \underline{\Sk}(M) : \SkCat(\Sigma_{in})^{op} \times \SkCat(\Sigma_{out}) & \to \Vect \\
      X,Y &\mapsto \Sk(M;\overline{X}\sqcup Y).
    \end{aligned}
  \end{equation}
  On morphisms it's given by the linear maps
  \begin{equation}
    \Sk(\Sigma_{in}\times[0,1];\overline{X}_1\sqcup\overline{X}_2) \times \Sk(M;\overline{X}_2\sqcup Y_1) \times \Sk(\Sigma_{out}; Y_1\sqcup Y_2) \to \Sk(M;\overline{X}_1\sqcup Y_2)
  \end{equation}
  induced by gluing skeins along their shared endpoints.
\end{definition}

\subsubsection{Internal skein algebras and modules}\label{sec:internal-skein-constructions}

Skein categories are functorial with respect to decorated embeddings and their isotopies.
This functoriality induces additional algebraic structures on skein categories, powering both the constructions of internal skein algebras and modules \cite{BBJ2018Integrating}.
We state the definitions, but refer to \cite{BBJ2018Integrating,GJS2023,Brown_Jordan_2025} for more in depth discussions. 

Let $\Sigma$ be a decorated surface with non-empty boundary and denote by $\G \subset \partial \Sigma$ a collection of disjoint boundary intervals, called \emph{\bf gates}, disjoint from any defects. For brevity's sake, let $\mathcal{S}_\G := \SkCat(\G\times [0,1])$.
The inclusion $\G \times [0,1] \sqcup \Sigma \hookrightarrow \Sigma$ makes $\SkCat(\Sigma)$ into an $\mathcal{S}_\G$-module category. This action restricted to the empty skein is called the \emph{\bf disk insertion functor}
\begin{align}
    \P : \mathcal{S}_\G \to \SkCat(\Sigma).
\end{align}

After free completion $\widehat{\P}$ admits a right adjoint $\widehat{\P}^R$, the monad of this adjunction provides an algebra that tracks all possible labellings along $\G$.
\begin{definition}
    The \emph{\bf{internal defect skein algebra}} is the algebra object 
    \begin{align}
      \SkAlg^{int}_\G(\Sigma) := \widehat{\P}^R\widehat{\P}(\widehat{\1}) \in \widehat{\mathcal{S}}_\G.
    \end{align}
\end{definition}
Note that internal skein algebras are exactly the internal endomorphism algebras of the empty labelling for the action $\mathcal{S}_\G \times \SkCat(\Sigma) \to \SkCat(\Sigma)$.

The internal skein algebra of $\DD_1$ is commutative.
\begin{lemma}\label{lemma:d1-is-commutative}
  Consider $\DD_1$ decorated by the central $(\A,\C)$-bimodule $\D$.
  The internal skein algebra $\SkAlg^{int}(\DD_1)$ is a commutative algebra object in the free co-completion of $\A\boxtimes \C^{bop}$.
\end{lemma}
Note that this holds in particular for $\D = \A$ the trivial $(\A,\A)$-defect, in which case we get a commutative algebra object in $\A\boxtimes \A^{bop}$.
\begin{proof}
  Let $S := \SkAlg^{int}(\DD_1)$ in this proof, and for brevity's sake suppose that $\A,\C$ are co-complete. The claim is that the following diagram commutes:
  \begin{equation}
    \begin{tikzcd}[sep=small]
      S \otimes S \ar[rr,"\sigma_{S,S}"]\ar[rd,"\mu"'] & & \ar[ld,"\mu"] S \otimes S \\
                                                       & S &
    \end{tikzcd}
  \end{equation}
  where $\mu$ is the internal multiplication on $S$. 
  By \eqref{eq:its-a-coend} and co-continuity of the tensor product, the braiding on the internal skein algebra is induced by the braiding $\sigma_{X,Y} : X\otimes Y \to Y\otimes X$ applied to the terms of the two coends.
  The relation $\mu = \mu\sigma$ follows directly from the coend relation between the $X\otimes Y$ and $Y\otimes X$ summands induced by $\sigma_{X,Y}$. See Figure \ref{fig:d1-is-commutative}, noting that the left most skein is the result of $\mu$ and the right most skein is the result of $\mu\sigma_{S,S}$.

  \begin{figure}
    \centering
    \includegraphics{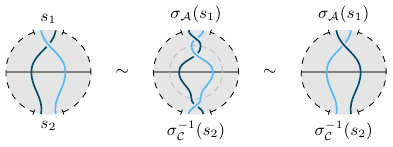}
    \caption{Commutativity of the defect disk, see Lemma \ref{lemma:d1-is-commutative}. The first equivalence is a coend relation and the second is an isotopy.}\label{fig:d1-is-commutative}
  \end{figure}
\end{proof}
Conversely, this implies that defects between pairs of ribbon categories $\A,\C$, can be constructed from commutative algebra objects in their product $\A\boxtimes \C^{bop}$.

The internal construction extends to general decorated three manifolds.
\begin{definition}
Let $M$ be a decorated $3$-manifold with a fixed decomposition $\partial M = \Sigma \cup_C R$. We treat $\partial \Sigma = C$ as the boundary of $\partial M$ and require that labellings are contained in $\Sigma$.  
    The \emph{\bf{defect internal skein module}} $\Sk^{int}_\G(M) \in \widehat{\SkCat(\G \times [0,1])}$ is given by the functor
    \begin{align}
        V \mapsto \Sk(M,\P(V)).
    \end{align}
    Where, as above, $\P$ is the disk insertion functor induced by some embedding $\G \times [0,1] \hookrightarrow \partial \Sigma$. 
\end{definition}

Note that $\Sk^{int}(\Sigma\times[0,1]) \cong \SkAlg^{int}(\Sigma)$ in $\widehat{\mathcal{S}}_\G$.
Both internal skein algebras and modules can be expressed as coends \cite{GJS2023,Brown_Jordan_2025}, leading to a more concrete description in terms of skeins with state data at the gates.
\begin{equation}\label{eq:its-a-coend}
    \Sk^{int}(M) \cong \int^{X \in \mathcal{S}_{\G}} \widehat{X} \otimes \Sk(M;\P(X)).
\end{equation}
We recover traditional (non-internal) skein algebras and modules by evaluating at $\Dist_{\mathcal{S}_\G}$, i.e. $\Sk^{int}(M)(\Dist_{\mathcal{S}_\G}) = \Sk(M)$.

\subsection{The composition junction}\label{sec:composition-junction}

Fix two central algebras $F_1, \sigma_1 : \A \boxtimes \B^{\bop} \to \Zdr(\D_1)$ and $F_2, \sigma_2 : \B \boxtimes \C^{\bop} \to \Zdr(\D_2)$.
We have seen that their composition is the balanced Deligne--Kelly tensor product $\D_1\boxtimes_\B \D_2$. 
In this section we consider composition from a topological perspective.

It is sometimes useful to compose two codimension-1 defects in a small region, and not be restricted to composing only globally parallel defects.
This need for locality leads us to find a canonical codimension-2 defect which forms a three way junction between $\D_1, \D_2$, and $\D_1\boxtimes_\B \D_2$, as shown in Figure \ref{fig:composition-junction}. 

\begin{figure}
  \begin{center}
    \includegraphics{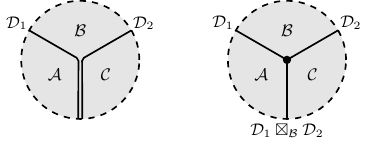}
    \caption{Composition can be modelled by a three-way junction, see Lemma \ref{lemma:composition-defect}.}\label{fig:composition-junction}
  \end{center}
\end{figure}

\begin{lemma}\label{lemma:composition-defect}
  Fix central algebras $\D_1, \D_2$, as above. Their Deligne--Kelly balanced tensor product $\D_1\boxtimes_\B \D_2$ has a canonical\footnote{Note though that the balanced tensor product is defined up to categorical equivalence.} structure of a codimension-2 defect between $\D_1, \D_2$, and their composition $\D_1\boxtimes_\B \D_2$.
\end{lemma}
When it's acting in this capacity, we call $\D_1 \boxtimes_\B \D_2$ the \emph{composition codimension-2 defect}. Treating this as a codimension-2 defect is a matter of useful bookkeeping, since, for example, the two disks of Figure \ref{fig:composition-junction} have equivalence skein categories.

In the following proof we will refer to objects of $\D_1\boxtimes_\B \D_2$ by a representative in $\D_1\boxtimes \D_2$.
\begin{proof}

Following Lemma \ref{lemma:unpacking-line-data}, we will endow $\D_1\boxtimes_\B\D_2$ with the structure of an $\A$-centered $(\D_1, \D_1\boxtimes_\B \D_2)$-bimodule, a $\B$-centered $(\D_1,\D_2)$-bimodule, and a $\C$-centered $(\D_1\boxtimes_\B \D_2,\D_2)$-bimodule.

The bimodule structures are given by the tensor product $F_{cp}$ on $\D_1\boxtimes_\B \D_2$ together with the functors
\begin{equation}
  \begin{aligned}
    F_1 : \D_1 \boxtimes \left( \D_1\boxtimes_\B \D_2\right) &\to \D_1\boxtimes_\B \D_2 &\hspace{2cm} F_2 : \left( \D_1\boxtimes_\B \D_2\right)\boxtimes \D_2 &\to \D_1\boxtimes_\B \D_2\\
    x_1 \boxtimes (y_1 \boxtimes y_2) & \mapsto (x_1\otimes y_1) \boxtimes y_2 &  (y_1\boxtimes y_2) \boxtimes z_2 &\mapsto y_1\boxtimes (y_2\otimes z_2).
  \end{aligned}
\end{equation}

Let $\triangleright_1$, $\triangleright_{cp}$, $\triangleleft_2$, and $\triangleleft_{cp}$ denote the following left and right actions on $D_1\boxtimes_\B D_2$: 
\begin{equation}\label{eq:centering-nat-trans}
  \begin{aligned}
    x_1\triangleright_1 \!(y_1 \boxtimes y_2) &= (x_1\otimes y_1) \boxtimes y_2,  &
    (x_1\boxtimes x_2)\triangleright_{cp} \!(y_1\boxtimes y_2) &= (x_1\boxtimes x_2)\otimes (y_1\boxtimes y_2)
    \\
    (y_1 \boxtimes y_2) \triangleleft_2 z_2 &= y_1 \boxtimes (y_2\otimes z_2), &
    (y_1\boxtimes y_2) \triangleleft_{cp} (z_1\boxtimes z_2) &= (y_1\boxtimes y_2) \otimes (z_1\boxtimes z_2)
\end{aligned}
\end{equation}
To define the balanced structures we must write down natural isomorphisms:
\begin{equation}
  \begin{aligned}\label{eq:the-balancings}
    \mu_{a,\ell} : &F_1(a\boxtimes \idty_B) \triangleright_1 \ell_1\boxtimes \ell_2 \to \ell_1\boxtimes \ell_2 \triangleleft_{cp} F_{cp}(a \boxtimes \idty_C)  &\hspace{1cm} \A-\text{balancing},\\
    \eta_{b,\ell} : &F_1(\idty_A\boxtimes b) \triangleright_1 \ell_1\boxtimes \ell_2 \to \ell_1\boxtimes \ell_2 \triangleleft_2 F_2(b\boxtimes \idty_C)  & \hspace{1cm} \B-\text{balancing},\\
    \nu_{c,\ell} : &F_{cp}(\idty_A \boxtimes c) \triangleright_{cp} \ell_1\boxtimes \ell_2 \to \ell_1\boxtimes \ell_2 \triangleleft_{2} F_{2}(\idty_B \boxtimes c) &\hspace{1cm} \C-\text{balancing}.
  \end{aligned}
\end{equation}
We start with $\mu$, the $\A$-balancing. Its source is
\begin{equation}
  F_1(a\boxtimes \idty_B)\triangleright_1 \ell_1\boxtimes \ell_2 = \left(F_1(a\boxtimes \idty_B) \otimes \ell_1\right)\boxtimes \ell_2.
\end{equation}
Next we use the fact that $F_{cp}(a\boxtimes c) = F_1(a\boxtimes \idty_B) \boxtimes F_2(\idty_B\boxtimes c)$ to rewrite the target as follows
\begin{equation}
  \begin{aligned}
    \ell_1\boxtimes \ell_2 \triangleleft_{cp} F_{cp}(a\boxtimes \idty_C)&= (\ell_1\boxtimes \ell_2) \otimes F_{cp}(a\boxtimes \idty_C) \\
                                                                        &\cong (\ell_1 \otimes F_1(a\boxtimes \idty_B)) \boxtimes (\ell_2 \otimes F_2(\idty_B\boxtimes\idty_C)) \\
                                                                        &\cong (\ell_1 \otimes F_1(a\boxtimes \idty_B) \boxtimes \ell_2
  \end{aligned}
\end{equation}
Recall that $F_1$ comes with a half-braiding $\sigma: F_1(-)\otimes - \Rightarrow -\otimes F_1(-)$ in $\D_1$.
This provides the required natural isomorphism $F_1(a\boxtimes \idty_C)\otimes \ell_1 \boxtimes \ell_2 \to \ell_1 \otimes F_1(a\boxtimes \idty_C)\boxtimes \ell_2$: 
\begin{equation}
  \mu_{a,\ell} = \sigma_{a\boxtimes\idty_C, \ell_1}\boxtimes \id_{\ell_2}.
\end{equation}
The $\C$-balancing is constructed analogously, from the half braiding $F_2(-)\otimes - \Rightarrow - \otimes F_2(-)$, so that $\nu_{c,\ell} = \id_{\ell_1}\boxtimes \sigma_{\idty_B\boxtimes c,\ell_2}$.
Unwrapping the middle line \eqref{eq:the-balancings}, we find that the necessary natural isomorphisms for the $\B$-balancing $\eta$ are given by
\begin{equation}
  \begin{aligned}
    F_1(\idty_A\boxtimes b) \otimes \ell_1 \boxtimes \ell_2 &\cong \ell_1 \otimes F_1(\idty_A \boxtimes b) \boxtimes \ell_2  & & \hspace{2cm}\text{half-braiding on } F_1\\ 
                                                            & \cong \ell_1 \boxtimes F_2(b\boxtimes \idty_C) \otimes \ell_2 && \hspace{2cm}\B\text{-balancing on } \D_1\boxtimes_\B\D_2 \\
                                                            &\cong \ell_1 \boxtimes \ell_2 \otimes F_2(b\boxtimes \idty_C) && \hspace{2cm} \text{half-braiding on } F_2
  \end{aligned}
\end{equation}

\end{proof}

\begin{lemma}\label{lemma:composing-locally}
  Let $\D_1$ be a defect from $\B$ to $\A$, and $\D_2$ from $\C$ to $\B$, let $\D_1\D_2 = \D_1\boxtimes_\B \D_2$ denote their composition.
  We have the following equivalences of categories
\begin{equation}\label{eq:tunnel-composition-cat}
  \includegraphics{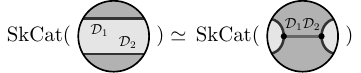}
\end{equation}
and
\begin{equation}\label{eq:bubble-composition-cat}
  \includegraphics{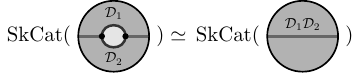},
\end{equation}
where all codimension-2 defects are the composition junction.
\end{lemma}
\begin{proof}
This follows from Lemma \ref{lemma:skein_cat_of_line_defect_is_L} and the excision property of skein categories.
\end{proof}

\begin{corollary}
\label{corollary:composing-locally-algebras}
  Let $\D_1$ be a defect from $\B$ to $\A$, and $\D_2$ from $\C$ to $\B$.
  We have the following equivalences as algebra objects in $(\A,\C)$-bimodules
\begin{equation}\label{eq:tunnel-composition}
  \includegraphics{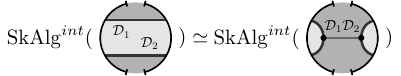}
\end{equation}
  and
\begin{equation}\label{eq:bubble-composition}
  \includegraphics{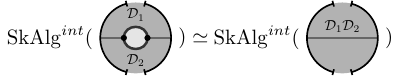}.
\end{equation}
\end{corollary}

\section{Weyl and parabolic defects}\label{sec:weyl-and-parabolic-defects}

In this section we will define Weyl defects and establish some of their salient properties.
A major focus will be the relationship between Weyl and parabolic defects, culminating in Theorem \ref{thm:borel-almost-invertible}, which exhibits a \emph{``generic invertibility''} of parabolic defects by showing that localising their compositions produces the invertible Weyl defect.
The word generic here alludes to the fact that localising a coordinate corresponds to restricting to a dense subset of the underlying moduli space. 

Throughout this section we will depict $G$-regions in light blue and $T$-regions in a darker orange. The two types of regions will always be separated by a parabolic defect $\D_B$ or its affinisation $\widetilde{D_B}$, here we do not distinguish between duals. We will only consider Weyl defects $\T_{w_0}$ in the $T$-region, and we depict these in green.

 \begin{definition}
    Recall that the Weyl group $W$ of a reductive group $G$ acts by conjugation on the maximal torus $T$, which in turn produces a braided auto-equivalence of $\Rep_qT$. Let $\chi_\lambda \mapsto \chi_{w\cdot\lambda}$ denote this equivalence on weight representations.
    The \emph{\bf Weyl defect} $\T_w$ is the codimension-1 endo-defect of $\Rep_qT$ determined by the pivotal braided tensor functor
\begin{equation}
    \begin{aligned}
        \Rep_qT \boxtimes \Rep_qT^{\bop} &\to \Zdr(\Rep_qT) \\
        \chi_\alpha\boxtimes \chi_\beta &\mapsto \left(\chi_\alpha \otimes\chi_{w\cdot \beta}, (\sigma_{\chi_\alpha,-}\otimes\id)\circ(\id \otimes \sigma^{-1}_{-,\chi_{w\cdot\beta}}) \right).
    \end{aligned}
\end{equation}
 \end{definition}
 
A skein passing through a Weyl defect must be labelled by an equivariant morphism, this heavily restricts what coupons can live on the defect:
\begin{lemma}\label{lemma:weyl-defect-weights}
  Let $\chi_\alpha, \chi_\beta$ be the highest weight representations of $T$ associated to characters $\alpha, \beta$, and let $X_{\alpha}$ (resp. $X_\beta$) be the $\Rep_qT$-labellings of $\Rib_{\T_w}$ with a single positively oriented point labelled $\chi_\alpha$ (resp. $\chi_\beta$) on the side of the inwards (resp. outwards) pointing normal of the defect, as shown in Figure \ref{fig:weyl-defect}.
  Then 
  \begin{equation}
    \Hom_{\SkCat_{\T_w}(\DDsurf)}(X_\alpha,X_\beta) = \begin{cases}
      \K, & \text{ if } \beta = w\cdot \alpha \\
      0, &\text{ otherwise.}
    \end{cases}
  \end{equation}
\end{lemma}
This proof follows from a direct computation. In light of this, we can think of the Weyl defect as acting to change the labelling representation of skeins from $\chi_{\lambda}$ to $\chi_{w\cdot \lambda}$, see Figure \ref{fig:weyl-defect}.
\begin{figure}
  \centering
  \includegraphics[scale=.8]{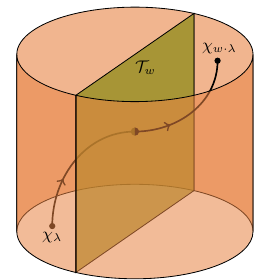}\hspace{20mm}
  \raisebox{10mm}{\includegraphics{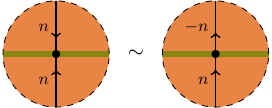}}
  \caption{\textbf{Left}: The de-facto effect of the Weyl defect. \textbf{Right}: When $G=\SL_2\CC$ and $w = (12) \in S_2 $, the effect can be understood a flipping the orientation of the skein or, equivalently, multiplying the weight by $-1$.}\label{fig:weyl-defect}
\end{figure}

\begin{example}[Kauffman Bracket]
Let $G = \SL_2\CC$, so that the Weyl group is $S_2 = \langle e, (12)\rangle$, the weight lattice is $\ZZ$, and the action is $(12)\cdot \lambda = - \lambda$.  

The practical effect of the defect $\T_{(12)}$ is to implement this action, sending a weight $n$ $\CC^*$-skein to a weight $-n$ one. Since $\chi_n^* \cong \chi_{-n}$, we can equivalently interpret this as changing the orientation of the skein, see Figure \ref{fig:weyl-defect}.
\end{example}

\subsection{The junction between Weyl and parabolic defects}\label{sec:weyl-parabolic-junction}

 Fix a Borel subgroup $B\subset G$ with corresponding Cartan $T:= B/[B,B]$.
 We will now construct a three-way junction between two parabolic defects associated to $B$ and a Weyl defect. 
 In future works, such a junction will allow us to move between different parabolic defects. For the time being we use it to articulate the generic invertibility of the parabolic defect, see Theorem \ref{thm:borel-almost-invertible}.
 
 First, let $U_q(\mathfrak{g})$ be the Drinfeld-Jimbo quantum group associated to $G$, with Serre-Chevalley generators $E_i,F_i$ and $K_i$. We define $U_q(\mathfrak{b})$ and $U_q(\mathfrak{t})$ as the subalgebras generated by all $E_i, K_i$ and all $K_i$, respectively.
 The canonical inclusion and projection maps $i: U_q(\mathfrak{b}) \to U_q(\mathfrak{g})$  and $\pi: U_q(\mathfrak{b}) \to U_q(\mathfrak{t})$ induce pullback functors $i^* \colon \Rep_q{G} \to \Rep_q{B}$ and $\pi^* \colon \Rep_q{T} \to \Rep_q{B}$, which we use to define a braided tensor functor
 \begin{equation}\label{eq:borel-defect}
    \begin{aligned} 
       \Rep_q{G} \boxtimes \Rep_q{T}^{\bop} & \rightarrow \Zdr(\Rep_q{B}) \\
       V \boxtimes \chi & \mapsto (i^\ast V \otimes \pi^\ast \chi , c_{(V , \chi) , -}).
    \end{aligned}
 \end{equation}
The half braiding is given by
    \begin{equation}\label{eq:half_braiding}
      c_{(V , \chi) , U}  \colon (i^\ast V \otimes \pi^\ast \chi) \otimes U \xrightarrow{\tau_{(1 2 3)} \circ (\mathcal{R}^G)_{1 3} \circ (\mathcal{R}^T)^{-1}_{3 2}} U \otimes (i^\ast V \otimes \pi^\ast \chi),
    \end{equation}
    where $\mathcal{R}^G, \mathcal{R}^T$ are the R-matrices of $\Rep_qG$ and $\Rep_qT$, respectively.
    From \eqref{eq:borel-defect} we can construct the \emph{\bf parabolic defect} $\D_B$.
    As discussed in \cite{Brown_Jordan_2025}, the category $\Rep_qB$ is not well behaved from a skein-theoretic viewpoint, e.g. it has no compact projectives.
    It's best replaced by its \emph{\bf affinisation} $\widetilde{\Rep_qB} := \SkAlg^{int}(\DD_1)-\mod_{G\times \overline{T}}$.
    In \cite{JLSS2021}, the authors prove the existence of a reflective embedding
    \begin{equation}\label{eq:B-affinisation}
        \textrm{Aff}: \Rep_qB \hookrightarrow\widetilde{\Rep_qB}.
    \end{equation}
    We use the shorthand $\widetilde{\D}_B := \widetilde{\Rep_qB}$ and call this the \emph{\bf redecorated parabolic defect}.

We will need to understand the dual of the parabolic defects. 
Dual codimension-1 defects are given by the opposite monoidal category, for parabolic defects these have a concrete description in terms of opposite parabolics. 
Namely, exchanging $E_i \leftrightarrow F_i$ and inverting $K_i^{\pm 1 } \leftrightarrow K_i^{\mp 1}$ defines an algebra and anti-coalgebra isomorphism $\phi: U_q(\mathfrak{b}_\minus)\cong U_q(\mathfrak{b})$, which induces a monoidal equivalence
\begin{align}\label{eq:exchange-borels}
    \Phi := \phi^* :\Rep_q{B}^{\mop} \to \Rep_q{B_{\minus}}.
\end{align}
In particular $\Phi( \pi^*(\chi_\lambda)) = \pi_{\minus}^*(\chi_{-\lambda})$ and $\Phi(i^*(V_\lambda)) = i_{\minus}^*(V_{-w_0\lambda})$, where $i_{\minus}$ and $\pi_{\minus}$ denote the inclusion and projection of $U_q(\mathfrak{b}_{\minus})$ into $U_q(\mathfrak{g})$ and onto $U_q(\mathfrak{t})$, respectively.
The dual defect $\D_B^\vee$ can be written in terms of this equivalence
    \begin{equation}\label{eq:borel-dual}
         \begin{aligned}
       \Rep_q{T} \boxtimes \Rep_q{G}^{\bop} & \rightarrow \Zdr(\Rep_q{B_{\minus}}) \\
    \chi_{\lambda} \boxtimes V_{\mu} & \mapsto (i_{\minus}^\ast V_{-w_0\mu} \otimes \pi^\ast_{\minus} \chi_{-\lambda} , \Phi(c^{-1}_{(V_{\mu} , \chi_{\lambda}) , \Phi^{-1}(\,\cdot\,)})).
    \end{aligned}
    \end{equation}
    
We can now construct a junction between $\T_{w_0}$ and two copies of $\widetilde{\D_B}$, with underlying linear category $\L:=\Rep_q T$.
We describe the centered bimodule structures first on $\D_B$, then show they lift naturally to $\widetilde{\D_B}$.
Throughout, we will use that $\L$ happens to be monoidal, so it's enough to define monoidal functors $\T_{w_0} \to \L$ and $\D_B \to \L$.
The following algebra homomorphisms endow $\L$ with left $\Rep_qB$ and $\Rep_q{B_{-}}$ actions
\begin{equation}
    \begin{aligned}
  \iota : U_q{\mathfrak{t}} &\to U_q{\mathfrak{b}},  &&\text{and}&  j: U_q{\mathfrak{t}} &\to U_q{\mathfrak{b}_{-}}\\
             K_i &\mapsto K_i                         &&&                   K_i  &\mapsto K_{w_0(i)}^{-1}
    \end{aligned}
\end{equation}
 while the monoidal structure on $\Rep_qT$ itself induces a left action $\T_{w_0} \times \L \to \L$.
The corresponding right module structures on the dual defects are induced by the equivalences $\Phi: \Rep_q{B}^{\mop} \cong \Rep_q{B_{-}}$:
\begin{align}
  \iota^*\circ \Phi^{-1}: \Rep_q{B_{-}} &\to \Rep_q{T}, & j^*\circ\Phi: \Rep_q{B} &\to \Rep_q{T}.
\end{align}
In each case let $\blacktriangleleft$ denote the induced right action on the dual defect.
Here we've used the fact that $\T_w^\vee \cong\T_{w}^{-1} \cong \T_{w^{-1}}$, meaning that $\T_{w_0}$ is self-dual.

Each pair of codimension-1 defects now gives a bimodule structure on $\L$, see Figure \ref{fig:weyl-junction-bimodules}.
\begin{figure}
  \centering
  \includegraphics{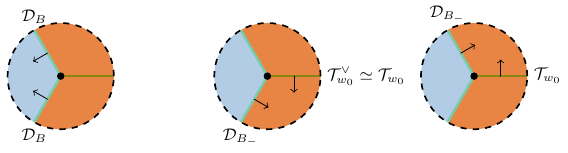}
    \caption{The three centered bimodule structures on $\L$ described in \eqref{eq:G-centered-BB-bimodule}, \eqref{eq:T-centered-BT-bimodule}, and \eqref{eq:T-centered-TB-bimodule}. The codimension-2 defect is oriented out of the page.}\label{fig:weyl-junction-bimodules}
\end{figure}
 It follows from direct computation that these bimodules are centered over the appropriate ribbon category:
\begin{description}
   \item[$\Rep_qG$-centered $(\D_{B},\D_B)$-bimodule:]
    \begin{equation}\label{eq:G-centered-BB-bimodule}
        \chi_\nu \blacktriangleleft i^*(V_\lambda)= j^*\Phi( i^*(V_\lambda)) \otimes \chi_\nu \cong \iota^*i^*(V_\lambda) \otimes \chi_\nu = i^*V_\lambda \triangleright \chi_\nu
    \end{equation}
    \item[$\Rep_qT$-centered $(\D_{B_{\minus}},\T_{w_0})$-bimodule:]
    \begin{equation}\label{eq:T-centered-BT-bimodule}
        \chi_\nu\blacktriangleleft \chi_{w_0\lambda}=\chi_{w_0\lambda} \otimes \chi_\nu  \cong j^*(\pi_{-}^*\chi_{-\lambda})\otimes \chi_\nu=(\pi_{-}^*\chi_{-\lambda})\triangleright \chi_\nu=\Phi(\pi^*\chi_\lambda)\triangleright \chi_\nu.
    \end{equation}
    \item[$\Rep_qT$-centered $(\T_{w_0}, \D_{B_{\minus}})$-bimodule:]
    \begin{equation}\label{eq:T-centered-TB-bimodule}
        \chi_\nu\blacktriangleleft \Phi\pi^*(\chi_\lambda)=\iota^*\Phi^{-1}(\Phi\pi^*(\chi_\lambda)) \otimes \chi_\nu \cong \chi_\lambda \otimes \chi_\nu  = \chi_\lambda\triangleright \chi_\nu
    \end{equation}
\end{description}
Here $\chi_\nu$ denote arbitrary weight representations in $\L$, whereas $\chi_\lambda$ and $V_\lambda$ denote arbitrary representations from the bulk regions. Pullback along the left adjoint $\mathrm{Aff}^L: \widetilde{\Rep_qB} \to \Rep_qB$ of \eqref{eq:B-affinisation} endows $\L$ with the appropriate $\widetilde{\Rep_qB}$ actions. 
Because $\mathrm{Aff}$ is fully faithful, we have that $\mathrm{Aff}^L\mathrm{Aff} \cong \id$. This implies that the induced actions remain centered.

We will see that this junction behaves similarly to the composition junction of Section \ref{sec:composition-junction}:
\begin{lemma}\label{lemma: skein passing through point defect}
Skeins can pass through the Weyl codimension-2 defect. That is, the following skeins are equal up to a scalar $\alpha$:
\begin{equation}
\begin{aligned}
        \includegraphics{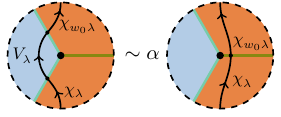},
\end{aligned}
\end{equation}
where $\lambda$ is a dominant weight for $G$.
\end{lemma}
\begin{proof}
Both skeins evaluate to a non-zero morphism in $\Hom_\L(\chi_{w_0\lambda}, \chi_{w_0\lambda})$, which, by Schur's lemma, is one-dimensional.
\end{proof}

\subsection{Generic invertibility of parabolic defects}\label{sec:invertible-statement}

We can now state our main result about Weyl and parabolic defects, namely that composing dual (affinised) parabolic defects we arrive at a $\Rep_qT$ endo-defect which is closely related to the Weyl defect. The proof will be deferred to Section \ref{sec:invertible-proof}, in the meanwhile we establish some necessary prerequisites. 

In what follows let $\PP_n$ denote a decorated disk with $n$ distinct $T$-regions which contract to points on the boundary, all separated by $\widetilde{\D_B}$ from a single contractible $G$-region. We make frequent implicit use of the fact that the associated internal skein algebra is the same whether the defects are labelled by $\D_B$ or by $\widetilde{\D_B}$, even though the skein categories are different.
We will call $\PP_2$ the \emph{\bf digon}. 

Suppose that $\PP_2$ is equipped with one gate in each region.  It was shown in \cite{JLSS2021} that 
\begin{equation}\label{eq:digon-skein-isomorphism}
    \SkAlg^{int}_{G\times T\times T^{bop}}(\mathbb{P}_2) \cong O_q(G^+/N)^{\otimes 2}
\end{equation}
as algebra objects in $\Rep_qG \boxtimes \Rep_qT \boxtimes \Rep_qT^{bop}$. 
Let $A_\lambda \in \SkAlg^{int}_{G\times T\times T^{bop}}(\mathbb{P}_2)$ denote the $G$-invariant skein which connects the two $T$-gates, oriented in the $T$-regions towards the defects, labelled by $\chi_\lambda$ in one $T$-region, and therefore necessarily by $\chi_{-w_0\lambda}$ in the other. See the discussion surrounding \eqref{eq:digon-coordinate} for a description of $A_\lambda$ in $O_q(G^+/N)^{\otimes 2}$.

\begin{theorem}\label{thm:borel-almost-invertible}
The parabolic defect enjoys generic invertibility. That is, there exists a categorical equivalence
     \begin{align}\label{eq:borel-almost-invertible}
         \widehat{\SkCat}(\Digon)[A_\rho^{-1}] \cong \widehat{\SkCat}( \FullWeylCat )
     \end{align}
    where $[A^{-1}_\rho]$ refers to the subcategory where the skein $A_\rho \in \SkAlg^{int}(\PP_2)$ labelled by $\idty\boxtimes\chi_\rho\boxtimes \chi_{-w_0\rho} \in \Rep_qG\boxtimes \Rep_qT\boxtimes\Rep_qT^{bop}$ acts invertibly. Recall moreover that $ \;\widehat{\cdot}$ denotes the free co-completion.
\end{theorem}
Note that even in the statement of Theorem \ref{thm:borel-almost-invertible} we've made implicit use of monadic reconstruction.
Namely, this is the process that describes (free co-completed) skein category as a category of modules for the internal skein algebra of the digon, and in that setting it makes sense to localise the skein $A_\rho$.

The connection to invertibility comes from the fact that Weyl defect is invertible, with $\T_w^{-1} = \T_{w^{-1}}$. Localisation appears often when discussing the skeins or cluster variables associated to the edges of a triangulated surface.
This is in particular a common feature when comparing skein and  cluster algebras of marked surfaces \cite{IY22,Muller_2016}.
From the skein category perspective, inverting edge variables allows us to apply monadic reconstruction to skein categories equipped with $\Rep_qT^{\boxtimes n}$-module category structures, i.e. without including $\Rep_qG$-actions.
This monadic reconstruction is then responsible for the good gluing properties of the associated internal skein algebras.

Concretely, this allows us to model the quantum cluster chart of a decorated triangulated surface from those of its individual triangles and digons.

\subsection{Localisation as a cobordism}\label{sec:localisation-cobordism}

We give a novel interpretation of localisation on quantum decorated character stacks using the defect skein TQFT with parabolic and Weyl defects.
Our basic approach is to build a decorated cobordism $M_\triangle$ from a surface $\Sigma$ with an ideal triangulation $\triangle$, which will relate the localised and non-localised theories. This gives a `sheafy' or categorical extension of localisation.

In what follows fix a Borel subgroup $B \hookrightarrow G$, so that by `parabolic defect' we mean the one associated to this distinguished Borel, and by `Weyl defect' we mean the one associated to the longest word $w_0 \in W$.

\begin{definition}\label{def:skein-cluster-functor}
Fix a closed, oriented surface $\Sigma$ with at least one boundary component, marked points $S = \{s_1,\ldots,s_{v-m}\} \subset \partial \Sigma$, punctures $P = \{s_{v-m+1},\ldots,s_{v}\} \subset \Sigma\setminus\partial \Sigma$, and a triangulation $\triangle$ of $\Sigma$ with vertices $S\cup P$.
We permit self-folded triangulations.

The \emph{\bf localisation functor $\underline{\Sk}(M_\triangle)$} is constructed as follows.

\begin{description}
\item[Step 1: Framed Surface]
   Let $\Sigma'$ denote the decorated surface built from $\Sigma$ as follows. First, remove a small open neighbourhood of each puncture and enclose the newly introduced boundary component with an annular $T$-region.
    Add a contractible $T$-region neighbourhood around each marked point $s_i$ and separate each $T$-region by a parabolic defect from the $G$-region.
    
    \item[Step 2: Decorated Cobordism] 
The decorated cobordism $M_{\triangle} : \Sigma' \to \Sigma_\triangle$ has underlying unstratified manifold with corners $\Sigma \times [0,1]$.

Each $e\in E(\triangle)$ can be made to cross parabolic defect arcs in $\Sigma'$ exactly twice.

Introduce a saddle in the thickened surface whose incoming boundary is a pair small intervals along parabolic arcs, containing their intersections with $e$. The outgoing boundary is two curves running parallel to $e$, between the endpoints of these same small intervals.
This saddle is decorated by the parabolic defect data, so that `inside' the saddle is $T$-region and `outside` is $G$-region.

The stable manifold of the saddle point intersects the saddle's upper boundary in two points.
Connect these points by a curve perpendicular to the edge $e$ and contained in the $T$-region, so that together this curve and the stable manifold bound a disk in the $T$ region. 
Endow this disk with the Weyl defect, see Figure \ref{fig:abelian-cobordism}.

By construction, the outgoing boundary $\Sigma_\triangle$ of $M_{\triangle}$ has a $G$-disk at the centre of each face of $\triangle$, as well as along the boundary between each pair of marked points.
These $G$-regions are connected by Weyl defects perpendicular to the original edges of the triangulation, see Figure \ref{fig:abelian-cobordism}.

\item[Step 3: Skein Functor]
 In the skein theory TFT, $M_\triangle$ induces a functor
\begin{equation}
  \begin{aligned}
    \Sk(M_\triangle) : \widehat{\SkCat(\Sigma')} &\to \widehat{\SkCat}(\Sigma_\triangle).
  \end{aligned}
\end{equation}
Equipping each $T$-region with a gate, we arrive at the internal skein module $\Sk^{int}(M_\triangle)$, which is a $(\SkAlg^{int}(\Sigma'), \SkAlg^{int}(\Sigma_\triangle))$-bimodule internal to $\Rep_qT^{\boxtimes v}$, we therefore get a functor
\begin{equation}
    \begin{aligned}
        \SkAlg^{int}(\Sigma')-\mod_{T^v} &\to \SkAlg^{int}(\Sigma_\triangle)-\mod_{T^v} \\
    \end{aligned}
\end{equation}
which, thanks to their gluing properties acts as $\Sk^{int}(N) \mapsto \Sk^{int}(M_\triangle \cup_{\Sigma'}N)$ on internal skein modules. 
\end{description}
\end{definition}

The name ``localisation functor'' is justified by Corollary \ref{cor:general-surfaces-invertibility}, which states that the freely co-completed skein category of $\Sigma_\triangle$ is a localisation of that of $\Sigma'$.

By the identification between skein categories and factorisation homology\footnote{See \cite{Cooke2019,Brown_Haioun_2026} for an unstratified version}, for $G=\SL_2\CC$ we recover the cluster charts for the quantum decorated character stacks of \cite{JLSS2021}. The precise correspondence is $\widehat{\SkCat}(\Sigma') \cong \mathcal{Z}(\mathbb{S})$ and $\SkAlg^{int}(\Sigma_\triangle)-\mod_{T^m} \cong \mathcal{Z}(\underline{\Delta})$, where the internal skein algebra $\SkAlg^{int}(\Sigma_\triangle)$ is constructed with a $T$-gate at each vertex $s_1,\ldots,s_m$ of the ideal triangulation $\triangle$.

\subsubsection{Localisation stabilises}\label{sec:localisation-stablises}

The procedure in Definition \ref{def:skein-cluster-functor} naturally stabilises once a maximal set of non-overlapping edge variables have been localised.

\begin{proposition}\label{prop:localisation-stablises}
  Let $(\Sigma,S,P)$ be a marked surface with an ideal triangulation $\triangle$ and associated decorated surface $\Sigma_\triangle$, as in Definition \ref{def:skein-cluster-functor}. Suppose that $e \hookrightarrow \Sigma$ is some additional edge, not in $\triangle$.
  The associated skein categories are equivalent:
  \begin{equation}
    \SkCat(\Sigma_\triangle) \cong \SkCat(\Sigma_{\{e\}\cup \triangle}),
  \end{equation}
  where $\Sigma_{\{e\}\cup\triangle}$ is constructed by composing/localising all adjacent pairs of parabolic defects along some choice of path $\Gamma \hookrightarrow \Sigma_\triangle$ which avoids codimension-2 defects and which is isotopic to $e$ upon forgetting the stratification on $\Sigma_\triangle$.
\end{proposition}
\begin{proof}
  Since skein categories satisfy excision, we can focus on the minimal local model without loss of generality.
  Let $\Sigma = \PP_4$ be a disk with four marked points labelled $1,2,3,4$ in clockwise order, and let $\triangle = \{e_{12},e_{23},e_{34},e_{41},e_{13}\}$. Here we've labelled each edge according to the vertices it connects.
  Then $e_{24} \not\in \triangle$ is our extra edge, which intersects $e_{13}$ in the interior. 

  Let $\Gamma_{24} \hookrightarrow \Sigma_\triangle$ be a lift of $e_{24}$ as in the proposition statement.
  If $\Gamma_{24}$ passes through no parabolic defects, then there's nothing to be done. Note that a non-trivial skein associated to such a path will be invertible.
  Suppose then that $\Gamma_{24}$ passes through one of the $G$-regions on $\Sigma_{\triangle}$.
  Composing parallel parabolic defects along such a path with introduce a `bubble', but together Lemma \ref{lemma:composing-locally} and Lemma \ref{lemma: skein passing through point defect} imply that we can remove this bubble without changing the skein category, recovering $\Sigma_\triangle$ up to a stratified isotopy.
See Figure \ref{fig:localisation-stablises}.
\end{proof}

\begin{figure}
  \centering
  \includegraphics{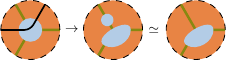}
  \caption{Localisation naturally stabilises, see Proposition \ref{prop:localisation-stablises}. Here $\cong$ means `have equivalent skein categories.'}\label{fig:localisation-stablises}
\end{figure}

\section{Proof of generic invertibility of parabolic defects}\label{sec:invertible-proof}

We can now focus on proving generic invertibility of the parabolic defect $\widetilde{\D_B}$ defined from a Borel subgroup $B\hookrightarrow G$.
Recall our goal:

\begin{theorem*}[Theorem \ref{thm:borel-almost-invertible}]
There exists a categorical equivalence
     \begin{align*}
         \widehat{\SkCat}(\Digon)[A_\rho^{-1}] \cong \widehat{\SkCat}( \FullWeylCat )
     \end{align*}
    where $[A^{-1}_\rho]$ refers to the subcategory where the skein $A_\rho \in \SkAlg^{int}(\PP_2)$ labelled by $\idty\boxtimes\chi_\rho\boxtimes \chi_{-w_0\rho} \in \Rep_qG\boxtimes \Rep_qT\boxtimes\Rep_qT^{bop}$ acts invertibly. Recall moreover that $ \;\widehat{\cdot}$ denotes the free co-completion.
\end{theorem*}

The first step to its proof is to establish monadic reconstruction. 

\subsection{Monadic reconstruction without G-gates}\label{sec:monadic-reconstruction-without-G}

The main result of this section is Proposition \ref{prop:G-invariants-equivalence}, which will crucially allow us to apply monadic reconstruction to skein categories of decorated surfaces without opening $G$-gates.
This comes at the cost of inverting certain skeins. 

This result is a quantum analogue of the fact that the moduli space of pairs decorated flags in generic position is a $G\times T$ - torsor. 
It is particularly useful because the associated quantum cluster theory concerns $G$-invariant elements.

A version of this statement for $G = \SL_2$ appears in \cite[Theorem 3.63]{JLSS2021}. We thank Gus Schrader for laying the groundwork for this section.
\begin{proposition}\label{prop:G-invariants-equivalence}
    The functor $\mathcal{G}$ of taking $U_q(\g)$-invariants is an equivalence of categories
    \begin{equation}
        \begin{aligned}
        \mathcal{G}: O_q(G^+/N)^{\otimes2}[A_\rho^{-1}]-\mod_{G\times T\times T^{bop}} \to  (O_q(G^+/N)^{\otimes2}[A_\rho^{-1}])^{U_q(\g)}-\mod_{T\times T^{bop}}.
    \end{aligned}
    \end{equation}
\end{proposition}
Here we write $A-\mod_{G\times T \times T^{bop}}$ to mean the category of $A$-modules in the free co-completion of $\Rep
_q^{fd}G\boxtimes \Rep_q^{fd}T \boxtimes \Rep_q^{fd}T^{bop}$, and will use $\sigma$ to denote the braiding in any ambient braided tensor category, e.g. $\sigma(v_\lambda \otimes v_\mu) = q^{(\lambda,\mu)}v_\mu \otimes v_\lambda$ in $\Rep_qT$.
The algebra $O_q(G^+/N)$ is a twisted version of $O_q(G/N)$, both introduced below.

Recall that the Peter-Weyl theorem decomposes the Faddeev-Reshetikhin-Takhtajan (FRT) algebra $O_q(G)$ into a direct sum over finite dimensional irreducible representations:
\begin{equation}\label{eq:peter-weyl-FRT}
     O_q(G) \cong \bigoplus V_\lambda^* \boxtimes V_\lambda.
\end{equation}
The isomorphism \eqref{eq:peter-weyl-FRT} holds both as algebras and $U_q(\g)^{cop}\otimes U_q(\g)$ modules, where the $U_q(\g)$-actions on the first and second factor of the right hand side correspond to the right and left co-regular action, respectively.
One can then decompose $O_q(G/N) := O_q(G)^N$ as
\begin{equation}
    O_q(G/N) \cong \bigoplus V_\lambda^* \boxtimes \mathbb{C}_\lambda,
\end{equation} where the second factor on the right hand side is the highest weight line of $V_\lambda$. 
We will denote an element $\varphi \boxtimes v_\lambda$ in $V_\lambda^* \boxtimes \mathbb{C}_\lambda \subset O_q(G/N)$ by $\varphi^\lambda$ and the left co-regular action by $x \rhd \varphi^\lambda := \varphi \boxtimes (x \cdot v_\lambda)$. 

We will need a twisted version of $O_q(G)$, denoted by $O_q(G^+)$, which has the same underlying vector space as $O_q(G)$ and multiplication given by
\begin{align}\label{eq:Gp-multiplication}
    (f \boxtimes v)(g\boxtimes w) = (f \otimes g) \boxtimes \sigma(v \otimes w). 
\end{align}
The Peter-Weyl theorem still holds but now as $U_q(\g) \otimes U_q(\g)$-modules.
The tensor equivalence 
\begin{equation}
    (\id,\sigma)\boxtimes (\id,\id) : \Rep_q G \boxtimes \Rep_qG \to \Rep_qG^{mop} \boxtimes \Rep_qG
\end{equation}
sends $O_q(G^+) \to O_q(G)$, making the two \emph{twist equivalent}.
Taking $N$-invariants, we get a twisted version $O_q(G^+/N)$ of the algebra $O_q(G/N)$.
The connection to skein theory is via the isomorphisms
  \begin{align}
    O_q(G^+/N) &\cong \SkAlg^{int}(\DD_1) &&\text{in } \Rep_q G \boxtimes \Rep_q T^{bop}, \\
    O_q(G^+) &\cong \SkAlg^{int}(\DD) &&\text{in }  \Rep_qG\boxtimes \Rep_qG, \\
    O_q(G^+/N)^{\otimes2} &\cong \SkAlg^{int}(\PP_2) && \text{in } \Rep_qG \boxtimes \Rep_qT \boxtimes \Rep_qT^{bop}.
  \end{align}

Next we build a homomorphism $\tau$ which we will modify until we arrive at the isomorphism \eqref{eq:generic-position-flags-iso} of Proposition \ref{prop:generic-position-flags-iso}.
This map features the \emph{\bf half-twist element $Y$} in the Lusztig's completion of $U_q(\g)$, see e.g. \cite[\S 4.1]{Snyder_Tingley_2009}.
\begin{lemma}\label{lemma:GpN-to-TGp}
    The following is an algebra homomorphism in $\Rep_qG \boxtimes \Rep_qT \boxtimes \Rep_qT^{bop}$
    \begin{equation}
        \begin{aligned}
        \tau : O_q(G^+/N) &\to O_q(T)\otimes O_q(G^+)\\
        \varphi^\lambda &\mapsto \chi_{-\lambda} \otimes( Y\rhd \varphi^\lambda ),
    \end{aligned}
    \end{equation}
    where $U_q(\mathfrak{g})\boxtimes \idty \boxtimes \idty$ acts by the right co-regular action on both $O_q(G^+/N)$ and $O_q(G^+)$, while $\idty\boxtimes U_q(\t)\boxtimes \idty$ acts trivially on $O_q(G^+)$, by the right co-regular action on $O_q(T)$, and by the left co-regular action on $O_q(G^+/N)$.
    Conversely, $\idty\boxtimes \idty \boxtimes U_q(\t)$ acts by the left co-regular action on $O_q(G^+)$, by the right co-regular action twisted by the involution $\theta(K_\lambda) = K_{-w_0(\lambda)}$ on $O_q(T)$, and trivially on $O_q(G^+/N).$
\end{lemma}
\begin{proof}
The left and right co-regular action commute, so the map $\tau$ is $U_q(\g)$-linear.

The character $\chi_{-\lambda}$ has weight $\lambda$ with respect to the right co-regular action and hence has weight $(\lambda,-w_0\lambda)$ in $O_q(T)$.
The weight of the left co-regular action on $Y \rhd \varphi^\lambda = \varphi \boxtimes (Y\cdot v_\lambda)$ is given by the weight of $Y\cdot v_\lambda$. As $Y$ interchanges the highest and lowest weight lines, the $U_q(\t) \otimes U_q(\t)$-weight of $Y \rhd \varphi^\lambda$ is $(0, w_0\lambda)$. Thus, the resulting weight of $\tau(\varphi^\lambda)$ is $(\lambda,0)$, matching that of $\varphi^\lambda$ and proving that the map $\tau$ is a morphism of $U_q(\t) \otimes U_q(\t)$-modules.

We can check directly that $\tau$ respects the product:
\begin{equation}
    \begin{aligned}
    \tau(\varphi^\lambda\psi^\mu) &= \tau(q^{(\lambda,\mu)}(\varphi \otimes \psi)\boxtimes (w_\mu \otimes v_\lambda))
    \\
    &= q^{(\lambda,\mu)}\chi_{-\lambda-\mu}\otimes(\varphi \otimes \psi)\boxtimes (\Delta(Y)\cdot(w_\mu \otimes v_\lambda))\\
    &= q^{2(\lambda,\mu)}\chi_{-\lambda-\mu}\otimes(\varphi \otimes \psi)\boxtimes (Y\cdot w_\mu \otimes Y\cdot v_\lambda))\\
    &= q^{2(\lambda,\mu)-(w_0\lambda,w_0\mu)}\chi_{-\lambda-\mu}\otimes(Y\rhd \varphi^\lambda)( Y\rhd \psi^\mu)\\
    &= q^{(\lambda,\mu)}\chi_{-\lambda-\mu}\otimes(Y\rhd \varphi^\lambda)( Y\rhd \psi^\mu).
\end{aligned}
\end{equation}
In the third equality we use that by definition of a half-twist element, the $R$-matrix can be written as $R = (Y^{-1}\otimes Y^{-1})\Delta(Y)$. The other $q$-factors follow from \eqref{eq:Gp-multiplication}.
We conclude
\begin{equation}
    \begin{aligned}
    \tau(\varphi^\lambda)\tau(\psi^\mu) &=  (  \chi_{-\lambda}\otimes Y\rhd \varphi^\lambda )(\chi_{-\mu} \otimes Y\rhd \psi^\mu ) \\ 
    &= q^{-(w_0\lambda,-w_0\mu)} \chi_{-\lambda-\mu} \otimes(Y\rhd \varphi^\lambda)( Y\rhd \psi^\mu) \\
    & = \tau(\varphi^\lambda\psi^\mu),
\end{aligned}
\end{equation}
where now the power of $q$ comes from the braided tensor product of the algebras $O_q(T)$ and $O_q(G^+)$ in $ \Rep_qT \boxtimes \Rep_qT^{bop}$.
\end{proof}
Consider the tensor product $O_q(G^+/N)^{\otimes2}$ in $\Rep_qG\boxtimes \Rep_qT\boxtimes\Rep_q T^{bop}$, where now the actions of $\idty\boxtimes U_q(\t) \boxtimes \idty$ and $\idty\boxtimes\idty\otimes U_q(\t)$ correspond to the left co-regular action on the first and second $O_q(G^+/N)$ factor, respectively. Meanwhile, $U_q(\g)$ acts diagonally by the right co-regular action.
\begin{lemma}\label{lemma:non-localised-GN2-to-TxG-map}
    The map
    \begin{equation}
        \begin{aligned}
        \eta: O_q(G^+/N)^{\otimes 2} &\to O_q(T) \otimes O_q(G^+)\\
        \varphi^\lambda \otimes \psi^\mu &\mapsto \tau(\varphi^\lambda)(1\otimes\psi^\mu),
        \end{aligned}
    \end{equation}
        is an injective morphism of algebras in $\Rep_qG\boxtimes \Rep_qT\boxtimes \Rep_qT^{bop}$.
\end{lemma}
\begin{proof}
Both $\varphi^\lambda \otimes \psi^\mu$ and its image have $U_q(\t) \otimes U_q(\t)$-weight $(\lambda,\mu)$, making $\eta$ a morphism in $\Rep_qT\boxtimes \Rep_qT^{bop}$. Both $\tau$ and multiplication in $O_q(G^+)$ are $U_q(\g)$-equivariant, so $\eta$ is as well.

For an element $f^\lambda := f \boxtimes v_\lambda$ in $O_q(G^+)$, a weight vector $v_\lambda$, and $\psi^\mu := \psi \boxtimes v_\mu$ in $O_q(G^+/N)$, we have
\begin{equation}
    q^{(\mu,\lambda)} \psi^\mu f^\lambda = (R_{(2)}\cdot f^\lambda)(R_{(1)}\cdot \psi^\mu),
\end{equation}
where the $R$-matrix is understood to be acting by the right co-regular action.
In particular
\begin{equation}\label{eq:FRT-relation}
    q^{(\mu,w_0\lambda)}\psi^\mu (Y \rhd \varphi^\lambda) = (R_{(2)}\cdot (Y\rhd \varphi^\lambda))(R_{(1)}\cdot \psi^\mu).
\end{equation}
We can verify directly that $\eta$ respects the product:
\begin{equation}
    \begin{aligned}
    \eta(1 \otimes \psi^\mu)\eta(\varphi^\lambda \otimes 1) &= (1 \otimes \psi^\mu)\tau(\varphi^\lambda)\\
    &=(1 \otimes \psi^\mu)(\chi_{-\lambda}\boxtimes Y \rhd \varphi^\lambda)\\
    &=q^{-(\mu,-w_0\lambda)}\chi_{-\lambda}\boxtimes \psi^\mu (Y \rhd \varphi^\lambda)) \\
    &=\chi_{-\lambda}\boxtimes (R_{(2)}\cdot (Y\rhd \varphi^\lambda))(R_{(1)}\cdot \psi^\mu)\\
    &= \tau(R_{(2)}\cdot\varphi^\lambda )(1\otimes R_{(1)}\cdot\psi^\mu ) \\
    &=  \eta(R_{(2)}\cdot\varphi^\lambda \otimes R_{(1)}\cdot\psi^\mu )\\
     &=\eta((1 \otimes \psi^\mu)(\varphi^\lambda \otimes 1)) ,\\  
\end{aligned}
\end{equation}
where the power of $q$ in the third line comes from the braided tensor product of the algebras $O_q(T)$ and $O_q(G^+)$ in $ \Rep_qT \boxtimes \Rep_qT^{bop}$.
We use \eqref{eq:FRT-relation} to obtain the fourth equality and the fact that left and right co-regular action commute to obtain the fifth equality. The last line follows from the braided tensor product of $O_q(G^+/N)^{\otimes2}$ in $\Rep_qG\boxtimes \Rep_qT\boxtimes \Rep_qT^{bop}$.

Injectivity follows in the standard way from the triangular decomposition of $U_q(\g)$.
\end{proof}

Next we establish that we can localise at the set of $A_\lambda$'s. For any dominant $G$-weight $\lambda$, let $v_\lambda$ be a highest weight vector of $V_\lambda$ and $v_\lambda^*$ be its dual, normalised such that $\langle v_\lambda^*,v_\lambda\rangle = 1$.
From $M^\lambda := v_\lambda^* \boxtimes v_\lambda$ in $O_q(G^+/N)$ we obtain an algebraic description of the element $A_\lambda$ referenced in Theorem \ref{thm:borel-almost-invertible}:
\begin{equation}\label{eq:digon-coordinate}
    \begin{aligned}
    A_\lambda &= (Y^{-1}\rhd SM^{\lambda}_{(1)})\otimes M^\lambda_{(2)}
\end{aligned}
\end{equation}
Where $S$ is the antipode.
\begin{lemma}\label{lemma:Alambda-good-properties}
    Each $A_\lambda$ is $U_q(\g)$-invariant with $U_q(\t)\otimes U_q(\t)$-weight $(-w_0\lambda,
    \lambda)$, and any multiplicative set generated by $A_\lambda$'s is an Ore set in $O_q(G^+/N)^{\otimes 2}$.
\end{lemma}
\begin{proof}
    A direct computation shows that
\begin{align}\label{eq:eta-of-digon-element}
    \eta(A_\lambda) = \chi_{w_0\lambda} \otimes (SM^{\lambda}_{(1)}M^\lambda_{(2)}) = q^{-(\lambda,\lambda)}\chi_{w_0\lambda} \otimes 1.
\end{align}
Since $\eta$ respects the $U_q(\mathfrak{g})$ and $U_q(\mathfrak{t})$ actions, we see that $A_\lambda$ is $U_q(\g)$-invariant and has $U_q(\t) \otimes U_q(\t)$-weight $(-w_0\lambda,\lambda)$.
It also follows from \eqref{eq:eta-of-digon-element} that $\eta(A_\lambda)$ and $\eta(\varphi^\mu\otimes \psi^\nu)$ $q$-commute for any element of the form $\varphi^\mu\otimes \psi^\nu$ in $O_q(G^+/N)^{\otimes 2}$. Because $\eta$ is an injective algebra morphism, it follows that $A_\lambda$ and $\varphi^\mu\otimes \psi^\nu$ $q$-commute as well. Because the homogeneous elements span $O_q(G^+/N)^{\otimes 2}$, we have that any multiplicative set generated by $A_\lambda$'s for dominant weights $\lambda$, is an Ore set. 
\end{proof}

We can now establish the following proposition:
\begin{proposition}\label{prop:generic-position-flags-iso}
Let $\rho := \sum\varpi_i$ denote the sum of all fundamental $G$-weights. The map 
    \begin{equation}\label{eq:generic-position-flags-iso}
        \begin{aligned}
             \eta: O_q(G^+/N)^{\otimes 2}[A_\rho^{-1}] &\to O_q(T) \otimes O_q(G^+)
        \end{aligned}
    \end{equation}
        is an isomorphism of algebras in $\Rep_qG\boxtimes \Rep_qT\boxtimes \Rep_qT^{bop}$.
\end{proposition}
\begin{proof}
    The earlier discussion makes the localisation and the map well-defined. Note that because $\eta$ is an algebra map and $\eta( A_{\rho}^{-1}) \neq 0$, the injectivity is preserved by localisation. Recall that $\eta(A_\lambda) = q^{-(\lambda,\lambda)}\chi_{w_0\lambda} \otimes 1$ and that localising at a product is equivalent to localising at each of its factors. Moreover, because $A_{\lambda+\mu}$ is a scalar multiple of $A_\lambda A_\mu$ and because the fundamental weights form a basis of the weight lattice, it follows that $O_q(T) \otimes 1$ is contained in the image of $\eta$. Its surjectivity follows from the fact that $O_q(G^+/N)(Y \rhd O_q(G^+/N))$ generates $O_q(G^+)$, \cite[Section 9]{Joseph}.
\end{proof}

We are now able to prove Proposition \ref{prop:G-invariants-equivalence}.
\begin{proof}[Proof of Prop. \ref{prop:G-invariants-equivalence}]
Because $\Rep_qG$ is semisimple the functor  of taking invariants is exact. In order to prove that its monadic lift $\mathcal{G}$ is an equivalence, it suffices to show it is conservative. By Proposition \ref{prop:generic-position-flags-iso}, there exists an injective morphism $\phi: O_q(G^+) \to  O_q(G^+/N)^{\otimes2}[A_\rho^{-1}]$, identifying $O_q(G^+)$ as a subalgebra. Consequently, one can consider the following commuting diagram
\begin{equation}
\begin{tikzcd}
O_q(G^+/N)^{\otimes2}[A_\rho^{-1}]-\mod_{G\times T\times T^{bop}}
  \arrow[r, "\mathcal{G}"]
  \arrow[d, "\mathcal{F}"']
&
(O_q(G^+/N)^{\otimes2}[A_\rho^{-1}])^{U_q(\g)}-\mod_{T\times T^{bop}}
  \arrow[d, "\mathcal{F}'"]
\\
O_q(G^+)-\mod_{G\times T\times T^{bop}}
  \arrow[r, "\mathcal{G}'"]
&
\Rep_q{T\times T^{bop}}
\end{tikzcd}
\end{equation}
where the horizontal functors are the ones taking $U_q(\g)$-invariants, and the vertical ones forget part of the algebra to a subalgebra. Because forgetful functors are conservative, it suffices to show that $\mathcal{G}'$ is conservative. This is immediate because $\mathcal{G}'$ is an equivalence with inverse given by $V \mapsto O_q(G^+) \otimes V$.
\end{proof}

\subsubsection{Localisation for general surfaces}
A major consequence of Proposition \ref{prop:G-invariants-equivalence} is that we can apply excision to localised internal skein algebras without $G$-gates and therefore to our co-completed and localised skein categories, so that:

\begin{corollary}\label{cor:general-surfaces-invertibility}
  Let $\Sigma$ be an arbitrary redecorated surface, denote by $\Delta$ a collection of arcs on $\Sigma$ connecting boundary $T$ regions, by  $\widehat{\SkCat}(\Sigma)[\Delta^{-1}]$ the localisation of the skeins $A_\rho$ situated along each arc, and by $\Sigma_\triangle$ the Weyl defect surface where we composed the parabolic defects locally along each arc, as in Definition \ref{def:skein-cluster-functor}. 
  Then we have an equivalence of categories,
\begin{align}
         \widehat{\SkCat}(\Sigma)[\Delta^{-1}] \cong \widehat{\SkCat}( \Sigma_\triangle ),
     \end{align}
    induced by an isomorphism of internal skein algebras
    \begin{equation}
      \SkAlg^{int}(\Sigma)[\Delta^{-1}] \cong \SkAlg^{int}(\Sigma_\triangle).
    \end{equation}
\end{corollary}

\subsection{Proof of Theorem \ref{thm:borel-almost-invertible} on generic invertibility}\label{sec:proof-of-generic-invertibility}
We can now prove Theorem \ref{thm:borel-almost-invertible}, concerning the relationship between Weyl defects and pairs of parabolic defects.

\begin{proof}
Following \cite{JLSS2021} and \cite{Brown_Jordan_2025}, by monadic reconstruction and Proposition \ref{prop:G-invariants-equivalence} one has that
\begin{equation}
    \begin{aligned}
    \widehat{\SkCat(\PP_2)^{loc}} &\cong \SkAlg^{int}_{G\times T\times T^{bop}}(\PP_2)[A^{-1}_\rho]-\mod_{G\times T \times T^{bop}}\\ &\cong (\SkAlg^{int}_{G\times T \times T^{bop}}(\PP_2)[A^{-1}_\rho])^{U_q(\mathfrak{g})}-\mod_{T\times T^{bop}} .
\end{aligned}
\end{equation}
Localising and taking $U_q(\mathfrak{g})$ invariants are commuting operations, because $A_\rho$ is $U_q(\mathfrak{g})$-invariant. Taking $G$-invariants of an internal skein algebra corresponds to closing the corresponding $G$-gate. We have
    \begin{equation}
    (\SkAlg^{int}_{G\times T\times T^{bop}}(\PP_2)[A^{-1}_\rho])^{U_q(\mathfrak{g})} \cong 
    \SkAlg^{int}_{T\times T^{bop}}(\PP_2)[A^{-1}_\rho].
\end{equation} 
By the excision property of skein categories we can write
\begin{equation}
 \widehat{\SkCat}(\FullWeylCat) \cong \widehat{\SkCat}(\LeftWeylCat)\, \mathbin{\mathop{\boxtimes}\limits_{\widehat{\SkCat}(\WeylCat)}}\, \widehat{\SkCat}(\RightWeylCat).
 \end{equation}
Again, by monadic reconstruction we have the following equivalence
\begin{equation}
    \widehat{\SkCat}(\LeftWeylCat) \cong \SkAlg^{int}(\LeftWeylAlgGTT)-\mod_{G\times T\times T^{bop}}.
\end{equation}
Note that there's an injective homomorphism of algebra objects
\begin{equation}
    \SkAlg^{int}_{G\times T\times T^{bop}}(\PP_2) \hookrightarrow \SkAlg^{int}(\LeftWeylAlgGTT)
\end{equation} whose image is generated by skeins passing to the left of the codimension-2 defect. By Lemma \ref{lemma: skein passing through point defect}, the skeins not in the image are those passing through Weyl defect (i.e. right of the codimension-2 defect) and oriented towards it, labelled by non-dominant $G$-weights. These are exactly the elements introduced by localisation, because the fundamental weights form a basis of the weight lattice and multiplication is given by $A_\lambda A_\mu = q^{2(\lambda,\mu)}A_{\lambda+\mu}$.
It follows that
\begin{equation}\label{eq:generic-invertiblity-for-int-algs}
    \SkAlg^{int}(\LeftWeylAlgGTT) \cong \SkAlg^{int}_{G\times T \times T^{bop}}(\PP_2)[A_\rho^{-1}].
\end{equation}
Analogous arguments holds when the $G$-region is on the right.
By Proposition \ref{prop:G-invariants-equivalence},  closing the $G$-gate gives us

\begin{equation}
    \begin{aligned}
    \widehat{\SkCat}(\FullWeylCat)   
    &\cong \SkAlg^{int}(\LeftWeylAlgTT)-\mod_{ T\times T^{bop}} \mathbin{\mkern+20mu\mathop{\boxtimes}\limits_{\mathclap{\SkAlg^{int}(\WeylAlg)-\mod_{ T\times T^{bop}}}}\mkern+20mu} \SkAlg^{int}(\RightWeylAlgTT)-\mod_{ T\times T^{bop}}\\
    &\cong \left( \SkAlg^{int}(\LeftWeylAlgTT) \mathbin{\mathop{\otimes}\limits_{\SkAlg^{int}(\WeylAlg)}} \SkAlg^{int}(\RightWeylAlgTT) \right)-\mod_{T \times T^{bop}}\\
    &\cong \SkAlg^{int}(\FullWeylAlg)-\mod_{T \times T^{bop}},
    \end{aligned}
\end{equation}
where we used that the various algebras in play are braided commutative in $\Rep_qT \boxtimes \Rep_qT^{bop}$, see Lemma \ref{lemma:d1-is-commutative}.
The final result then follows from the isomorphism of algebra objects
\begin{equation}
    \SkAlg^{int}(\FullWeylAlg) \cong \SkAlg^{int}_{T\times T^{bop}}(\PP_2)[A^{-1}_\rho].
    \end{equation}
Indeed, the algebra on the right hand side before localisation is given by $\bigoplus_{\lambda \in \Lambda^+}\chi_\lambda  \boxtimes \chi_{-w_0\lambda}$ with the element $A_\lambda$ in a single summand and multiplication given by $A_\lambda A_\mu = q^{2(\lambda,\mu)}A_{\lambda+\mu}$.
Inverting $A_\rho$ corresponds to summing over the whole weight lattice, thereby recovering the internal skein algebra on the left hand side as a consequence of Lemma \ref{lemma: skein passing through point defect}.
\end{proof}

\begin{remark}\label{rmk:parabolics-later}
    Theorem \ref{thm:borel-almost-invertible} and the discussion in Section \ref{sec:composition-junction} give us a model for interfaces between the opposite Borel defects.
We expect this to generalise to pairs of parabolic subgroups sharing a Levi subgroup, but leave this to future publications.
\end{remark}

\section{Connections to quantum cluster algebras}\label{sec:quantum-cluster-algebras}

We use the theory developed above to explore the relationship between internal skein algebras of decorated surfaces and quantum cluster algebras of marked surfaces.
While the full story is out of scope for this article, we recover the results of \cite{IY22} for $G = \SL_3\CC$ and $\Sigma = \PP_3$ and generalise the reduced stated\footnote{Here we rely on the correspondence between internal and stated skeins \cite{Haioun_2022}.} skein algebras appearing in \cite{IY22,Muller_2016}, see also \cite{Ishibashi_Yuasa_2025,Ishibashi_Kano_Yuasa_2025,Ishibashi_Oya_2026} for $\mathfrak{sp}_4$ and development of the general correspondence.

Marked points in the quantum cluster setting correspond to contractible $T$-regions at the boundary with a single gate. In \cite{IY22} it is shown that the $\mathfrak{sl}_3$ reduced stated skein algebra has a quantum cluster structure. We note that this cluster structure is the one obtained from a natural projection on the level of the moduli spaces introduced in \cite{FG06,Goncharov_Shen_2020}.

\subsection{Background on quantum cluster algebras}\label{sec:clusters-background}

We briefly recall the definition of a quantum algebra introduced in \cite{BZ05} and broadly follow the presentation from \cite{SS25}, while focusing on the A-variables. Recall that decorated character varieties have associated cluster X-varieties and cluster A-varieties.  Any cluster X-variety admits a canonical quantisation in terms of the quivers underlying the cluster construction. On the other hand, an a-priori non-canonical non-degenerate ensemble map is needed to quantise the A-variety.
Goncharov and Shen describe one such map using a natural projection on the level of the moduli spaces of local systems. 
\begin{definition}
    A \emph{\bf seed} is the datum $(I,I_{\ast}, \Lambda, \langle.,.\rangle,\{e_i\}_{i\in I})$ where
    \begin{itemize}
      \item $I$ is a finite set with a distinguished subset $I_{\ast} \subset I$;
      \item $\Lambda$ is a lattice with a skew-symmetric bilinear form $\langle \cdot,\cdot \rangle : \Lambda \times \Lambda \to \frac{1}{2} \ZZ$;
      \item $\{e_i\}_{i \in I}$ is a basis of $\Lambda$, with $\epsilon_{i,j} := \langle e_i,e_j\rangle \in \ZZ$ for $i,j \in I \backslash I_{\ast}$.
    \end{itemize}
    We will sometimes emphasize the lattice by writing $\langle\cdot,\cdot\rangle_\Lambda$.
    The \emph{\bf quiver associated to a seed}, $Q$, has vertex set $V(Q) = I$ and adjacency matrix $\epsilon_{ij} = \langle e_i,e_j\rangle$. When $\langle e_i,e_j\rangle$ is half-integer the corresponding arrow is instead weighted by $\langle e_i,e_j\rangle$.
 One can likewise associate a seed to a appropraite quiver. The sets $I_\ast$ and $I\backslash I_\ast$ are referred to as the frozen and mutable vertices, respectively.
 \end{definition}

\begin{definition}
    Given a seed $ \Theta = (I,I_{\ast}, \Lambda, \langle.,.\rangle,\{e_i\}_{i\in I})$, an extension of $\Theta$ to a \emph{\bf compatible pair} is the data of a lattice $\Xi$ such that 
    \begin{equation}
        \Lambda \subseteq \Xi\subset \Lambda \otimes_{\ZZ} \mathbb{Q},
    \end{equation}
    and a basis $\{\xi_i\}_{i\in I}$ for $\Xi$ such that 
    \begin{equation}
        \langle e_i, \xi_j\rangle = \delta_{ij} \text{ for all } i \in I\backslash I_\ast \text{ and } j \in I.
    \end{equation}
    Here the skew-form $\langle.,.\rangle$ on $\Xi$ is obtained from $\langle.,.\rangle_{\Lambda}$ by extension of scalars.
\end{definition}
A compatible pair defines an integer \emph{ensemble matrix} $b$ by writing the elements $\{e_i\}_{i\in I}$ in terms of the basis $\{\xi_i\}_{i\in I}$. We write
\begin{equation}
    e_j = \sum_{i\in I} b_{ij}\xi_i, \qquad \text{  for all $j \in I$}.
\end{equation}
Note that for any integer matrix $b$ that is invertible over $\mathbb{Q}$, one can define a compatible pair by defining the lattice $\Xi$ as spanned by the elements
\begin{align}\label{eq: xi in terms of e}
    \xi_j =\sum_{i\in I} (b^{-1})_{ij}e_i, \qquad \text{  for all $j \in I$}.
\end{align}
This fact was already implicit in \cite{SS25}.
Additionally, for a given compatible pair, one obtains a $\mathbb{Q}$-valued skew-symmetric matrix $\kappa$, whose entries are given by 
\begin{equation}
    \kappa_{ij} := (\xi_i,\xi_j),\qquad \text{  for all $i,j \in I$}.
\end{equation}
Plugging in \eqref{eq: xi in terms of e}, we obtain and expression of $\kappa$ in terms of the adjacency and ensemble matrices: $\kappa = (b^{-1})^T\epsilon b^{-1}$.
Given a seed with corresponding quiver $Q$ and compatible pair $(\Xi,\{\xi_i\}_{i\in I})$ we define the quantum torus
\begin{equation}
    \mathcal{T}_Q^\mathcal{A} = \mathcal{K}\langle Y_{\pm\xi_i} | i \in I\rangle/\langle q^{\kappa_{ij}}Y_{\xi_i}Y_{\xi_j} = q^{\kappa_{ji}}Y_{\xi_j}Y_{\xi_i}  \rangle.
\end{equation}
 Here, $\mathcal{K}$ denotes the ring one chooses to work over. A common choice is $\mathcal{K} = \ZZ[q^{\pm1/d}]$, where $d$ denotes an integer such that $d\kappa_{ij} \in \ZZ$ for all $i,j \in I$. We will always choose $d$ to be minimal. We can equivalently define $\mathcal{T}_{Q}^{\mathcal{A}}$ as the $\mathcal{K}$-algebra generated by the elements $Y_{\lambda}$ for $\lambda \in \Xi$ and multiplication given by 
\begin{equation}
    q^{\langle\lambda,\mu\rangle}Y_\lambda Y_{\mu}= Y_{\lambda+\mu}.
\end{equation}
\begin{definition}
    The elements $Y_{\xi_i} \in \mathcal{T}_Q^{\mathcal{A}}$ are called the quantum cluster $\mathcal{A}$-variables associated to the compatible pair $(\Xi,\{\xi_i\})$.
\end{definition}
Consider a seed $\Theta$, with quiver $Q$ and a compatible pair $(\Xi, \{\xi_i\})$. For an element $k \in I\backslash I_{\ast}$, one introduces a new seed $\Theta'$ with quiver denoted as $\mu_k(Q)$ and compatible pair $(\Xi', \{\xi_i'\})$. They have the same underlying lattices and skew form and the new bases are given by
\begin{equation}
  e_i' :=  \begin{cases}
      -e_i  & \text{ if }i = k,\\
      e_i + [\epsilon_{ik}]_{+}e_k  & \text{ if }i \neq k,
    \end{cases} \qquad \text{ and } \qquad  \xi_i' := \begin{cases}
      -\xi_i + \sum_{j\neq k } [b_{jk}]_+\xi_j  & \text{ if }i = k,\\
      \xi_i   & \text{ if }i \neq k,
    \end{cases},
\end{equation}
here $[a]_+ := \max(a,0)$. The quiver $\mu_k(Q)$ is called the mutation of $Q$ in the direction $k$. Consequently, the new ensemble matrix is given by 
\begin{align}\label{eq: mutated b}
    b'_{ij} = \begin{cases}
       -b_{ij} & \text{ if } k=i \text{ or } k=j,\\
       b_{ij} + \frac{|b_{ik}|b_{kj}+ b_{ik}|b_{kj}|}{2} & \text{ otherwise.}
    \end{cases}
\end{align}
One associates to a quiver mutation an isomorphism on the level of the quantum tori,
\begin{equation}
\begin{aligned}
    \mu_k^q: \Frac(\mathcal{T}_{\mu_k(Q)}^\mathcal{A})  &\to \Frac(\mathcal{T}_{Q}^\mathcal{A})\\
    Y_{\xi_k'} &\mapsto Y_{-\xi_k + \xi_k^+} +  Y_{-\xi_k + \xi_k^-}\\
    Y_{\xi_i'} &\mapsto Y_{\xi_i} \hspace{29mm} \text{ for } i \neq k,
\end{aligned}
\end{equation}
    where
    \begin{equation}
        \xi_k^{\pm} = \sum_{j\in I}[\pm b_{jk}]_{+}\xi_j.
    \end{equation}
    By the quantum Laurent phenomenon \cite{BZ05}, any quantum $\mathcal{A}$-variable defined through finitely many mutations, remains a Laurent polynomial in the original quantum torus. Hence, we can state the following definition.
\begin{definition}
    The quantum cluster algebra $\mathscr{A}_{Q,b}^q$ is given by the subalgebra of $\mathcal{T}_Q^{\mathcal{A}}$ generated by all quantum $\mathcal{A}$-variables. We assume the frozen variables are invertible.
\end{definition}

\subsection[Computations for SL3]{Computations for $\SL_3$}\label{sec:computations-SL3-PGL3}

Our goal is to compare the quantum cluster structures on the moduli spaces $\mathcal{A}_{\SL_3,\Sigma}$ and $\mathcal{P}_{\PGL_3,\Sigma}$ introduced by Fock-Goncharov \cite{FG06} and Goncharov-Shen \cite{Goncharov_Shen_2020} to the internal skein algebras of appropriate decorated surfaces.
We focus on $\Sigma$ being the triangle with three marked points, so that the decorated surface is $\PP_3$, the disk with three contractible $T$-regions on its boundary, all separated by parabolic defects $\widetilde{\D}_B$ from a single $G$ region. In this section $\mathcal{K} = k(q^{1/6})$ for a field  $k$ of characteristic zero and $q^{1/6}$ generic. 

The initial cluster seed is given by the quiver in Figure \ref{fig:SL3-quiver}. Square vertices indicate frozen variables and circular vertices correspond to mutable ones. The dashed arrows on the edges of the triangles are weighted by $1/2$.
\begin{figure}
    \centering
    \includegraphics{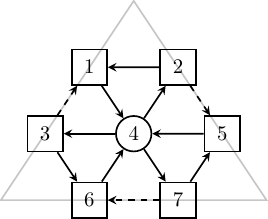}
    \caption{Quiver for $\SL_3$}
    \label{fig:SL3-quiver}
\end{figure}

Let $i(l)$ denote the index of a frozen variable $l$, $C$ denote the Cartan matrix of $\SL_3\CC$, and $\hat{C}$ be the $I\times I$ matrix with coefficients
\begin{equation}
\hat{C}_{l_1l_2} = 
    \begin{cases}
    \frac{1}{2} C_{i(l_1)i(l_2)} & \text{if } l_1,l_2 \in I_\ast \text{ lie on the same edge,}\\
    0 &  \text{else.}
\end{cases}
\end{equation}
It is shown in \cite[Lemma 4.16]{SS25} that the sum of $\hat{C}$ and the adjacency matrix defines an ensemble matrix,
\begin{equation}
    b = \epsilon + \hat{C},
\end{equation}
and hence a compatible pair for the initial cluster seed.
This matrix is obtained from expressing a natural projection map between the moduli spaces $\mathcal{A}_{G',\Sigma} \to \mathcal{P}_{G,\Sigma}$, in terms of the initial cluster variables, as described in \cite[Section 13]{Goncharov_Shen_2020}.

The quiver in Figure \ref{fig:SL3-quiver} has only one mutable vertex and therefore only one cluster chart $\mathcal{T}_{\mu_4(Q)}^\mathcal{A}$ beyond the initial one.
This second chart has generators $Y_{\xi_i'}$ whose q-commutation relations can be expressed in terms of $\kappa'$, which in turn can be computed from the mutated ensemble matrix $b'$ described in \eqref{eq: mutated b}. The mutated variables expressed in $\mathcal{T}_Q^{\mathcal{A}}$ are given by 
\begin{equation}
\begin{aligned}
    Y_{\xi_4}'&:= \mu_k^q(Y_{\xi_4'}) = Y_{-\xi_4 + \xi_1 +\xi_5 + \xi_6} + Y_{-\xi_4 + \xi_2 +\xi_3 + \xi_7} \in \mathcal{T}_Q^{\mathcal{A}},\\
    Y_{\xi_i}' &:= \mu_k^q(Y_{\xi_i'}) = Y_{\xi_i} \text{ for } i \in I_\ast .
\end{aligned}
\end{equation}
Finally, the quantum cluster algebra associated to the triangle for $\SL_3$ is given by
\begin{equation}\label{eq: quantum cluster algebra}
     \mathscr{A}^q_{\SL_3, \Delta} = \mathcal{K}\langle Y_{\pm\xi_i}, Y_{\xi_4}, Y_{\xi_4}' | i \in I_\ast\rangle/\sim
\end{equation}
with relations given by
\begin{equation}\label{eq: quantum cluster relations}
\left\{
\begin{aligned}
Y_{\xi_i}Y_{-\xi_i}
&= Y_{-\xi_i}Y_{\xi_i} =1\text{ for } i \in I_\ast,\\
q^{\kappa_{ij}}Y_{\xi_i}Y_{\xi_j}
&= q^{\kappa_{ji}}Y_{\xi_j}Y_{\xi_i} \text{ for } i,j \in I,\\
q^{\kappa'_{i4}}Y_{\xi_i}Y_{\xi_4}'
&= q^{\kappa'_{4i}}Y_{\xi_4}'Y_{\xi_i}\text{ for } i \in I_\ast,\\
Y_{\xi_4}Y_{\xi_4}'
&= q^{1/3}Y_{\xi_1}Y_{\xi_5}Y_{\xi_6}
 + q^{-1/3}Y_{\xi_2}Y_{\xi_3}Y_{\xi_7},\\
 Y'_{\xi_4}Y_{\xi_4} &= q^{-2/3}Y_{\xi_1}Y_{\xi_5}Y_{\xi_6}
 + q^{2/3}Y_{\xi_2}Y_{\xi_3}Y_{\xi_7}.
\end{aligned}
\right\}.
\end{equation}

Denote by $\PP_\triangle$ the decorated disk where each pair of parabolic defects is replaced by the Weyl defect as in Theorem \ref{thm:borel-almost-invertible}, see Figure \ref{fig:PP-triangle}.
We have the following algebra isomorphism, recovering the result from \cite{IY22}.
\begin{theorem}
    For $\PP_\triangle$ with a single gate in every $T$-region we have
    \begin{equation}
        \SkAlg^{int}_{\G}(\PP_\triangle) \cong \mathscr{A}^q_{\SL_3, \Delta}.
    \end{equation}
\end{theorem}

\begin{figure}
    \centering
    \includegraphics[width=0.5\linewidth]{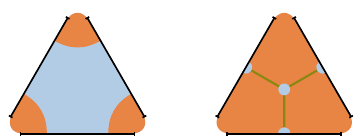}
    \caption{Left: $\PP_3$ with a gate in each $T$-region. Right: $\PP_\triangle$.}
    \label{fig:PP-triangle}
\end{figure}

\begin{proof}
    Denote by $P_q$ the subalgebra of $\mathscr{A}^q_{\SL_3, \Delta}$ generated by $Y_{\xi_i}$ for $i \in I$ and $Y_{\xi_4}'$. In other words, it is the quantum cluster algebra obtained without inverting the frozen variables.
    We will describe an explicit isomorphism $P_q\cong \SkAlg^{int}_{\G}(\PP_3)$.
    We know that $\SkAlg^{int}_{\G}(\PP_3) \cong (O_q(G/N)^{\otimes 3})^G$, and that the Peter-Weyl theorem implies
    \begin{equation}\label{eq:Peter-Weyl-for-OqGN3}
        (O_q(G/N)^{\otimes 3})^G \cong  \bigoplus_{\lambda, \mu,\nu \in \Lambda^+}(V^*_{\lambda}\otimes V_{\mu}^* \otimes V_{\nu}^*)^G \boxtimes \CC_{\lambda} \otimes \CC_{\mu} \otimes \CC_{\nu}.
    \end{equation} 
    This is a graded algebra with the triple $(\lambda, \mu,\nu) \in (\Lambda^+_{\SL_3})^3$ denoting the degree. We now specify elements $S_{ijk}$ with degrees given by fundamental weights $(\varpi_i,\varpi_j,\varpi_k)$ indexed by triples $(i,j,k)$ with $i,j,k<3$ and either $i+j+k =3$ or $i=j=k=2$.
    
    In terms of \eqref{eq:Peter-Weyl-for-OqGN3} the $S_{ijk}$ are as follows.
    Pick a highest weight vector $v_{\varpi_1} \in V_{\varpi_1}$ and its dual $v_{\varpi_1}^*$, normalised such that $\langle v_{\varpi_1}^*, v_{\varpi_1}\rangle=1$. Let $S_{ij0}$ denote to the element $A_{\varpi_j}$ as defined in \eqref{eq:digon-coordinate}, i.e. in the first two copies of $O_q(G/N)$. Analogously, $S_{0jk}$ and $S_{i0k}$ are equal to $A_{\varpi_k}$ in their corresponding pairs of $O_q(G/N)$ inside $(O_q(G/N)^{\otimes 3})^G$. After identification $V_{\varpi_2}\cong V_{\varpi_1}^*$, $A_{\varpi_1}$ and $A_{\varpi_2}$ are defined in terms of the highest weight vectors $v_{\varpi_1}$ and $Y^{-1} \rhd v_{\varpi_1}^*$, respectively. We define $S_{111}$ and $S_{222}$ as 
    \begin{equation}
        \begin{aligned}
    S_{111} &= \sum_{\sigma \in S_3}(-q)^{-l(\sigma)}e^{\sigma(1)} \otimes e^{\sigma(2)}\otimes e^{\sigma(3)} \boxtimes    v_{\varpi_1} \otimes  v_{\varpi_1} \otimes  v_{\varpi_1},
    \\
    S_{222} &= \sum_{\sigma \in S_3}(-q)^{-l(\sigma)}e_{\sigma(1)} \otimes e_{\sigma(2)}\otimes e_{\sigma(3)} \boxtimes   Y^{-1} \rhd v_{\varpi_1}^* \otimes Y^{-1} \rhd v_{\varpi_1}^* \otimes Y^{-1} \rhd v_{\varpi_1}^*.
\end{aligned}
    \end{equation}
    The corresponding skeins are shown in Figure \ref{fig:P3-variables}.
    In the $G$-region each skein is labelled by the defining representation and in the $T$-regions by the unique compatible 1-dimension $U_q\t$ representations, i.e. $\chi_{\varpi_1}$ or $\chi_{\varpi_2}$ depending on whether the $G$-skein is oriented out of or into the defect, respectively.
    States are given by $v_{\varpi_1} \in \chi_{\varpi_1}$ and $Y^{-1} \rhd v_{\varpi_1}^* \in \chi_{\varpi_2}$. 
\begin{figure}
    \centering
    \includegraphics[width=.8\linewidth]{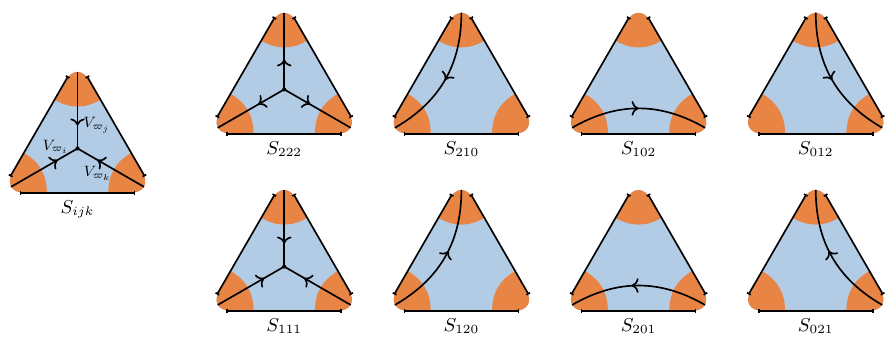}
    \caption{Generators for $\SkAlg^{int}(\PP_3).$}
    \label{fig:P3-variables}
\end{figure}
\begin{figure}
    \centering
    \includegraphics[width=0.7\linewidth]{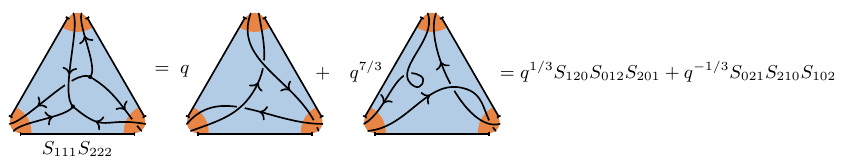}
    \caption{Cluster mutation and skein relations agree.}
    \label{fig:P3-mutation}
\end{figure}

 We can now define algebra morphism $ \varphi : P_q \to \SkAlg^{int}_\mathcal{G}(\PP_3)$, given on generators by
\begin{equation}
\begin{aligned}\label{eq: iso between cluster and skein}
       Y_{\xi_1} \mapsto S_{120}, \qquad Y_{\xi_2} \mapsto S_{021}, \qquad Y_{\xi_3} \mapsto S_{210}, \qquad Y_{\xi_4} \mapsto S_{111}, \\
    Y_{\xi_5} \mapsto S_{012}, \qquad Y_{\xi_6} \mapsto S_{201}, \qquad Y_{\xi_7} \mapsto S_{102}, \qquad Y_{\xi_4}' \mapsto S_{222}.
\end{aligned}
   \end{equation}  

Next, we check directly that \eqref{eq: iso between cluster and skein} respects the product and is therefore a morphism of algebras.
Assume without loss of generality that the surface boundary is oriented counter clockwise.
Recall that in the internal skein algebra a product $XY$ means $X$ is placed above $Y$ and enters each gate later according to the boundary orientation. This allows us to recover the relations from the quantum cluster algebra \eqref{eq: quantum cluster relations},
\begin{equation}\label{eq:P3-q-commuting-relations}
    \begin{aligned}
    S_{ij0}S_{ab0} &= S_{ab0} S_{ij0},\\
    S_{ij0}S_{0bc} &= q^{(\varpi_b,\varpi_j)}S_{0bc} S_{ij0},\\
    S_{ij0}S_{aaa} &= q^{(\varpi_j-\varpi_i,\varpi_a)}S_{aaa}S_{ij0},\\
    S_{111}S_{222}&= q^{1/3}S_{120}S_{012}S_{201} + q^{-1/3}S_{021}S_{210}S_{102},\\
    S_{222}S_{111}&= q^{-2/3}S_{120}S_{012}S_{201} + q^{2/3}S_{021}S_{210}S_{102}.
\end{aligned}
\end{equation}
Note that these relations imply the other omitted ones, because of the symmetries of the triangle.
The first three equations in \eqref{eq:P3-q-commuting-relations} are obtained by isotoping braidings to the $T-$regions and applying the coend relations at the gate to either absorb or emit a braiding at the cost of a power of $q$. The last two equation follows from the $\sl_3$ skein relations.
Recall from \cite{Sikora_2005} that for $\sl_n$, we have
\begin{equation} \label{eq: SL3 skein relations}
  \loopstrandJ{strand} = q^{1/n-n} \singlestrand{strand}  \quad \text{ and } \quad
  \NtoNvertexSL{n} = \quad (-q)^{\binom{n}{2}} \sum_{\sigma \in S_n} (-q^{1/n-1})^{\ell(\sigma)} \, \Nbox{\Tilde{\sigma}}, \notag
\end{equation}
where each skein is labelled by the defining representation and $\Tilde{\sigma}$ denotes the minimum positive crossing braid representing $\sigma$.
Applying this to the product $S_{111}S_{222}$ we arrive at the $\SL_3\CC$ quantum cluster mutation, see Figure \ref{fig:P3-mutation}.
Here we've used that any skein crossing a Borel defect twice in $\DD_1$ vanishes \cite[Lemma 3.1]{Brown_Jordan_2025}. The expression for $S_{222}S_{111}$ is obtained in a similar fashion.

It now remains to show that the map $\varphi: P_q \to \SkAlg^{int}_\mathcal{G}(\PP_3)$ is an isomorphism. We do so by appealing to the classical limit.

First note that $\varphi$ induces a $(\Lambda^+)^3$-grading on $P_q$, so that the degree of $Y_{\xi_i}$ is given by the degree of $\varphi(Y_{\xi_i})$.
This grading is well-defined since the defining relations \eqref{eq: quantum cluster relations} are homogeneous. We now show that the Hilbert series of the two algebras are equal.

Because $\Rep_q(G)$ and $\Rep(G)$ are equivalent as monoidal categories we have that
\begin{equation}
    {H(O_q(G/N)^{\otimes 3})^G) = H((O(G/N)^{\otimes 3})^G)}.
\end{equation}
Here we've used $H$ to denote the Hilbert series.

Classically, the space $((\SL_3/N)^3)^{\SL_3}$ can be identified with the configuration space of three decorated flags, i.e. the $\SL_3$-invariant subspace of $\{(v_i,\phi_i)_{i=1,2,3} \in V \times V^*| \phi_i(v_i) = 0\}$, with $V$ a $3$-dimensional complex vector space. The first and second fundamental theorems of invariant theory \cite{weyl1939classical} state that the algebra of functions is generated by $\phi_i(v_j), \det(v_1,v_2,v_3)$, and $\det(\phi_1,\phi_2,\phi_3)$ and that moreover the relations are such that $(O(G/N)^3)^G \cong P_{q=1}$, the algebra over $k$ with the same presentation as $P_q$ with relations evaluated at $q=1$ . Hence, it remains to show that $H(P_q) = H(P_{q=1})$.

The algebra $P_q$ has a $PBW$ type basis of elements of the form $Y_{\xi_1}^{a_1}Y_{\xi_2}^{a_2}\ldots Y_{\xi_7}^{a_7}(Y'_{\xi_4})^b$ where $a_4b =0$. Indeed, the defining relations allow us to rewrite any elements as a linear combination of the proposed basis elements. We have that 
\begin{equation}
    [Y_{\xi_i}, Y_{\xi_4}Y_{\xi_4}'] = [Y_{\xi_i},Y_{\xi_4}'Y_{\xi_4}] =[Y_{\xi_i},Y_{\xi_1}Y_{\xi_5}Y_{\xi_6}] =[Y_{\xi_i},Y_{\xi_2}Y_{\xi_3}Y_{\xi_7}] = 0 \qquad \text{ for } i \in I_*
\end{equation}
and
\begin{equation}
\begin{aligned}
    Y_{\xi_4}Y_{\xi_1}Y_{\xi_5}Y_{\xi_6} = qY_{\xi_1}Y_{\xi_5}Y_{\xi_6}Y_{\xi_4},  \qquad Y_{\xi_4}Y_{\xi_2}Y_{\xi_3}Y_{\xi_7}= q^{-1}Y_{\xi_2}Y_{\xi_3}Y_{\xi_7}Y_{\xi_4}\\
    Y_{\xi_4}'Y_{\xi_1}Y_{\xi_5}Y_{\xi_6} = q^{-1}Y_{\xi_1}Y_{\xi_5}Y_{\xi_6}Y_{\xi_4}',  \qquad Y_{\xi_4}'Y_{\xi_2}Y_{\xi_3}Y_{\xi_7}= qY_{\xi_2}Y_{\xi_3}Y_{\xi_7}Y_{\xi_4}'.
\end{aligned}    
\end{equation}

These equations imply there are no ambiguities when performing reductions, and hence the reduction process is well-defined. For the algebra $P_{q=1}$, we can perform the same reduction process and thus obtain a clear bijection between the two graded bases. Therefore, we conclude that the Hilbert series of $P_q$ and $P_{q=1}$ match.

We conclude by proving that $\varphi: P_q \to \SkAlg^{int}_\mathcal{G}(\PP_3)$ is surjective. 
Any skein in $\SkAlg^{int}_\mathcal{G}(\PP_3)$ can be written as a sum of skeins with $T$-weights $(\lambda,\mu,\nu) \in (\Lambda^+_{\SL_{3}})^3$ with trivalent coupon labelled by some morphism  $f:V_\lambda \otimes V_{\mu} \otimes V_{\nu} \to \1$ in the $G$-region.
Any irreducible representation in $\Rep_q(\SL_3)$ appears as a subrepresentation of tensor powers of the defining representation and its dual, i.e. we always have a monomorphism $i_\lambda:V_{\lambda =a\varpi_1 + b\varpi_2} \hookrightarrow V_{\varpi_1}^{\otimes a} \otimes V_{\varpi_2}^{\otimes b}$.
The morphism $f$ can be expressed as $(i_\lambda \otimes i_\mu \otimes i_\nu) \circ f'$ where $f'$ corresponds to a collection of skeins labelled by the defining representation, with crossings, cups, caps and trivalent sources and sinks. The map $i_\lambda$ sends the highest weight vector $v_\lambda$ to $v_{\varpi_1}^{\otimes a} \otimes v_{\varpi_2}^{\otimes b}$.
This fact allows the coupon labelled by $i_\lambda$ to pass through the Borel defect to the $T$-region, where it can be absorbed into a gate using the boundary relations of internal skein algebras.

Thus, we have reduced our initial skein to a sum of products of skeins all labelled by the defining representation in the $G$-region.
Possibly after a number of isotopies, we can use the second skein relation in equation \eqref{eq: SL3 skein relations} to further reduce this collection of skeins to the case where the trivalent vertices are all sinks or all sources. If there are no trivalent vertices remaining, the resulting element is, again up to skein relations, a product of the skeins $S_{ijk}$ where exactly of the indices is zero. Suppose there are only sinks remaining, they must be equal to $S_{111}$, up to a scalar. Indeed, as there are only sinks each incident edge must end at one of the gates. However, they must each end at a different gate. Indeed, a skein with two or three edges ending at the same gate would equal the zero skein after evaluation at the Borel defect. This follows from the fact that the defining representation has no elements of weight $2\varpi_1$ and the trivial representation has no elements of weight $3\varpi_1$.
Analogously, one argues that any source, without sinks present, must be equal to $S_{222}$. Thus, the original skein can be written as a product of the elements $S_{ijk}$ and the map $\varphi$ is surjective.

We have established a surjective graded algebra map $\varphi: P_q \cong \SkAlg^{int}_\mathcal{G}(\PP_3)$ between two graded algebras, for which the Hilbert series coincide. Thus, for each graded piece there is a surjective map between two vector spaces of the same dimension, hence it is an isomorphism.

Inverting the skeins $S_{ijk}$ where either $i,j$ or $k$ is zero, one obtains the internal skein algebra for $\PP_\triangle$ and the theorem follows.
\end{proof}

\bibliography{main.bib}

\end{document}